\documentclass{article}
\usepackage[a4paper, total={6in, 9in}]{geometry}
\usepackage[utf8]{inputenc}
\usepackage{amsfonts}
\usepackage{graphicx}
\usepackage{placeins}
\usepackage{float}
\usepackage{subfigure}
\usepackage{amsmath}
\usepackage{amsthm}
\newtheorem{theorem}{Theorem}
\newtheorem{lemma}{Lemma}
\newtheorem{corollary}{Corollary}
\newtheorem{proposition}{Proposition}
\newtheorem{definition}{Definition}
\newtheorem{example}{Example}
\newtheorem{remark}{Remark}
\usepackage{amssymb}
\usepackage{adjustbox}
\usepackage{xcolor}
\usepackage{amsthm}
\usepackage{ulem}
\usepackage{tikz-cd}
\usepackage[numbers]{natbib}
\usepackage{hyperref}
\title{Poincaré duality for singular tropical hypersurfaces}
\author{Samuel Dentan}
\date{}

\begin{document}
\maketitle

\begin{abstract}
    We establish a partial extension of the Poincaré duality theorem of Jell–Rau–Shaw to tropical hypersurfaces arising from non-primitive triangulations. We introduce a notion of level of primitivity for triangulations of lattice polytopes and show that tropical hypersurfaces satisfy a partial form of Poincaré duality determined by this level. This notion of primitivity is defined modulo a fixed integral domain and is weaker than the classical notion of primitivity.
    
    Moreover, we obtain a generalization of complete Poincaré duality over this integral domain for tropical hypersurfaces whose underlying triangulations are primitive modulo the integral domain. As a corollary, we show that any tropical hypersurface obtained by patchworking from a triangulation of a simple lattice polytope satisfies complete Poincaré duality over the field of rational numbers, providing a converse to a theorem of Aksnes. Throughout, we allow triangulations that are not necessarily convex.
\end{abstract}

\tableofcontents
\newpage

\section{Introduction}

\subsection{Context and Motivations}

Viro’s combinatorial patchworking (\cite{Vi 1984}, \cite{Vi 2006} and \cite{Ri 1993}) is a powerful method to construct real algebraic hypersurfaces in real projective toric varieties with a control on their topology. It is one of the most efficient known methods of construction, thus the study of the topological spaces obtained by combinatorial patchworking has become a field of research in itself.

We obtain a generalization of the work of Renaudineau and Shaw \cite{RS 2023}, which establishes bounds on the Betti numbers of primitive patchworkings, extending it to the almost-primitive case. We will present this work in a next paper which will be pre-printed in 2026. To this end, we find a partial extension of Poincaré duality theorem to non-primitive tropical hypersurfaces. In this paper we present our partial Poincaré duality theorem for non-primitive tropical hypersurfaces. We define also the notion of a $R$-primitive patchworking in a $R$-non-singular lattice polytope for an integral domain $R$, and we present our generalization of the complete Poincaré duality theorem of \cite{JRS 2018} in this framework. In particular, our corollary is that any non-necessarily primitive simplicial patchworking of a simple lattice polytope satisfies a complete Poincaré duality theorem over $\mathbb{Q}$, which is a converse of the main theorem of \cite{Ak 2023}.

Let us start by briefly recalling Viro’s combinatorial patchworking for projective hypersurfaces. Let $n\geq 2, d\geq1$ and $\Delta$ a $n$-simplex of $\mathbb{Z}^n$ of vertices $(0,...,0)$, $(d,0,...,0)$, $...$, $(0,...,0,d)$. Let $\Gamma$ be a triangulation of $\Delta$ with integral vertices, and fix a distribution of signs $\epsilon(v)\in\{-1,1\}$ for each vertex $v$ of $\Delta$. Symmetrize $\Delta$ by each coordinate hyperplane to obtain $2^n-1$ copies and identify the points of the boundary of the union of $\Delta$ and its copies with their symmetric by the origin. We denote $\Delta_*$ this union of $\Delta$ and its copies with these identifications. The space $\Delta_*$ is homeomorphic to $\mathbb{R}P^n$. Symmetrise the triangulation $\Gamma$ of $\Delta$ in each one of its copies. We obtain a triangulation $\Gamma_*$ of $\Delta_*$. Symetrize the sign of each vertex of $\Gamma$ to obtain a sign on each vertex of $\Gamma_*$, using the following rule: $ \epsilon(x_1,...,x_n) = \epsilon(|x_1|,...,|x_n|).(\mathrm{sign} ~ x_1)^{x_1}...(\mathrm{sign} ~ x_n)^{x_n}$. Consider the union $X$ of barycentric cells $\mathrm{Conv}(\mathrm{bar}(\sigma^a),\mathrm{bar}(\sigma^{a+1})$, $...$, $\mathrm{bar}(\sigma^b))$ where $\sigma^a\subseteq \sigma^{a+1}\subseteq ... \subseteq \sigma^b$ are incident simplices of $\Gamma_*$ with $\mathrm{dim}~\sigma^i =i$ and such that every $\sigma^i$ has different signs on its vertices.

Recall that the triangulation of $\Delta$ is said to be convex if there exists a convex piecewise affine function $\nu: \Delta \rightarrow \mathbb{R}$, affine on each $n$-simplex of the triangulation and non-affine on each union of two distinct $n$-simplices. Then the combinatorial Viro's patchworking theorem is:

\begin{theorem}
    (Combinatorial Viro's patchworking theorem) If the triangulation $\Gamma$ is convex there exists a real algebraic hypersurface of degree $d$ of $\mathbb{R}P^n$ satisfying the following homeorphism of pair:
    \begin{equation*}
        (\mathbb{R}P^n,\mathbb{R}X) \simeq (\Delta_*,X).
    \end{equation*}
\end{theorem}

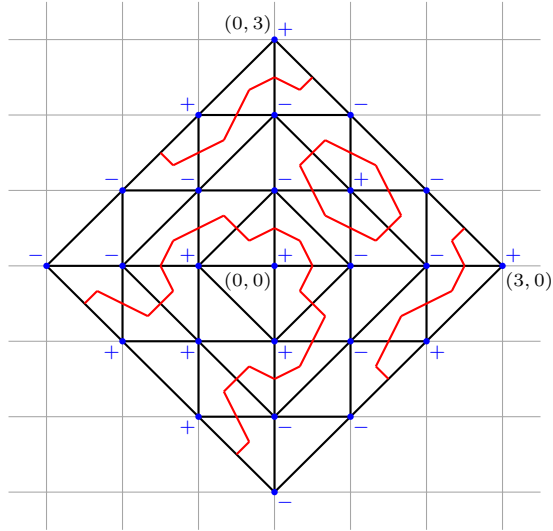
\begin{figure}[h]
\centering

\begin{tikzpicture}[scale=1]
  \draw[step=1, gray!70] (-3.5,-3.5) grid (3.5,3.5);

  \coordinate (A) at (3,0);
  \coordinate (B) at (0,3);
  \coordinate (C) at (-3,0);
  \coordinate (D) at (0,-3);
  \coordinate (A1) at (2,0);
  \coordinate (B1) at (0,2);
  \coordinate (C1) at (-2,0);
  \coordinate (D1) at (0,-2);
  \coordinate (A2) at (1,0);
  \coordinate (B2) at (0,1);
  \coordinate (C2) at (-1,0);
  \coordinate (D2) at (0,-1);
  \coordinate (E) at (-2,1);
  \coordinate (F) at (-2,-1);
  \coordinate (E1) at (2,1);
  \coordinate (F1) at (2,-1);
  \coordinate (E2) at (-1,2);
  \coordinate (F2) at (-1,-2);
  \coordinate (E3) at (1,2);
  \coordinate (F3) at (1,-2);

  \draw[thick, black] (A) -- (B) -- (C) -- (D) -- cycle;
  \draw[thick, black] (A1) -- (B1) -- (C1) -- (D1) -- cycle;
  \draw[thick, black] (A2) -- (B2) -- (C2) -- (D2) -- cycle;
  \draw[thick, black] (A) -- (C) -- cycle;
  \draw[thick, black] (B) -- (D) -- cycle;
  \draw[thick, black] (E) -- (F) -- (F1) -- (E1) -- cycle;
  \draw[thick, black] (E2) -- (F2) -- (F3) -- (E3) -- cycle;

  \draw[thick, red] (-0.5,-2.5) -- (-1/3,-7/3) -- cycle;
  \draw[thick, red] (-2/3,-5/3) -- (-1/3,-7/3) -- cycle;
  \draw[thick, red] (-2/3,-5/3) -- (-1/3,-4/3) -- cycle;
  \draw[thick, red] (0,-1.5) -- (-1/3,-4/3) -- cycle;
  \draw[thick, red] (1/3,-4/3) -- (0,-1.5) -- cycle;
  \draw[thick, red] (1/3,-4/3) -- (2/3,-2/3) -- cycle;
  \draw[thick, red] (1/3,-1/3) -- (2/3,-2/3) -- cycle;
  \draw[thick, red] (1/3,-1/3) -- (1/2,0) -- cycle;
  \draw[thick, red] (1/2,0) -- (1/3,1/3) -- cycle;
  \draw[thick, red] (1.5,-1.5) -- (4/3,-4/3) -- cycle;
  \draw[thick, red] (5/3,-2/3) -- (4/3,-4/3) -- cycle;
  \draw[thick, red] (5/3,-2/3) -- (7/3,-1/3) -- cycle;
  \draw[thick, red] (2.5,0) -- (7/3,-1/3) -- cycle; 
  \draw[thick, red] (2.5,0) -- (7/3,1/3) -- cycle;
  \draw[thick, red] (2.5,0.5) -- (7/3,1/3) -- cycle;
  \draw[thick, red] (2/3,2/3) -- (4/3,1/3) -- cycle;
  \draw[thick, red] (5/3,2/3) -- (4/3,1/3) -- cycle;
  \draw[thick, red] (5/3,2/3) -- (4/3,4/3) -- cycle;

  \draw[thick, red] (-2.5,-0.5) -- (-7/3,-1/3) -- cycle;
  \draw[thick, red] (-5/3,-2/3) -- (-7/3,-1/3) -- cycle;
  \draw[thick, red] (-5/3,-2/3) -- (-4/3,-1/3) -- cycle;
  \draw[thick, red] (-1.5,0) -- (-4/3,-1/3) -- cycle;
  \draw[thick, red] (-4/3,1/3) -- (-1.5,0) -- cycle;
  \draw[thick, red] (-4/3,1/3) -- (-2/3,2/3) -- cycle;
  \draw[thick, red] (-1/3,1/3) -- (-2/3,2/3) -- cycle;
  \draw[thick, red] (-1/3,1/3) -- (0,1/2) -- cycle;
  \draw[thick, red] (0,1/2) -- (1/3,1/3) -- cycle;
  \draw[thick, red] (-1.5,1.5) -- (-4/3,4/3) -- cycle;
  \draw[thick, red] (-2/3,5/3) -- (-4/3,4/3) -- cycle;
  \draw[thick, red] (-2/3,5/3) -- (-1/3,7/3) -- cycle;
  \draw[thick, red] (0,2.5) -- (-1/3,7/3) -- cycle;
  \draw[thick, red] (0,2.5) -- (1/3,7/3) -- cycle;
  \draw[thick, red] (0.5,2.5) -- (1/3,7/3) -- cycle;
  \draw[thick, red] (2/3,2/3) -- (1/3,4/3) -- cycle;
  \draw[thick, red] (2/3,5/3) -- (1/3,4/3) -- cycle;
  \draw[thick, red] (2/3,5/3) -- (4/3,4/3) -- cycle;

  \fill[blue] (0,-3) circle (1.2pt) node[font=\scriptsize, xshift=4,yshift=-4] {$-$};
  \fill[blue] (-1,-2) circle (1.2pt) node[font=\scriptsize, xshift=-4,yshift=-4] {$+$};
  \fill[blue] (0,-2) circle (1.2pt) node[font=\scriptsize, xshift=4,yshift=-4] {$-$};
  \fill[blue] (1,-2) circle (1.2pt) node[font=\scriptsize, xshift=4,yshift=-4] {$-$};
  \fill[blue] (-2,-1) circle (1.2pt) node[font=\scriptsize, xshift=-4,yshift=-4] {$+$};
  \fill[blue] (-1,-1) circle (1.2pt) node[font=\scriptsize, xshift=-4,yshift=-4] {$+$};
  \fill[blue] (0,-1) circle (1.2pt) node[font=\scriptsize, xshift=4,yshift=-4] {$+$};
  \fill[blue] (1,-1) circle (1.2pt) node[font=\scriptsize, xshift=4,yshift=-4] {$-$};
  \fill[blue] (2,-1) circle (1.2pt) node[font=\scriptsize, xshift=4,yshift=-4] {$+$};
  \fill[blue] (-3,0) circle (1.2pt) node[font=\scriptsize, xshift=-4,yshift=4] {$-$};
  \fill[blue] (-2,0) circle (1.2pt) node[font=\scriptsize, xshift=-4,yshift=4] {$-$};
  \fill[blue] (-1,0) circle (1.2pt) node[font=\scriptsize, xshift=-4,yshift=4] {$+$};
  \fill[blue] (0,0) circle (1.2pt) node[font=\scriptsize, xshift=4,yshift=4] {$+$};
  \fill[blue] (1,0) circle (1.2pt) node[font=\scriptsize, xshift=4,yshift=4] {$-$};
  \fill[blue] (2,0) circle (1.2pt) node[font=\scriptsize, xshift=4,yshift=4] {$-$};
  \fill[blue] (3,0) circle (1.2pt) node[font=\scriptsize, xshift=4,yshift=4] {$+$};
  \fill[blue] (-2,1) circle (1.2pt) node[font=\scriptsize, xshift=-4,yshift=4] {$-$};
  \fill[blue] (-1,1) circle (1.2pt) node[font=\scriptsize, xshift=-4,yshift=4] {$-$};
  \fill[blue] (0,1) circle (1.2pt) node[font=\scriptsize, xshift=4,yshift=4] {$-$};
  \fill[blue] (1,1) circle (1.2pt) node[font=\scriptsize, xshift=4,yshift=4] {$+$};
  \fill[blue] (2,1) circle (1.2pt) node[font=\scriptsize, xshift=4,yshift=4] {$-$};
  \fill[blue] (-1,2) circle (1.2pt) node[font=\scriptsize, xshift=-4,yshift=4] {$+$};
  \fill[blue] (0,2) circle (1.2pt) node[font=\scriptsize, xshift=4,yshift=4] {$-$};
  \fill[blue] (1,2) circle (1.2pt) node[font=\scriptsize, xshift=4,yshift=4] {$-$};
  \fill[blue] (0,3) circle (1.2pt) node[font=\scriptsize, xshift=4,yshift=4] {$+$};

  \node[color=black,font=\scriptsize] at (-0.35,3.2) {$(0,3)$};
  \node[color=black,font=\scriptsize] at (3.35,-0.2) {$(3,0)$};
  \node[color=black,font=\scriptsize] at (-0.35,-0.2) {$(0,0)$};

\end{tikzpicture}
\caption{A real algebraic curve of $\mathbb{R}P^2$ of degree $3$ obtained by Viro's patchworking.}
\label{Patchwork}
\end{figure}

\begin{example}
    Figure \ref{Patchwork} shows an example of algebraic curve of $\mathbb{R}P^2$ of degree $3$ obtained by Viro's patchworking. In this example $\Delta$ is the triangle of vertices $(0,0)$, $(0,3)$ and $(3,0)$. $\Delta_*$ is the square of vertices $(-3,0)$, $(0,3)$, $(0,-3)$ and $(3,0)$ where the line segments $[(3,0);(0,3)]$ and  $[(-3,0);(0,-3)]$ are glued together, as soon as the line segments $[(3,0);(0,-3)]$ and  $[(-3,0);(0,3)]$. A triangulation and a sign distribution are fixed on $\Gamma$ and symmetrized by the rules enounced above to obtain a triangulation and a sign distribution on $\Delta_*$. $X$ is the red curve with the glueing inherited from $\Delta_*$. $X$ is an union of some barycentric cells of the triangulation with the rule enounced above, it separates the signs "$+$" and "$-$". According to Viro's theorem there exists a real algebraic curve $\mathbb{R}X$ of degree $3$ in $\mathbb{R}P^2$ such that the pair $(\mathbb{R}P^2,\mathbb{R}X)$ is homeomorphic to this pair $(\Delta_*,X)$.
\end{example}

This theorem generalizes to hypersurfaces of projective toric varieties $\mathbb{R}Y$ instead of $\mathbb{R}P^n$, to do that we consider a lattice polytope $P$ of $\mathbb{Z}^n$ instead of $\Delta$, with a triangulation $\Gamma$ with integral vertices. For each facet of $P$ there is a copy of $P$, and we consider the union $P_*$ of those copies with some rule of glueing. Similarly, signs on each vertex of the triangulation $\Gamma$ are fixed, which determines signs on each vertex of the triangluation $\Gamma_*$ of $P_*$. with some rule of symmetry. Those signs determine a union $X$ of some barycentric cells of the $\Gamma_*$, then $P_*$ is homeomorphic to the real projective toric variety $\mathbb{R}Y$ associated to the polytope $P$, and Viro's Patchworking theorem enounces that if the triangulation $\Gamma$ of $P$ is convex, then there exists a real algebraic hypersurface $\mathbb{R}X$ of Newton polytope $P$, in the projective toric variety $\mathbb{R}Y$ such that the pair $(\mathbb{R}Y,\mathbb{R}X)$ is homeomorphic to the pair $(P_*,X)$.

The projective toric variety is non-singular if the polytope $P$ is non-singular, that is that the polytope is simple (each vertex of $P$ has exactly $n$ adjacent edges) and that for every vertex of $P$ the primitive integral direction vectors of the adjacent edges span the ambient lattice $\mathbb{Z}^n$. We say that the triangulation $\Gamma$ of the polytope $P$ is primitive if for every $n$-simplex the integral direction vectors of the edges span the ambient lattice $\mathbb{Z}^n$ (it means equivalently that the the volume of every $n$-simplex is $1/n!$). If the polytope $P$ is non-singular and if the triangulation is primitive, the real hypersurface $\mathbb{R}X$ obtained by patchworking is non-singular and Renaudineau and Shaw established the following theorem:

\begin{theorem}
    (\cite{RS 2023}) The Betti numbers $b_q(\mathbb{R}X)$ of a real hypersurface obtained by patchworking of a non-singular lattice polytope of $\mathbb{Z}^n$ and a primitive triangulation are bounded by the Hodge numbers $h^{p,q}(\mathbb{C}X)$ of the complexification in the following way:
    \begin{equation*}
        b_q(\mathbb{R}X) \leq \left \{ \begin{array}{ll}
            h^{q,q}(\mathbb{C}X) & \text{for } q = (n-1)/2 \\
            h^{q,n-q}(\mathbb{C}X) + h^{q,q}(\mathbb{C}X) & \text{otherwise} 
        \end{array} \right. .
    \end{equation*}
    \label{RS}
\end{theorem}

The idea of their proof is to construct a spectral sequence degenerating to the homology of the real hypersurface $\mathbb{R}X$ whose first page corresponds to tropical homology groups $H_q(\mathcal{F}_p^{\mathbb{F}_2})$, which depends on the polytope $P$ and its triangulation $\Gamma$. Then they compute these tropical homology groups using a tropical Lefschetz hyperplan section theorem, a tropical Poincaré duality theorem and some Euler characteristic arguments (\cite{ARS 2021}).

To obtain a real algebraic hypersurface by patchworking from a triangulation of a lattice polytope, it is necessarily to assume that the triangulation is convex. However, if we remove this convexity hypothesis, Viro's Patchorking process continues to exist and we can still construct the topological object $X \subseteq P_*$ which is a piecewise-linear hypersurface. In this case we cannot interpret $X$ as a real hypersurface of the projective variety $\mathbb{R}Y$. However, the Hodge numbers of a generic algebraic hypersurface of Newton polytope $P$ depends only on this polytope $P$, and we denote them $h^{p,q}(P)$. Brugallé, Lopez de Medrano and Rau have proven that the main theorem of \cite{RS 2023} can be generalized in that context:

\begin{theorem}
    (\cite{BMR 2024}) The Betti numbers of the piecewise-linear hypersurface $X$ obtained by non-necessarily convex patchworking of a non-singular lattice polytope P of $\mathbb{Z}^n$ and a primitive triangulation are bounded in the following way:
    \begin{equation*}
        b_q(X) \leq \left \{ \begin{array}{ll}
            h^{q,q}(P) & \text{for } q = (n-1)/2 \\
            h^{q,n-q}(P) + h^{q,q}(P) & \text{otherwise} 
        \end{array} \right. .
    \end{equation*}
    \label{BMRt}
\end{theorem}

Our interest is to generalize these bounds for non-necessarily primitive patchworking. The non-necessarily primitive triangulations of a lattice polytope can be ordered by a level $k$ of primitivity: a triangulation is said to be $k$-primitive if every $k$-simplex is primitive (a $k$-simplex is said to be primitive if it is primitive in its tangent lattice). However, as often in real algebraic geometry, the Betti numbers appearing in the theorem of \cite{RS 2023} and the theorem of \cite{BMR 2024} are the Betti numbers over the field $\mathbb{F}_2$. For this reason we are interested by a homology over $\mathbb{F}_2$, and what really matters is the primitivity modulo $2$, what we call $\mathbb{F}_2-$primitivity. A $k$-simplex of the triangulation $\Gamma$ is said to be $\mathbb{F}_2-$primitive if it is a dilation of a simplex with integral vertices whose volume is of the form $m/k!$ with $m$ an odd number. Here what we call the volume of a $k$-simplex $\sigma$ is its volume in its tangent space $T_{\mathbb{R}}\sigma$ endowed with the lattice $T_{\mathbb{Z}}\sigma:=T_{\mathbb{R}}\sigma\cap\mathbb{Z}^n$. Then we say that the triangulation $\Gamma$ is $(k,\mathbb{F}_2)$-primitive if every $k$-simplex is $\mathbb{F}_2$-primitive. As a $\mathbb{F}_2$-primitive triangulation is defined to be a $(n,\mathbb{F}_2)-$primitive triangulation, we say that a $(n-1,\mathbb{F}_2)$-primitive triangulation is an almost $\mathbb{F}_2$-primitive triangulation (see Section \ref{secprim} and Figures \ref{primtri}, \ref{primsimp} and \ref{triangulation} for more details). Remark that if $n$ is even, every almost $\mathbb{F}_2$-primitive triangulation is a $\mathbb{F}_2$-primitive triangulation, because for any $n$-simplex with $n$ even, if every facet is $\mathbb{F}_2$-primitive, then the $n$-simplex itself is $\mathbb{F}_2$-primitive.

It is also possible to defined the hypothesis of $\mathbb{F}_2$-non-singularity. We say that the lattice polytope $P$ is $\mathbb{F}_2$-non-singular if $P$ is simple and if for every vertex $s$ of $P$ the determinant of the family of primitive integral vectors directing the adjacent edges of $s$ is odd (see Section \ref{secprim} and Figure \ref{fignonsing} for more details). It is weaker than the non-singularity, but it is enough to generalize Theorem \ref{RS} of \cite{RS 2023} and Theorem \ref{BMRt} of \cite{BMR 2024}:

\begin{theorem}
    ([De 2026]) The Betti numbers of a piecewise-linear hypersurface $X$ obtained by non-necessarily convex patchworking of a $\mathbb{F}_2$-non-singular lattice polytope $P$ of $\mathbb{Z}^n$ and an almost $\mathbb{F}_2$-primitive triangulation are bounded by the following expressions:
    \begin{align*}
        b_{\frac{n-1}{2}}(X) \leq & ~~h^{\frac{n-1}{2},\frac{n-1}{2}}(P)+\mathrm{rk}~H_\frac{n-1}{2}(\mathcal{F}^{\mathbb{F}_2}_\frac{n+1}{2})+\mathrm{rk}~H_\frac{n+1}{2}(\mathcal{F}^{\mathbb{F}_2}_\frac{n-1}{2}) +(-1)^\frac{n-1}{2}(\chi(\mathcal{F}^{\mathbb{F}_2}_\frac{n-1}{2})-\chi_{\frac{n-1}{2}}(P)), \\
        b_0(X) \leq & ~~1+ h^{n-1,0}(P)+\mathrm{inf}~\Bigg[~\mathrm{rk}~H_1(\mathcal{F}^{\mathbb{F}_2}_{n-1}) - N_++N_-~;~\mathrm{rk}~H_{n-1}(\mathcal{F}^{\mathbb{F}_2}_1)~\Bigg], \\
        b_q(X) \leq & ~~\mathrm{inf}~ \Bigg [ \mathrm{rk}~H_{n-q}(\mathcal{F}^{\mathbb{F}_2}_q) + \mathrm{rk}~H_{n-1-q}(\mathcal{F}^{\mathbb{F}_2}_{q+1}) ~;\mathrm{rk}~H_{q}(\mathcal{F}^{\mathbb{F}_2}_{n-q}) + \mathrm{rk}~H_{q+1}(\mathcal{F}^{\mathbb{F}_2}_{n-1-q}) \Bigg] \\ & + 1+h^{n-1-q,q}(P) + (-1)^{n-1-q}(\chi(\mathcal{F}^{\mathbb{F}_2}_q)-\chi_q(P))~~~~~~~~\mathrm{for}~q\notin \Bigg\{0,n-1,\frac{n-1}{2}\Bigg\}
    \end{align*}
    where $\chi(\mathcal{F}^{\mathbb{F}_2}_q)$ is the Euler-characteristic of the homology $H_*(\mathcal{F}^{\mathbb{F}_2}_q)$, where $\chi_q(P):=\sum_{r=0}^nh^{q,r}(P)$, and where $N_+$ and $N_-$ are the numbers of non-$\mathbb{F}_2$-primitive $n$-simplices for which the product of the signs on the vertices are $+1$ and $-1$ respectively.
\end{theorem}

This theorem will be the purpose of the article \cite{De 2026} we will pre-print next year. As for \cite{RS 2023}, to prove this theorem we need a tropical Poincaré duality theorem. It was our motivation to find generalizations of the tropical Poincaré duality theorem discovered in \cite{JRS 2018} for non-singular tropical varieties.

 \subsection{Poincaré Duality}

As described in the previous section, we needed to generalize tropical Poincaré duality to non-primitive triangulation, for the field $\mathbb{F}_2$. However a natural question was to find generalization for any integral domain, and not specifically $\mathbb{F}_2$. Some people have already worked on tropical Poincaré duality over the rational field $\mathbb{Q}$. Aksnes \cite{Ak 2023} has proven that a necessarily condition to have tropical Poincaré duality over $\mathbb{Q}$ of a non-singular lattice Polytope is that the subdivision is simplicial. Then, a Poincaré duality theorem over $\mathbb{Q}$ for a non-singular polytope $P$ with a triangulation is a converse of his theorem. Moreover, knowing tropical objects satisfying Poincaré duality over $\mathbb{Q}$ is an interesting problem because there exists some research field about the class of tropical objects satisfying Poincaré duality over $\mathbb{Q}$ (\cite{AP 2023}). As $\mathbb{Q}$ is a field of characteristic zero, every simplex is $\mathbb{Q}$-primitif (the volume of every $n$-simplex with integral vertices in $\mathbb{Z}^n$ is of the form $m/n!$, and $m$ is always invertible in $\mathbb{Q}$ because it is an integer). Then every triangulation is $\mathbb{Q}$-primitive, thus a Poincaré duality theorem for $\mathbb{Q}$-primitive triangulation means a Poincaré duality theorem for any triangulation. That is the reason why we decide to write a more general proof of the partial Poincaré duality theorem, to include both $\mathbb{F}_2$ and $\mathbb{Q}$. We have proven a partial Poincaré duality theorem for any integral domain $R$:

\begin{theorem}
    (Partial Poincaré duality theorem) If $R$ is an integral domain, for a $R$-non-singular lattice polytope $P$ of $\mathbb{Z}^n$ with a $(k,R)$-primitive triangulation, with $2\leq k\leq n-1$, there are canonical isomorphism between some tropical cohomology groups an tropical homology groups:
    \begin{equation*}
        \forall p+q \geq 2n-k,~~~~~~~~~H^q(\mathcal{F}^p_{R}) \xrightarrow{\sim}H_{n-1-q}(\mathcal{F}_{n-1-p}^{R}).
    \end{equation*}
    \label{part}
\end{theorem}

Moreover, if $k=n$ (that is a $R$-primitive triangulation), we have a complete Poincaré duality theorem:

\begin{theorem}
    (Complete Poincaré duality theorem) If $R$ is an integral domain, for a $R$-non-singular lattice polytope $P$ of $\mathbb{Z}^n$ with a $R$-primitive triangulation, there are canonical isomorphisms between some tropical cohomology groups and tropical homology groups:
    \begin{equation*}
        \forall p,q,~~~~~~~~~H^q(\mathcal{F}^p_{R}) \xrightarrow{\sim}H_{n-1-q}(\mathcal{F}_{n-1-p}^{R}).
    \end{equation*}
    \label{comp}
\end{theorem}

Then the corollary is a complete tropical Poincaré duality over $\mathbb{Q}$:

\begin{corollary}
    (Poincaré duality over $\mathbb{Q}$) For a simple polytope $P$ of $\mathbb{Z}^n$, for any triangulation with integral vertices, there are canonical isomorphisms between some tropical cohomology groups and tropical homology groups over $\mathbb{Q}$:
    \begin{equation*}
        \forall p,q,~~~~~~~~~H^q(\mathcal{F}^p_{\mathbb{Q}}) \xrightarrow{\sim}H_{n-1-q}(\mathcal{F}_{n-1-p}^{\mathbb{Q}}).
    \end{equation*}
\end{corollary}

Of course, this corollary remains true if we replace $\mathbb{Q}$ by any field of characteristic zero. The strategy to prove the partial Poincaré duality theorem is based on three ingredients. The first one is to remark that the dual object of the triangulation, on which are defined the tropical cosheaves $\mathcal{F}_p^R$, can be refined into a sub-complex of the cubical subdivision of the triangulation. The cubes of the triangulation are indexed by couples of incident simplices $\sigma\subseteq \tau,~\mathrm{dim} ~\sigma\geq 1$ of the triangulation, their dimension is the difference of dimension between $\sigma$ and $\tau$. Then for this cubical subdivision, the cosheaves complexes can be refined into double complexes $C_{a,b}(\mathcal{F}_p^R)$ defined as being the direct sum of the $\mathcal{F}_p^
R(\sigma,\tau)$ for $\sigma\subseteq\tau$ simplices of the triangulation of dimensions $a$ and $b$ respectively. Those double complexes induce two spectral sequences degenerating on the homology $H_q(\mathcal{F}_p^R)$. We can do the same in cohomology: there are double complexes $C^{a,b}(\mathcal{F}^p_R)$ which induces two spectral sequences degenerating on the cohomology $H^q(\mathcal{F}^p_R)$. The idea is then to construct morphisms from the pages of one of these two cohomological spectral sequence to the pages of one of these two homological spectral sequence. The second ingredient of the proof is to construct a fundamental class for these complexes of cosheaves $\mathcal{F}_p$, to construct a cap-product, and to prove that the cap product with the fundamental class induces morphisms from the $0^{th}$ page of one of the two cohomological spectral sequence to the $0^{th}$ page of one of the two homological spectral sequence. The third ingredient of the proof is then to prove that the cap product with the fundamental class commutes with the differentials of the $0^{th}$ page and induces isomorphisms at the first page of the spectral sequence. In the case of $R$-primitivity, the first page of the cohomological spectral sequence concentrated on a column, and is completely isomorphic to the first page of the homological spectral sequence, concentrated on a line, and the isomorphisms commutes with the differentials, thus at the infinite page we get the Poincaré duality. In the case of $k-R$-primitivity, for $2\leq k\leq n-1$, it remains $k$ isomorphisms from the first page of the cohomological spectral sequence to the first page of the homological spectral sequence. It induces at the infinite page a partial Poincaré duality theorem.

The Partial Poincaré duality theorem and the Poincaré duality theorem can also be generalized to a non-compact framework, considering hypersurfaces of some toric varieties obtained from a polytope $P$ where some faces are removed. Our Partial Poincaré duality theorem can be also ennounced in this non-compact framework, considering Borel-Moore homology and compact-support cohomology. The result of Edvard Aksnes \cite{Ak 2023} on the necessarily condition to satisfy tropical Poincaré duality over $\mathbb{Q}$ is formulated in an non-compact framework. However in that level of generality the proofs are longer and for the convenience of the reader we decide in this paper to establish our results in the compact framework, for hypersurfaces of projective toric varieties. The reader can find this non-compact generalization in my PhD thesis which will appear in 2026.

In the preliminaries of this article we begin with the construction of the cubical subdivision of the triangulation, the construction of the cosheaves $\mathcal{F}_p^R$, their homology, their dual sheaves and cohomology. We define also the double complexes and appropriate spectral sequences filtrating these homology and cohomology. In the second part of the preliminaries we define the notions of $R$-non-singularity for the lattice polytope, the notions of $R$-primitivity and other linked properties for the triangulation. At the end of the preliminaries we construct a fundamental class and a cap-product. In the second part of this article we announce the main lemmas useful to prove that the cap-product with the fundamental class induces isomorphisms between some terms of the first page of the homological spectral sequence and some term of the first page of the cohomological spectral sequence, and we use these lemmas to deduce the Partial Poincaré duality theorem and the complete Poincaré duality theorem. In the third partwe prove the lemmas used in the second part.

\section{Acknowledgment}

I am very grateful to my PhD advisors Arthur Renaudineau and Patrick Popescu-Pampu. I had a lot of prolific discussions with Arthur during my research work and the redaction of this paper, it helped me a lot. I had also very useful discussions with Patrick, which helped me to organize this article and to see possible generalizations of my work. I am also grateful to Jules Chenal for discussions.

\section{Preliminaries}

\subsection{Homology and Cohomology from a lattice polytope and an integral triangulation}

Let $P$ a convex lattice polytope $P$ of dimension $n \geq 2$ living in $\mathbb{R}^n$ endowed with the canonical unimodular lattice $M:= \mathbb{Z}^n$. It defines a projective toric variety $\mathbb{C}Y$. Let consider a triangulation $\Gamma$ of $P$ with integral vertices (non-necessarily convex). By commodity, in the following, when we write an exponent on a simplex of the triangulation $\Gamma$, it will denote its dimension. For instance, a simplex $\sigma^a \in \Gamma$ will design a simplex of dimension $a$ of $\Gamma$.

There exist different subdivisions associated to the triangulation $\Gamma$. In their article \cite{BMR 2024}, Brugallé, Lopez de Medrano and Rau use a subdivision indexed by the poset $\Xi:=\{(F,\sigma),~F \text{ face of}$ $P,~\sigma \in \Gamma \}$. However for the convenience of our proof we use in this article a finer subdivision, named the cubical subdivision (or dihomologic subdivision), which has more symetries than the one indexed by $\Xi$, leading to a better definition of the cap-product. The use of the cubical subdivision of the triangulation to compute tropical homology and cohomology was first introduced by Jules Chenal in his PhD thesis \cite{Ch 2024}. He has proven that the resulting tropical homology and cohomology are the same as for the subdivision indexed by the poset $\Xi$, which justify the possibility of this choice.

The cubical subdivision associated to the triangulation $\Gamma$ is the polyhedral complex whose cells are of the form:
$$c(\sigma^a,\sigma^b):=\mathrm{Conv}\Bigg(\Bigg\{\mathrm{bar}(\tau),~~\tau\in\Gamma,~~\sigma^a\subseteq\tau\subseteq\sigma^b\Bigg\}\Bigg),~~~~~~\sigma^a,\sigma^b\in\Gamma,~\sigma^a\subseteq\sigma^b.$$

In the case of a convex triangulation, the $(n-1)$-skeleton of the cubical subdivision is homeomorphic to the tropical hypersurface associated to the triangulation. Then this $(n-1)$-skeleton is a good generalization of tropical hypersurfaces for non-necessarily convex triangulation. The cubical cells which composed this $(n-1)$-skeleton are indexed by the couples of incident simplices $(\sigma^a,\sigma^b)$ of the triangulation, with the first one of dimension $a\geq1$. Then we introduce the following poset:
\begin{equation*}
    \Omega:=\{(\sigma^a,\sigma^b)\in\Gamma^2,~~\sigma^a\subseteq\sigma^b, ~~a\geq1\}.
\end{equation*}

The cubical cells of the $(n-1)$-skeleton of the cubical subdivision of the triangulation $\Gamma$ are partially ordered by inclusion, and graded by their dimension. It induces on the set $\Omega$ a structure of graded poset with the following partial order and dimension:
\begin{equation*}
    (\sigma_1,\sigma_2) \leq (\tau_1,\tau_2) \iff\sigma_1\subseteq\tau_1 ~\mathrm{and}~\sigma_2\subseteq\tau_2~~;
\end{equation*}
\begin{equation*}
    \mathrm{dim}~(\sigma^a,\sigma^b):=b-a.
\end{equation*}

The diamond property of a graded poset $\Lambda$ is the assumption that for every element $\omega,\gamma\in\Lambda$ with $\omega\leq \gamma$, if the difference of dimension between $\gamma$ and $\omega$ is $2$, then there exist exactly two distinct elements $\epsilon\in\Lambda$ satisfying $\omega <\epsilon< \gamma$. We can denote $\epsilon_1,\epsilon_2$ these two distinct elements, and we say that $\omega<\epsilon_1,\epsilon_2<\gamma$ forms a diamond. The immediately incident couples of a poset $\Lambda$ are the couples $(\alpha,\beta)\in\Lambda^2$ with $\alpha <\beta$ such that there exists no $\epsilon\in\Lambda$ satisfying $\alpha <\epsilon<\beta$. We denote the $\mathcal{I}(\Lambda)$ the set of immediately incident couples. A balancing signature of a graded poset $\Lambda$ satisfying the diamond property is a map $\rho: \mathcal{I}(\Lambda)\rightarrow\{-1,1\}$ such that for every diamond $\omega<\epsilon_1,\epsilon_2<\gamma$ we have:
$$\rho(\omega,\epsilon_1)\rho(\omega,\epsilon_2)\rho(\epsilon_1,\gamma)\rho(\epsilon_2,\gamma) =-1.$$

If we work with $R$-modules with $R$ an integral domain of characteristic different from $2$, the existence of a balancing signature is required to define homological and cohomological theories on the poset, it is necessarily to use it in the definition of the differentials, in order to make the square of the differential vanish. It is well known that any simplicial complex (seen as the poset of its faces ordered by inclusion and graded by their dimension) satisfy the diamond property and admits a balancing signature (constructed considering orientation morphisms from each simplex to its facets). Then the triangulation $\Gamma$ is a graded poset that satisfies the diamond property and admits a balancing signature $\rho:\mathcal{I}(\Gamma)\rightarrow\{-1,1\}$. As a consequence, the poset $\Omega$ also satisfies the diamond property and admits a balancing signature $\eta$ defined by:

\begin{align*}
    \forall (\sigma,\gamma),(\epsilon,\gamma) \in \Omega, ~~\mathrm{if}~(\sigma,\epsilon) \in\mathcal{I}(\Gamma),~~\eta((\epsilon,\gamma),(\sigma,\gamma)):= (-1)^{\mathrm{dim}~(\epsilon,\gamma)}\rho(\sigma,\epsilon)~~;
\end{align*}
\begin{align*}
    \forall (\sigma,\epsilon),(\sigma,\gamma) \in \Omega, ~~\mathrm{if}~(\epsilon,\gamma) \in\mathcal{I}(\Gamma),~~\eta((\sigma,\epsilon),(\sigma,\gamma)):= \rho(\epsilon,\gamma).
\end{align*}

In the following, we fix an integral domain $R$. Let $N:= \mathrm{Hom}(M, \mathbb{Z})$ the dual of the lattice $M=\mathbb{Z}^n$, and let $M_R = M \otimes R$ and $N_R = N \otimes R = \mathrm{Hom}_{\mathbb{Z}}(M, R) = \mathrm{Hom}_{R}(M_R,R)$. For any $\tau \in \Gamma$, we denote $F_{\tau}$ the minimal face of the polytope $P$ containing $\tau$. We can consider the tangent unimodular lattice $T_{\mathbb{Z}}F_{\tau}:=T_{\mathbb{R}}F_{\tau}\cap M$. Its tensorisation by the integral domain $R$ is called the tangent over $R$ of $F_{\tau}$: $T_{R}F_{\tau}:= T_{\mathbb{Z}}F_{\tau} \otimes R \subseteq M_R$. Similarly, for any element $\sigma$ of the triangulation we consider the tangent over $R$ of $\sigma$: $T_{R}\sigma:=T_{\mathbb{Z}}\sigma\otimes R \subseteq M_R$. As $M_R$ has for dual $N_R$, we can also consider the orthogonal tangent spaces over $R$ of $F_{\tau}$ and $\sigma$: $T_{R}F_{\tau}^{\perp}$ and $T_{R}\sigma^{1\perp}$, living in $N_R$. With these notations we can write for any integer $p \geq 0$ and $(\sigma, \tau) \in \Omega$:

\begin{equation*}
    \mathcal{F}^R_p(\sigma,\tau):= \sum_{\sigma^1 \subseteq \sigma} \bigwedge^p \frac{T_R \sigma^{1 \perp}}{T_R F_{\tau}^{\perp}}.
\end{equation*}

In tropical geometry these objects are often called tropical cosheaves, or multi-tangent cosheaves, there were first introduced by \cite{IKMZ 2018} for $R=\mathbb{R}$ and in the case of convex triangulation. Renaudineau and Shaw, and Brugallé, Lopez de Medrano and Rau used them for $R=\mathbb{F}_2$, for a coaser subdivision than the cubical subdivision. Jules Chenal is the first who use them for cubical subdivision. The reason why they are called cosheaves is because there are canonical morphisms of inclusions and projections:

\begin{equation*}
        \forall (\sigma, \tau),(\gamma,\tau) \in \Omega, ~\sigma \subseteq \gamma, ~~ \begin{array}{c} \iota_{\sigma,\tau}^{\gamma,\tau}: \mathcal{F}^R_p(\sigma,\tau) \hookrightarrow \mathcal{F}^R_p(\gamma,\tau) \\ \pi_{\sigma,\tau}^{\sigma,\gamma}: \mathcal{F}^R_p(\sigma,\tau) \twoheadrightarrow \mathcal{F}^R_p(\sigma,\gamma) 
        \end{array}.
\end{equation*}

These morphisms induce by composition some morphisms $\mathcal{F}^R_p(\sigma,\tau) \rightarrow \mathcal{F}^R_p(\gamma,\epsilon)$ for every ordered element $(\gamma,\epsilon)\leq(\sigma,\tau)$, then $\mathcal{F}^R_p$ is a contravariant functor from the poset (seen as a category whose objects are its elements, and whose morphisms are defined by the partial order), to the category of $R$-modules. The name cosheaf (or cellular cosheaf) is often used to design a contravariant functor from a poset to a category of $R$-modules. Then we can take the direct sum over the elements of $\Omega$ with the graduation, it define the following complex:
\begin{equation*}
    C_q(\mathcal{F}^R_p):=\bigoplus_{(\sigma,\tau) \in \Omega,~~\mathrm{dim}~(\sigma,\tau)=q} \mathcal{F}^R_p(\sigma,\tau) = \bigoplus_{b-a=q,~~1\leq a\leq b\leq n}\bigoplus_{\sigma^a\subseteq\sigma^b} \mathcal{F}^R_p(\sigma^a,\sigma^b).
\end{equation*}

The homology of the complex $C_*(\mathcal{F}^R_p)$ is constructed using a differential $d_q$ by the cosheaves morphisms, with respect to the balancing signature $\eta$ of the poset $\Omega$.We denote $H_*(\mathcal{F}_p^R)$ this homology. We remark that a double complex $C_{a,b}(\mathcal{F}^R_p):=\bigoplus_{\sigma^a\subseteq\sigma^b} \mathcal{F}^R_p(\sigma^a,\sigma^b)$ filters the complex $C_*(\mathcal{F}^R_p)$. We remark that the differential $d_q$ induces by restrictions and projections some morphisms on this double complex:
\begin{align*}
    d^1_{a,b}: C_{a,b}(\mathcal{F}^R_p) \rightarrow C_{a+1,b}(\mathcal{F}^R_p), \\
    d^2_{a,b}: C_{a,b}(\mathcal{F}^R_p) \rightarrow C_{a,b-1}(\mathcal{F}^R_p),
\end{align*}
and these morphisms satisfy:
\begin{equation*}
    d_q = \Bigg(\bigoplus_{b-a=q} d^1_{a,b} \Bigg) \oplus \Bigg( \bigoplus_{b-a=q} d^2_{a,b} \Bigg).
\end{equation*}

Then we see that two other morphisms appear: $d_q^1:= \bigoplus_{b-a=q} d^1_{a,b}$ and $d_q^2:= \bigoplus_{b-a=q} d^2_{a,b}$. They satisfy $d_q=d^1_q+d^2_q$, and we can check that $d^1_q$ and $d^2_q$ are both differentials for the complex $C_*(\mathcal{F}^R_p)$. The fact that these differentials are direct sums of the morphisms $d^1_{a,b}$, and the morphisms $d^2_{a,b}$ respectively make them compatible with the structure of double complex $C_{a,b}(\mathcal{F}^R_p)$, then they define two spectral sequences $E^*_{*,*}(d^1)$ and $E^*_{*,*}(d^2)$ whose $0^{th}$ page are $E^0_{*,*}(d^1)=E^0_{*,*}(d^2)=C_{*,*}(\mathcal{F}^R_p)$, whose $0^{th}$ differentials are $\partial^0_{*,*}(d^1):=d^1_{*,*}$ and $\partial^0_{*,*}(d^2):=d^2_{*,*}$ respectively, and whose differentials of the successive pages are induced by the differential $d_*=d^1_*+d^2_*$ of the complex $C_*(\mathcal{F}^R_p)$. These two spectral sequences both degenerates on the homology $H_*(\mathcal{F}_p^R)$ of the complex $C_*(\mathcal{F}^R_p)$ for the differential $d_*$.

Moreover we can remark that the double complex $C_{*,*}(\mathcal{F}^R_p)$ can be refined in the two following ways:
\begin{equation*}
    C_{a,b}(\mathcal{F}^R_p):= \bigoplus_{\sigma^b \in \Gamma} \bigoplus_{\sigma^a \subseteq \sigma^b} \mathcal{F}_p^R(\sigma^a,\sigma^b) = \bigoplus_{\sigma^a \in \Gamma} \bigoplus_{\sigma^a \subseteq \sigma^b} \mathcal{F}_p^R(\sigma^a,\sigma^b).
\end{equation*}

Then we introduce the following complexes:
\begin{align*}
    \forall \sigma^b \in \Gamma,~~C_{b-a}(\mathcal{F}_p^R(*,\sigma^b)):=\bigoplus_{\sigma^a \subseteq \sigma^b} \mathcal{F}_p^R(\sigma^a,\sigma^b), \\
    \forall \sigma^a \in \Gamma,~~ C_{b-a}(\mathcal{F}_p^R(\sigma^a,*)):=\bigoplus_{\sigma^a \subseteq \sigma^b} \mathcal{F}_p^R(\sigma^a,\sigma^b).
\end{align*}

For every $\sigma^b \in \Gamma$, the differential $d^1_q$ of $C_q(\mathcal{F}^R_p)$ induces a differential $d^{1,\sigma^b}_{b-a}$ of the complex $C_{b-a}(\mathcal{F}_p^R(*,\sigma^b))$, and we obtain an homology $H_{b-a}(\mathcal{F}_p^R(*,\sigma^b))$. For every $\sigma^a \in \Gamma$, the differential $d^2_q$ of $C_q(\mathcal{F}^R_p)$ induces likewise a differential $d^{2,\sigma^a}_{b-a}$ of the complex $C_{b-a}(\mathcal{F}_p^R(\sigma^a,*))$, and we obtain an homology $H_{b-a}(\mathcal{F}_p^R(\sigma^a,*))$. We obtain that the first pages of the spectral sequences $E^*_{*,*}(d^1)$ and $E^*_{*,*}(d^2)$ can be expressed by direct sums of these homology groups:
\begin{align*}
    E^1_{a,b}(d^1)=\bigoplus_{\sigma^b \in \Gamma} H_{b-a}(\mathcal{F}_p^R(*,\sigma^b)), \\
    E^1_{a,b}(d^2)=\bigoplus_{\sigma^a \in \Gamma} H_{b-a}(\mathcal{F}_p^R(\sigma^a,*)).
\end{align*}

From the introduction of the cosheaves $\mathcal{F}_p^R$ we have defined a lot of objects to describe the homology $H_q(\mathcal{F}_p^R)$ of these cosheaves. However, the partial Poincaré duality concerns both a cohomology and an homology. Let us now construct the dual objects, the sheaves $\mathcal{F}^p_R$ and their cohomology. For any integer $p \geq 0$ and $(\sigma,\tau) \in \Omega$ we define:
\begin{equation*}
    \mathcal{F}^p_R(\sigma,\tau):= \mathrm{Hom}_{R}(\mathcal{F}^R_p(\sigma,\tau),R) = \mathrm{Hom}_{\mathbb{Z}}(\mathcal{F}^{\mathbb{Z}}_p(\sigma,\tau),R).
\end{equation*}

The sheaves morphisms are dual to the cosheaves morphisms, then we obtain dual complexes of sheaves $C^q(\mathcal{F}_p^R)$ with dual differential $\delta^q:=(d_{q+1})^*$, which induces cohomologies $H^q(\mathcal{F}^p_R)$. There are dual double complexes $C^{a,b}(\mathcal{F}^p_R)$ with dual morphisms $\delta_1^{a,b}:=(d^1_{a-1,b})^*$ and $\delta_2^{a,b}:=(d^2_{a,b+1})^*$ which induces by direct sums the differentials $\delta^q_1:=(d^q_1)^*$ and $\delta^q_2:=(d^q_2)^*$ satisfying $\delta^q_1+\delta_q^2=\delta^q$. It defines also two spectral sequences $E^{*,*}_*(\delta_1)$ and $E^{*,*}_*(\delta_2)$ whose $0^{th}$ page are $E^{*,*}_0(\delta_1)=E^{*,*}_0(\delta_2)=C^{*,*}(\mathcal{F}^p_R$, whose $0^{th}$ differentials are $\partial_0^{*,*}(\delta_1):=\delta_1^{*,*}$ and $\partial_0^{*,*}(\delta_2):=\delta_2^{*,*}$ respectively, and whose differentials of the successive pages are induced by the differnetial $\delta^* = \delta_1^* + \delta_2^*$ of the complex $C^*(\mathcal{F}^p_R)$. Those two spectral sequence both degenerates on the cohomology $H^*(\mathcal{F}^p_R)$ of the complex $C^*(\mathbb{F}^R_p)$ for the differential $\delta^*$.

Moreover, we can also construct for every $\sigma^b \in \Gamma$ the dual complex $C^{b-a}(\mathcal{F}^p_R(*,\sigma^b))$, with the dual differential $\delta^{b-a}_{1,\sigma^b}:= (d_{b-a+1}^{1,\sigma^b})^*$ inducing the cohomology groups $H^{b-a}(\mathcal{F}^p_R(*,\sigma^b))$, and for every $\sigma^a \in \Gamma$ the dual complex $C^{b-a}(\mathcal{F}^p_R(\sigma^a,*))$, with the dual differential $\delta^{b-a}_{2,\sigma^a}:= (d_{b-a+1}^{2,\sigma^a})^*$ inducing cohomology groups $H^{b-a}(\mathcal{F}^p_R(\sigma^a,*))$. We obtain that the first pages of the spectral sequences $E^*_{*,*}(d^1)$ and $E^*_{*,*}(d^2)$ can be expressed by direct sums of these homology groups:
\begin{align*}
    E_1^{a,b}(\delta_1)=\bigoplus_{\sigma^b \in \Gamma} H^{b-a}(\mathcal{F}^p_R(*,\sigma^b)), \\
    E_1^{a,b}(\delta_2)=\bigoplus_{\sigma^a \in \Gamma} H^{b-a}(\mathcal{F}^p_R(\sigma^a,*)).
\end{align*}

We have yet constructed tropical homology and cohomology and their associated spectral sequences. Let us describe the notion of $R$-primitivity, $R$-non-singularity and related considerations on the triangulation before defining an associated cap-product and a fundamental class.

\subsection{Around $R$-primitivity and $R$-non-singularity}

\label{secprim}

From the beginning we have not assumed any hypothesis on the lattice polytope $P$ and its triangulation with integral vertices $\Gamma$. We can construct cosheaves and sheaves $\mathcal{F}_p^R$ and $\mathcal{F}^p_R$ and their related homology and cohomology, for any integral domain $R$. However to construct a fundamental class and then to obtain Poincaré duality or partial Poincaré duality, hypothesis will be progressively needed. In this article we have tried to assume the minimal for each proposition and lemma, in order that the reader may be able to use them for establishing even better generalizations of this partial Poincaré duality theorem. Let us begin by define the $R$-non-singularity of the lattice polytope $P$.

\begin{definition}
    (Global $R$-non-singularity) For any vertex $s$ of the lattice polytope $P$, let us denote $L^1_s$ the family of edges of $P$ which have $s$ as a vertex. For any edge of $P$, let us denote $u_e$ a generator of $T_Re$. Then we say that the lattice polytope $P$ is $R$-non-singular if for every vertex $s$ of $P$ the family $\{u_e,~~e \in L^1_s\}$ forms a basis of $M_R$. In particular it implies that the polytope $P$ is simple.

    \label{prim}
\end{definition}

\begin{definition}
    (Local $R$-non-singularity) We say that the lattice polytope $P$ is $R$-non-singular on a face $F$ of $P$ if there exists a free family $\{e_1,...,e_{n-\mathrm{dim}F}\}$ of vectors of $M_R$ such that for every face $G$ of the polytope for which $F$ is a face of $G$, there exists a subfamily of $\{e_{i_1},...,e_{i_{\mathrm{dim}G-\mathrm{dim}F}}\}$ of $\{e_1,...,e_{n-\mathrm{dim}F}\}$ such that:
    \begin{equation*}
        T_RG=T_RF\oplus Re_{i_1}\oplus...\oplus Re_{i_{\mathrm{dim}G-\mathrm{dim}F}}.
    \end{equation*}
    \label{locprim}
\end{definition}

\begin{remark}
    This definition of local $R$-non-singularity is compatible with the definition of global $R$-non-singularity. Indeed, the polytope $P$ is $R$-non-singular if and only if it is $R$-non-singular on every one of its faces.
\end{remark}

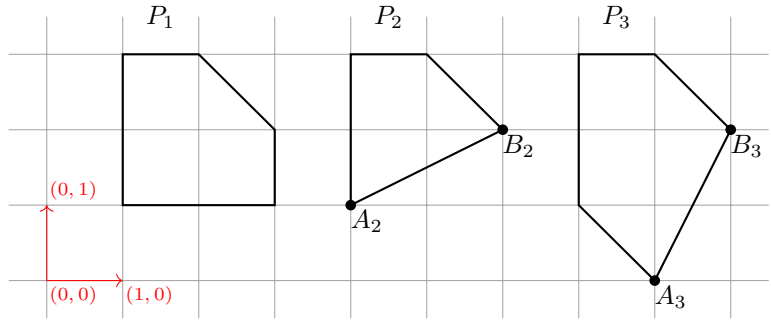
\begin{figure}[h]

\centering

\begin{tikzpicture}[scale=1]
  \draw[step=1, gray!70] (-3.5,-3.5) grid (6.5,0.5);

  \coordinate (A1) at (-2,-2);
  \coordinate (B1) at (-1,0);
  \coordinate (C1) at (0,-1);
  \coordinate (D1) at (-2,0);
  \coordinate (G1) at (0,-2);

  \coordinate (A2) at (1,-2);
  \coordinate (B2) at (2,0);
  \coordinate (C2) at (3,-1);
  \coordinate (D2) at (1,0);
  \coordinate (E2) at (1,-1);

  \coordinate (A3) at (4,-2);
  \coordinate (B3) at (5,0);
  \coordinate (C3) at (6,-1);
  \coordinate (D3) at (4,0);
  \coordinate (F3) at (5,-3);
  
  \draw[thick, black] (D1) -- (A1) -- (G1) -- (C1) --(B1) -- cycle;

  \draw[thick, black] (A2) -- (C2) -- (B2) -- (D2) -- (E2) -- cycle;

  \draw[thick, black] (A3) -- (F3) -- (C3) -- (B3) -- (D3) -- cycle;
  
  \draw[->, red] (-3,-3) -- (-2,-3); 
  \draw[->, red] (-3,-3) -- (-3,-2); 

  \node[color=red,font=\scriptsize] at (-2.65,-1.8) {$(0,1)$};
  \node[color=red,font=\scriptsize] at (-1.65,-3.2) {$(1,0)$};
  \node[color=red,font=\scriptsize] at (-2.65,-3.2) {$(0,0)$};

  \node[color=black,font=\normalsize] at (-1.5,0.5) {$P_1$};
  \node[color=black,font=\normalsize] at (1.5,0.5) {$P_2$};
  \node[color=black,font=\normalsize] at (4.5,0.5) {$P_3$};

  \fill[black] (1,-2) circle (2pt) node[font=\normalsize, xshift=6,yshift=-6] {$A_2$};
  \fill[black] (3,-1) circle (2pt) node[font=\normalsize, xshift=6,yshift=-6] {$B_2$};
  \fill[black] (5,-3) circle (2pt) node[font=\normalsize, xshift=6,yshift=-6] {$A_3$};
  \fill[black] (6,-1) circle (2pt) node[font=\normalsize, xshift=6,yshift=-6] {$B_3$};

\end{tikzpicture}

\caption{R-non-singularity of polytopes}
\label{fignonsing}
\end{figure}

In Figure \ref{fignonsing} the polygon $P_1$ is $R$-non-singular for any integral domain $R$. the polygon $P_2$ $\mathbb{Q}$-non-singular but it is $\mathbb{Z}$-singular and $\mathbb{F}_2$-singular, because it is $\mathbb{Z}$-singular on the vertices $A_2$ and $B_2$, and $\mathbb{F}_2$-singular on the vertex $A_2$. However, it is $\mathbb{F}_2$-non-singular on the vertex $B_2$ even if it is $\mathbb{Z}$-singular on this vertex. The polygon $P_3$ is $\mathbb{Z}$-singular (because $\mathbb{Z}$-singular on the vertices $A_3$ and $B_3$), but it is $\mathbb{Q}$-non-singular and $\mathbb{F}_2$-non-singular.

\begin{remark}
    This definition of local $R$-non-singularity implies that if the polytope $P$ is $R$-non-singular along a face $F$, then the number of faces $G$ of $P$ which has $F$ as a face is at most $2^{n-\mathrm{dim}F}$. However, as a general property of convex polytopes, this number is at least $2^{n-\mathrm{dim}F}$, thus it is equal to $2^{n-\mathrm{dim}F}$, and we obtain conversely that for any subfamily $\{e_{i_1},...,e_{i_{k-\mathrm{dim}F}}\}$ of $\{e_1,...,e_{n-\mathrm{dim}F}\}$, there exists a face $G$ of dimension $k$ of $P$ which has $F$ as a face such that:
    \begin{equation*}
        T_RG=T_RF\oplus Re_{i_1}\oplus...\oplus Re_{i_{\mathrm{dim}G-\mathrm{dim}F}}.
    \end{equation*}
    \label{bij}
\end{remark}

\begin{definition}
    (Integral length of an edge, index of a simplex) For any edge $\sigma^1$ of vertices $s_1,s_2 \in \mathbb{Z}^n$ we define the integral length $|\sigma^1|$ as being the greater common divisor of the coordinates of $s_2 - s_1$ in a basis of $\mathbb{Z}^n$. It does not depend on the choice of a basis of $\mathbb{Z}^n$ and on the choice of the ordering of $s_1,s_2$. For any simplex $\tau$ with vertices in $\mathbb{Z}^n$ we define the index $|\tau|$ of the simplex $\tau$ as being the greater common divisor of the integral lengths $|\sigma^1|$ for $\sigma^1$ edge of $\tau$.
    \label{index}
\end{definition}

\begin{definition}
    ($R$-primitivity of a simplex) A $k$-simplex $\sigma$ of $\mathbb{Z}^n$ with integral vertices $a_0,a_1,...,a_k$ is said to be $R$-primitive if for any $0\leq i \leq k$ the family of integral vectors $\{\frac{a_j-a_i}{|\sigma|},~~j\in\{0,...,r\}\backslash\{i\}\}$ forms a basis of $T_R\sigma$ ($|\sigma|$ is the index of $\sigma$ defined just above). Equivalently, it means that the volume of the $k$-simplex $\sigma$ is of the form $m |\sigma|^k/k!$ with $m$ invertible in the integral domain $R$.
    \label{primsimp}
\end{definition}

\begin{remark}
    In the case $R=\mathbb{Z}$, a simplex is $\mathbb{Z}$-primitive if and only if it is a dilatation of a primitive simplex. In the case $R=\mathbb{Q}$, every simplex is $\mathbb{Q}$-primitive.
\end{remark}

\begin{figure}[h]
  \centering
  \begin{tikzpicture}[scale=1]
  \draw[step=1, gray!70] (-3.5,-3.5) grid (3.5,1.5);

  \coordinate (A) at (3,1);
  \coordinate (B) at (1,1);
  \coordinate (C) at (0,0);
  \coordinate (A1) at (-3,-1);
  \coordinate (B1) at (-2,1);
  \coordinate (C1) at (-1,0);
  \coordinate (A2) at (-2,-1);
  \coordinate (B2) at (0,-1);
  \coordinate (C2) at (0,-3);
  \coordinate (A3) at (1,0);
  \coordinate (B3) at (1,-1);
  \coordinate (C3) at (2,-2);
  \coordinate (A4) at (1,-2);
  \coordinate (B4) at (1,-3);
  \coordinate (C4) at (2,-3);
  \coordinate (A5) at (2,-1);
  \coordinate (B5) at (3,0);
  \coordinate (C5) at (3,-3);
  
  \draw[thick, black] (A) -- (B) -- (C) -- cycle;
  \draw[thick, black] (A1) -- (B1) -- (C1) -- cycle;
  \draw[thick, black] (A2) -- (B2) -- (C2) -- cycle;
  \draw[thick, black] (A3) -- (B3) -- (C3) -- cycle;
  \draw[thick, black] (A4) -- (B4) -- (C4) -- cycle;
  \draw[thick, black] (A5) -- (B5) -- (C5) -- cycle;
  \draw[->, red] (-3,-3) -- (-2,-3); 
  \draw[->, red] (-3,-3) -- (-3,-2); 

  \node[color=red,font=\scriptsize] at (-2.65,-1.8) {$(0,1)$};
  \node[color=red,font=\scriptsize] at (-1.65,-3.2) {$(1,0)$};
  \node[color=red,font=\scriptsize] at (-2.65,-3.2) {$(0,0)$};

  \node[color=black,font=\normalsize] at (-2,0) {$T_1$};
  \node[color=black,font=\normalsize] at (1.2,0.7) {$T_2$};
  \node[color=black,font=\normalsize] at (-0.6,-1.6) {$T_3$};
  \node[color=black,font=\normalsize] at (1.3,-1) {$T_4$};
  \node[color=black,font=\normalsize] at (1.3,-2.7) {$T_5$};
  \node[color=black,font=\normalsize] at (2.6,-1.2) {$T_6$};

\end{tikzpicture}

  \caption{$R$-primitivity of triangles}
  \label{primtri}
\end{figure}
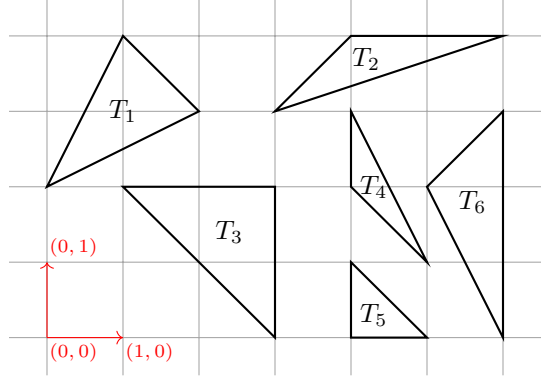

In Figure \ref{primtri} the triangles $T_3$, $T_4$ and $T_5$ are $R$-primitive for any integral domain $R$. The triangles $T_1$, $T_2$ and $T_6$ are not $\mathbb{Z}$-primitive but they are $\mathbb{Q}$-primitive. The triangle $T_2$ is not $\mathbb{F}_2$-primitive but the triangles $T_1$ and $T_6$ are $\mathbb{F}_2$-primitive.

\begin{definition}
    The triangulation $\Gamma$ is said to be $(k,R)$-primitive if any $k$-simplex is $R$-primitive. Similarly, a simplex with vertices in $\mathbb{Z}^n$ is said to be $(k,R)$-primitive if any of its $k$-faces is $R$-primitive. The triangulation $\Gamma$ is said to be almost $R$-primitive if it is $(n,R)$-primitive. It implies that any of its simplex is $R$-primitive. The triangulation $\Gamma$ is said to be almost $R$-primitive if it is $(n-1,R)$-primitive. Similarly, a simplex with vertices in $\mathbb{Z}^n$ is said to be almost $R$-primitive if any of its facet is $R$-primitive.
\end{definition}

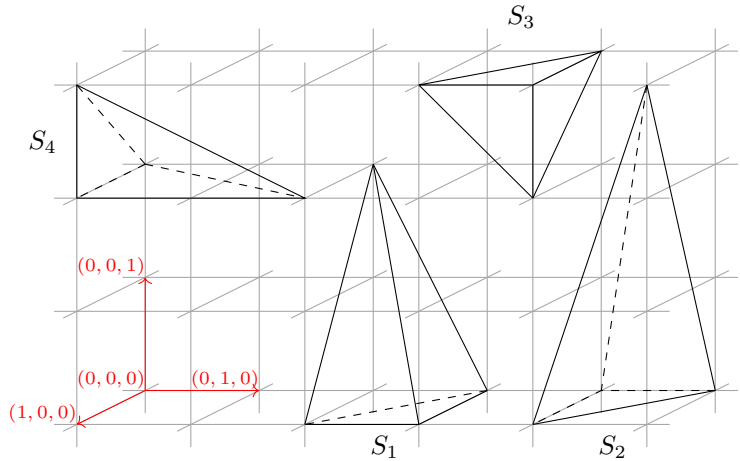
\begin{figure}[h]
  \centering
  \begin{tikzpicture}[
  x={(1cm,0cm)},
  y={(-0.6cm,-0.3cm)},
  z={(0cm,1cm)},
  scale=1.5]
  \draw[gray!70] (-2,0,-0.2) -- (-2,0,3.2);
  \draw[gray!70] (-2,1,-0.2) -- (-2,1,3.2);
  \draw[gray!70] (-1,0,-0.2) -- (-1,0,3.2);
  \draw[gray!70] (-1,1,-0.2) -- (-1,1,3.2);
  \draw[gray!70] (0,0,-0.2) -- (0,0,3.2);
  \draw[gray!70] (0,1,-0.2) -- (0,1,3.2);
  \draw[gray!70] (1,0,-0.2) -- (1,0,3.2);
  \draw[gray!70] (1,1,-0.2) -- (1,1,3.2);
  \draw[gray!70] (2,0,-0.2) -- (2,0,3.2);
  \draw[gray!70] (2,1,-0.2) -- (2,1,3.2);
  \draw[gray!70] (3,0,-0.2) -- (3,0,3.2);
  \draw[gray!70] (3,1,-0.2) -- (3,1,3.2);

  \draw[gray!70] (-2.2,0,0) -- (3.2,0,0);
  \draw[gray!70] (-2.2,1,0) -- (3.2,1,0);
  \draw[gray!70] (-2.2,0,1) -- (3.2,0,1);
  \draw[gray!70] (-2.2,1,1) -- (3.2,1,1);
  \draw[gray!70] (-2.2,0,2) -- (3.2,0,2);
  \draw[gray!70] (-2.2,1,2) -- (3.2,1,2);
  \draw[gray!70] (-2.2,0,3) -- (3.2,0,3);
  \draw[gray!70] (-2.2,1,3) -- (3.2,1,3);

  \draw[gray!70] (-2,-0.2,0) -- (-2,1.2,0);
  \draw[gray!70] (-1,-0.2,0) -- (-1,1.2,0);
  \draw[gray!70] (0,-0.2,0) -- (0,1.2,0);
  \draw[gray!70] (1,-0.2,0) -- (1,1.2,0);
  \draw[gray!70] (2,-0.2,0) -- (2,1.2,0);
  \draw[gray!70] (3,-0.2,0) -- (3,1.2,0);
  \draw[gray!70] (-2,-0.2,1) -- (-2,1.2,1);
  \draw[gray!70] (-1,-0.2,1) -- (-1,1.2,1);
  \draw[gray!70] (0,-0.2,1) -- (0,1.2,1);
  \draw[gray!70] (1,-0.2,1) -- (1,1.2,1);
  \draw[gray!70] (2,-0.2,1) -- (2,1.2,1);
  \draw[gray!70] (3,-0.2,1) -- (3,1.2,1);
  \draw[gray!70] (-2,-0.2,2) -- (-2,1.2,2);
  \draw[gray!70] (-1,-0.2,2) -- (-1,1.2,2);
  \draw[gray!70] (0,-0.2,2) -- (0,1.2,2);
  \draw[gray!70] (1,-0.2,2) -- (1,1.2,2);
  \draw[gray!70] (2,-0.2,2) -- (2,1.2,2);
  \draw[gray!70] (3,-0.2,2) -- (3,1.2,2);
  \draw[gray!70] (-2,-0.2,3) -- (-2,1.2,3);
  \draw[gray!70] (-1,-0.2,3) -- (-1,1.2,3);
  \draw[gray!70] (0,-0.2,3) -- (0,1.2,3);
  \draw[gray!70] (1,-0.2,3) -- (1,1.2,3);
  \draw[gray!70] (2,-0.2,3) -- (2,1.2,3);
  \draw[gray!70] (3,-0.2,3) -- (3,1.2,3);

  \coordinate (B) at (1,0,0);
  \coordinate (C) at (1,1,0);
  \coordinate (D) at (0,1,0);
  \coordinate (S) at (0,0,2);

  \coordinate (B1) at (3,0,0);
  \coordinate (C1) at (2,0,0);
  \coordinate (D1) at (2,1,0);
  \coordinate (S1) at (3,1,3);

  \coordinate (B2) at (-2,1,2);
  \coordinate (C2) at (-2,1,3);
  \coordinate (D2) at (-2,0,2);
  \coordinate (S2) at (0,1,2);

  \coordinate (B3) at (1,1,3);
  \coordinate (C3) at (2,1,3);
  \coordinate (D3) at (2,1,2);
  \coordinate (S3) at (2,0,3);

  \draw[->,red] (-2,0,0) -- (-1,0,0);
  \draw[->,red] (-2,0,0) -- (-2,1,0);
  \draw[->,red] (-2,0,0) -- (-2,0,1);

  \draw (B) -- (C) -- (D);
  \draw[dashed] (B) -- (D);
  \draw (S) -- (B);
  \draw (S) -- (C);
  \draw (S) -- (D);

  \draw[dashed] (B1) -- (C1) -- (D1);
  \draw (B1) -- (D1);
  \draw (S1) -- (B1);
  \draw[dashed] (S1) -- (C1);
  \draw (S1) -- (D1);

  \draw (S2) -- (B2) -- (C2) -- cycle;
  \draw[dashed] (D2) -- (S2);
  \draw[dashed] (D2) -- (B2);
  \draw[dashed] (D2) -- (C2);

  \draw (S3) -- (B3) -- (C3) -- cycle;
  \draw (D3) -- (S3);
  \draw (D3) -- (B3);
  \draw (D3) -- (C3);

  \node[color=red,font=\scriptsize] at (-2.3,1,0.1) {$(1,0,0)$};
  \node[color=red,font=\scriptsize] at (-2.3,0,1.1) {$(0,0,1)$};
  \node[color=red,font=\scriptsize] at (-2.3,0,0.1) {$(0,0,0)$};
  \node[color=red,font=\scriptsize] at (-1.3,0,0.1) {$(0,1,0)$};

  \node[color=black,font=\normalsize] at (0.1,0,-0.5) {$S_1$};
  \node[color=black,font=\normalsize] at (2.1,0,-0.5) {$S_2$};
  \node[color=black,font=\normalsize] at (1.3,0,3.3) {$S_3$};
  \node[color=black,font=\normalsize] at (-2.3,1,2.5) {$S_4$};

\end{tikzpicture}
  \caption{R-primitivity for $3$-simplices}
  \label{primsimpf}
\end{figure}

In Figure \ref{primsimpf} the simplex $S_3$ is $R$-primitive for any integral domain $R$. The simplices $S_2$, $S_1$ and $S_4$ are not $\mathbb{Z}$-primitive but they are $\mathbb{Q}$-primitive. Moreover any proper faces of the simplices $S_1$ and $S_2$ are $R$-primitive, then $S_1$ and $S_2$ are $(2,R)$-primitive for any integral doamin $R$ (that is almost $R$-primitive). However the simplex $S_4$ is not $(2,\mathbb{F}_2)$-primitive because the facet containing the edge of integral length $2$ are not $\mathbb{F}_2$-primitive, then we have also that $S_4$ is not $\mathbb{F}_2$-primitive. The simplex $S_1$ is also not $\mathbb{F}_2$-primitive but the simplex $S_2$ is $\mathbb{F}_2$-primitive.

\begin{figure}[h]
  \centering
  \begin{tikzpicture}[scale=1]
  \draw[step=1, gray!70] (-3.5,-3.5) grid (6.5,0.5);

  \coordinate (A1) at (-2,-2);
  \coordinate (B1) at (-1,0);
  \coordinate (C1) at (0,-1);
  \coordinate (D1) at (-2,0);
  \coordinate (E1) at (-2,-1);
  \coordinate (F1) at (-1,-2);
  \coordinate (G1) at (0,-2);

  \coordinate (A2) at (1,-2);
  \coordinate (B2) at (2,0);
  \coordinate (C2) at (3,-1);
  \coordinate (D2) at (1,0);
  \coordinate (E2) at (1,-1);
  \coordinate (F2) at (2,-2);
  \coordinate (G2) at (3,-2);
  \coordinate (H2) at (2,-1);

  \coordinate (A3) at (4,-2);
  \coordinate (B3) at (5,0);
  \coordinate (C3) at (6,-1);
  \coordinate (D3) at (4,0);
  \coordinate (E3) at (4,-1);
  \coordinate (F3) at (5,-2);
  \coordinate (G3) at (6,-2);
  
  \draw[thick, black] (A1) -- (B1) -- (C1) -- cycle;
  \draw[thick, black] (D1) -- (B1);
  \draw[thick, black] (E1) -- (B1);
  \draw[thick, black] (A1) -- (D1);
  \draw[thick, black] (F1) -- (C1);
  \draw[thick, black] (G1) -- (C1);
  \draw[thick, black] (A1) -- (G1);

  \draw[thick, black] (A2) -- (B2) -- (C2) -- cycle;
  \draw[thick, black] (D2) -- (B2);
  \draw[thick, black] (E2) -- (B2);
  \draw[thick, black] (A2) -- (D2);
  \draw[thick, black] (F2) -- (C2);
  \draw[thick, black] (G2) -- (C2);
  \draw[thick, black] (A2) -- (G2);
  \draw[thick, black] (H2) -- (A2);
  \draw[thick, black] (H2) -- (B2);
  \draw[thick, black] (H2) -- (C2);

  \draw[thick, black] (A3) -- (B3) -- (C3);
  \draw[thick, black] (D3) -- (B3);
  \draw[thick, black] (E3) -- (B3);
  \draw[thick, black] (A3) -- (D3);
  \draw[thick, black] (G3) -- (C3);
  \draw[thick, black] (A3) -- (G3);
  \draw[thick, black] (F3) -- (B3);
  \draw[thick, black] (G3) -- (B3);
  
  \draw[->, red] (-3,-3) -- (-2,-3); 
  \draw[->, red] (-3,-3) -- (-3,-2); 

  \node[color=red,font=\scriptsize] at (-2.65,-1.8) {$(0,1)$};
  \node[color=red,font=\scriptsize] at (-1.65,-3.2) {$(1,0)$};
  \node[color=red,font=\scriptsize] at (-2.65,-3.2) {$(0,0)$};

  \node[color=black,font=\normalsize] at (-1.5,0.5) {$\Gamma_1$};
  \node[color=black,font=\normalsize] at (1.5,0.5) {$\Gamma_2$};
  \node[color=black,font=\normalsize] at (4.5,0.5) {$\Gamma_3$};

\end{tikzpicture}
  \caption{Different triangulation of a lattice polygon}
  \label{triangulation}
\end{figure}
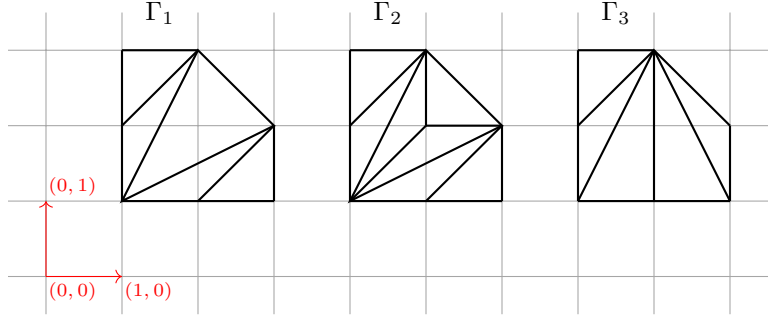

Figure \ref{triangulation} shows three triangulations $\gamma_1$, $\Gamma_2$ and $\Gamma_3$ of a lattice polygon. Then triangulation $\Gamma_2$ is $R$-primitive for any integral domain $R$ The triangualtion $\Gamma_1$ is not $\mathbb{Z}$-primitive but it is $\mathbb{Q}$-primitive and $\mathbb{F}_2$-primitive. The triangulation $\Gamma_3$ is not $\mathbb{Z}$-primitive and not $\mathbb{F}_2$-primitive but it is $\mathbb{Q}$-primitive.

The Poincaré duality theorem will be enounced for $R$-primitive triangulations of a $R$-non-singular lattice polytope. The partial Poincaré duality will be announced for $(k,R)$-primitive triangulations of a $R$-non-singular lattice polytope, with $2\leq k \leq n-1$. However, many lemmas are true for weaker hypothesis. For instance, the existence of a fundamental class requires weaker hypothesis. Then we will define two hypothesis $1$ and $2$ for the triangulation $\Gamma$, weaker than the $(k,R)$-primitivity for $2\leq k \leq n$ but enough to prove certain lemmas needed for Poincaré duality and partial Poincaré duality. To do that let us define the index $|\Gamma|$ of the triangulation $\Gamma$ and the reduced length $|\sigma^1|_r$ of any edge of it:

\begin{definition}
    We define the index $|\Gamma|$ of the triangulation $\Gamma$ as being the greater common divisor of the lengths $|\sigma^1|$ for $\sigma^1 \in \Gamma$. Then we define the reduced length $|\sigma^1|_r$ of each edge $\sigma^1 \in \Gamma$ as being the quotient of the length by this greater common divisor:
    \begin{equation*}
        |\sigma^1|_r:= \frac{|\sigma^1|}{|\Gamma|}.
    \end{equation*}
    \label{redlength}
\end{definition}

Then we can define the hypothesis $1$ and $2$:

\begin{definition}
    (Hypothesis 1 and 2)
    
    \underline{Hypothesis 1:} We say that the triangulation $\Gamma$ satisfies the hypothesis 1 for the integral domain $R$ if for any edge $\sigma^1 \in \Gamma$ the reduced length $|\sigma^1|_r$ is non-zero in the integral domain $R$.
    
    \underline{Hypothesis 2:} We say that the triangulation $\Gamma$ satisfies the hypothesis 2 for the integral domain $R$ if the reduced length $|\sigma^1|_r$ is invertible in $R$.
    \label{max}
\end{definition}

In Figure \ref{triangulation} the triangulations $\Gamma_1$ and $\Gamma_2$ satisfies Hypothesis 1 and 2 for any integral domain $R$. The triangulation $\Gamma_3$ satisfies Hypothesis 1 and 2 for $\mathbb{Q}$ but it satisfies only Hypothesis 1 for $\mathbb{Z}$ and it does not satisfy any of the Hypothesis $1$ and $2$ for $\mathbb{F}_2$. In the next section we will prove that if Hypothesis $1$ is satisfied, the homology group $H_{n-1}(\mathcal{F}_{n-1}^R)$ is a free $R$-module of rank $1$ and we define the fundamental class as a generator. This homology group can be of rank greater than $1$ if Hypothesis 1 is not satisfied and then the fundamental class does not exist. It is the case for the triangulation $\Gamma_3$ of this Figure \ref{triangulation}, the homology $H_{1}(\mathcal{F}_{1}^{\mathbb{F}_2})$ is of rank $2$ and the fundamental class cannot be defined.

The hypothesis $1$ is weaker than the hypothesis $2$. If the characteristic of the integral domain $R$ is non-zero, the hypothesis $1$ and $2$ are equivalent. If the characteristic of the integral domain $R$ is zero, the hypothesis $1$ is empty. To prove that these hypothesis are weaker that an hypothesis of $(k,R)$-primitivity of the triangulation for any $2\leq k \leq n$, it is enough to prove that the $(2,R)$-primitivity implies the hypothesis $2$, which is the following proposition:

\begin{proposition}
    If the triangulation $\Gamma$ is $(2,R)$-primitive, then it satisfies the hypothesis 2 for the integral domain $R$ (and thus also the hypothesis 1).
    \label{jt}
\end{proposition}

Before proving this proposition, we need an intuitive lemma:

\begin{lemma}
    If $\sigma^1_a,\sigma^1_b$ are two edges of the triangulation $\Gamma$ of the polytope $P$, then there exists a sequence $\tau^1_1,\tau^1_2,...,\tau^1_k$ of edges of the triangulation $\Gamma$ such that $\tau^1_1=\sigma^1_a$, $\tau^1_k=\sigma^1_b$ and for all $1\leq i \leq k-1$ the edges $\tau^1_i$ and $\tau^1_{i+1}$ belongs to a same triangle of the triangulation.
    \label{br}
\end{lemma}

This lemma is quite intuitive, it comes from the convexity of the lattice polytope $P$, but it is a bit technical to prove it. To obtain it we have to use the following lemma:

\begin{lemma}
    Let $\sigma^n_1,\sigma^n_2 \in \Gamma$. Then there exists an integer $k\geq 1$ and $\tau^n_1$, $\tau^n_2$,...,$\tau^n_k$ a sequence of $n$-dimensional simplices of $\Gamma$ such that $\sigma^n_1 = \tau^n_1$, $\sigma^n_2 = \tau^n_k$ and for every $1 \leq i \leq k-1$ there exists $\tau^{n-1}_i \in \Gamma$ a codimension $1$ face of both $\tau^n_i$ and $\tau^n_{i+1}$.

    \label{ur}
\end{lemma}

\begin{proof}
    \underline{Proof of Lemma \ref{ur}:} 
    Let $x\in\mathbb{R}^n$ a point of the interior of the simplex $\sigma^n_1$. For any $\mu^a \in \Gamma$, $a \leq n-2$, there exists an affine hyperplan $\mathcal{H}_{x,\mu^a}$ of $\mathbb{R}^n$ which contains $x$ and $\mu^a$. Let us consider the following union:
    \begin{equation*}
        A:= \bigcup_{\mu^a \in \Gamma, ~ a\leq n-2}\mathcal{H}_{x,\mu^a}.
    \end{equation*}
    It is a finite union of affine hyperplan of $\mathbb{R}^n$, thus its interior is empty. The interior of the simplex $\sigma^n_2$ is a non-empty open subset of $\mathbb{R}^n$. Thus, it is not included in $A$. Therefore, there exists a point $y$ of the interior of the simplex $\sigma^n_2$ such that $y \notin A$. Then for any hyperplan $\mathcal{H}_{x,\mu^a}$, the line $(xy)$ is transverse to it and the intersection is $\{x\}$. Then the intersection of the line $(xy)$ with $A$ is $\{x\}$. Thus the intersection of $A$ with the open line segment $]xy[$ is empty. As $A$ contains the union of the $\mu^a \in \Gamma$, $a \leq n-2$, we obtain that the simplices of $\Gamma$ crossed by $]xy[$ are only simplices of dimension $n-1$ or $n$. Moreover, as the line $(xy)$ does not cross any $\mu^a \in \Gamma, ~a\leq n-2$, it means that it does not cross any proper face of the crossed simplices of dimension $n-1$ of $\Gamma$, thus any simplex of $\Gamma$ of dimension $n-1$ crossed by the line $(xy)$ is crossed transversely. Let $a_1,...,a_{k-1}$ the intersection points of $]xy[$ with the simplices of dimension $n-1$, ordered by their distance from $x$, and let $\tau_1^{n-1}$,...,$\tau^{n-1}_{k-1}$ the associated simplices of dimension $n-1$.
    
    By convexity of the polytope $P$, as $x$ and $y$ belong to $P$, the line segment $[xy]$ is contained in $P$. Then the open line segments $]xa_1[$, $]a_1a_2[$, ..., $]a_{k-1}y[$ are contained in $P$. However they don't intersect any simplex of dimension $n-1$ of $\Gamma$. Then they are contained in the union of the interiors of the $n$-simplices of $\Gamma$. As these interiors of $n$-simplices are pairwise disjoint open subsets of $\mathbb{R}^n$, each open line segment $]xa_1[$, $]a_1a_2[$, ..., $]a_{k-1}y[$ must be included into the interior of one $n$-simplex of $\Gamma$. Let us denote $\tau_1^{n}$,...,$\tau^{n}_{k}$ these $n$-simplices of $\Gamma$ whose interior contains $]xa_1[$, $]a_1a_2[$, ..., $]a_{k-1}y[$ respectively. Then we have for every $1\leq i \leq k-1$ that $\tau^{n-1}_i$ is a face of both $\tau^n_i$ and $\tau^n_{i+1}$. Moreover, as $x$ and $y$ belong respectively to the interior of $\sigma^n_1$ and $\sigma^n_2$, we have that $\tau_1^n = \sigma_1^n$ and $\tau_k^n = \sigma_2^n$. It concludes the proof of the lemma.
\end{proof}

Now we can prove the Lemma \ref{br}:

\begin{proof}
    \underline{Proof of the Lemma \ref{br}:}
    Let $\sigma^1_a,\sigma^1_b$ edges of the triangulation $\Gamma$. Then there exists some $n$-simplices $\sigma^n_1,\sigma^n_2$ of $\Gamma$ containing $\sigma^1_a$ and $\sigma^1_b$ respectively. Applying Lemma \ref{ur}, there exists an integer $k$ and a sequence $\tau^n_1,...,\tau^n_k$ of $n$-simplices of $\Gamma$ such that $\sigma^n_1=\tau^n_1$, $\sigma^n_2 = \tau^n_k$ and for every $1 \leq i \leq k-1$ there exists $\tau^{n-1}_i \in \Gamma$ a codimension $1$ face of both $\tau^n_i$ and $\tau^n_{i+1}$.

    Let us fix an edge $\tau^1_i\subseteq\tau^{n-1}_i$ for any $1 \leq i \leq k-1$. Let us also fix the notations $\tau^1_{0}=\sigma^1_a$ and $\tau^1_{k} = \sigma^1_b$. Then for every $0\leq i \leq k-1$, the edges $\tau^1_i$ and $\tau^1_{i+1}$ belong to the same $n$-simplex $\sigma^n_{i+1}$. If $\tau^1_i\neq\tau^1_{i+1}$ there exist a vertex $s$ of $\tau^1_i$ and a vertex $t$ of $\tau^1_{i+1}$ such that $s \neq t$. But $s$ and $t$ are vertices of the same $n$-simplex $\sigma^n_{i+1}$, thus there exists an edge $\epsilon_i^1\subseteq\sigma^n_{i+1}$ of vertices $s$ and $t$. Then $\tau^1_i$ and $\epsilon^1_i$ are edges of the same $n$-simplex $\sigma^n_{i+1}$ and they share the vertex $s$, thus they belongs to a same triangle. Similarly, $\epsilon^1_i$ and $\tau^1_{i+1}$ belong also to a same triangle. In the case $\tau^1_i=\tau^1_{i+1}$ we let $\epsilon^1_i:=\tau^1_i=\tau^1_{i+1}$ and we also have that $\tau^1_i$ and $\epsilon^1_i$ belong to a same triangle and $\tau^1_{i+1}$ and $\epsilon^1_i$ belong to a same triangle.

    Finally the sequence $\tau^1_{0},\epsilon^1_0,\tau^1_1,\epsilon^1_1,...,\tau^1_{k-1}\epsilon^1_{k-1},\tau^1_k$ satisfy that each consecutive edge belong to a same triangle of $\Gamma$, and $\sigma^1_a=\tau^1_0$, $ \sigma^1_b = \tau^1_k$. It concludes the proof of Lemma \ref{br}
\end{proof}

We are now ready to prove the Proposition \ref{jt}:

\begin{proof}
    \underline{Proof of the Proposition \ref{jt}:}
    Let us consider a triangle $\sigma^2\in\Gamma$, of edges $\sigma^1_a$, $\sigma^1_b$, $\sigma^1_c$. As the triangulation $\Gamma$ is $(2,R)$-primitive, this triangle $\sigma^2$ is $R$-primitive. Then the volume of this triangle in a basis of $T_{\mathbb{Z}}\sigma^2$ is of the form $|\sigma^2|^2x/2!$ with $x$ an invertible element of $R$. Then the product of the lengths of two edges (for instance $|\sigma^1_a|| \sigma^1_b|$) divides $|\sigma^2|^2x$. Thus there exists an integer $y$ such that:
    \begin{equation*}
        x=y\frac{|\sigma^1_a|}{|\sigma^2|}\frac{|\sigma^1_b|}{|\sigma^2|}.
    \end{equation*}

    As $x$ is an invertible in $R$ it implies that the three integers $y$, ${|\sigma^1_a|/|\sigma^2|}$ and $|\sigma^1_b|/|\sigma^2|$ are invertible in $R$. Then, denoting $1_R$ the unit element of $R$, there exists $u,v\in R^*$ such that:
    \begin{equation*}
        \left \{ \begin{array}{l}
            |\sigma^1_a|1_R=|\sigma^2|u \\
            |\sigma^1_b|1_R=|\sigma^2|v
        \end{array} \right . .
    \end{equation*}
    Finally we have:
    \begin{equation*}
        |\sigma^1_a|1_R=|\sigma^1_b|v^{-1}u.
    \end{equation*}
    We have proven that if two edges $\sigma^1_a$ and $\sigma^1_b$ of the triangulation $\Gamma$ are edges of a same triangle of $\Gamma$, then there exists an element $w\in R^*$ such that $|\sigma^1_a|1_R=|\sigma^1_b|w$. However, according to Lemma \ref{br}, for any edges $\sigma_a^1$ and $\sigma_a^2$ of the triangulation $\Gamma$, there exists a sequence $\tau^1_1,\tau^1_2,...,\tau^1_k$ of edges of the triangulation such that $\tau^1_1=\sigma^1_a$, $\tau^1_k=\sigma^1_b$ and for all $1\leq i \leq k-1$ the edges $\tau^1_i$ and $\tau^1_{i+1}$ share a vertex. Applying by induction the property we just have proven, we obtain that there exists $z\in R^*$ such that $|\sigma^1_a|1_R=|\sigma^1_b|z$.

    Let us fix an edge $\sigma^1_0\in\Gamma$. For every edge $\tau^1\in \Gamma$, let us denote $z_{\tau^1}$ the element of $R^*$ such that $|\tau^1|1_R=|\sigma^1_0|z_{\tau^1}$. By definition, the index $|\Gamma|$ of the triangulation is the greater common divisor of the lengths $|\tau^1|$ for $\tau^1 \in \Gamma$. The reduced lengths are obtained by dividing the lengths by this index $|\Gamma|$, thus we obtain that the reduced lengths also satisfy:
    \begin{equation*}
        \forall\tau^1 \in \Gamma,~~|\tau^1|_r1_R=|\sigma^1_0|_rz_{\tau^1}.
    \end{equation*}
    There exists a family of integers $\lambda_{\tau^1}\in\mathbb{Z}, ~~\tau^1\in\Gamma$ such that the reduced lengths satisfy the Bezout relation:
    \begin{equation*}
        1=\sum_{\tau^1\in\Gamma}\lambda_{\tau^1}|\tau^1|_r.
    \end{equation*}
    Thus we obtain:
    \begin{equation*}
        1_R=\sum_{\tau^1\in\Gamma}\lambda_{\tau^1}|\tau^1|_r1_R=|\sigma^1_0|_r\sum_{\tau^1\in\Gamma}\lambda_{\tau^1}z_{\tau^1}.
    \end{equation*}
    Therefore $|\sigma^1_0|_r$ is invertible in $R$, and using again the equation $|\tau^1|_r1_R=|\sigma^1_0|_rz_{\tau^1}$, we obtain that for any $\tau^1 \in \Gamma$ the reduced length $|\tau^1|_r$ is invertible in $R$. It is exactly the purpose of hypothesis 2. It ends the proof.
\end{proof}

In the following we precise for each lemma or proposition the assumption we make. It can be noting, it can be the $R$-non-singularity of the polytope, or the local $R$-non-singularity on one of its faces. From the weakest to the strongest, it can be the hypothesis 1, the hypothesis 2, the $(k,R)$-primitivity for some integer $2 \leq k \leq n$, or the $R$-primitivity. The purpose of the following subsection is to construct the cap product.

\subsection{Cap-product}

\begin{definition}
    If $V$ is a $R$-module and $0 \leq p \leq p'$ integers, the contraction (denoted by a dot) is the unique bilinear operator $\bigwedge^p V^* \times \bigwedge^{p'} V \rightarrow \bigwedge^{p'-p} V$ such that:
    \begin{equation*}
        \forall \alpha \in \bigwedge^{p'-p}V^*, ~~\forall w \in \bigwedge^p V^*,~~ \forall v \in \bigwedge^{p'}V, ~~\alpha(w.v)=(\alpha\wedge w)(v).
    \end{equation*}
    \label{defcontr}
\end{definition}

If $V$ is a $R$-module, $I$ a finite subset of $\mathbb{Z}$ and $(v_i)_{i\in I}$ a family of elements of $V$, the notation $\bigwedge_{i\in I}v_i$ may be ambiguous because the sign depends on the order of computation of the wedges of the $v_i,~ I\in I$. But $\mathbb{Z}$ is totally ordered, then $I$ inherits a total order. In the following $\bigwedge_{i\in I}v_i$ denotes the wedge of the family $(v_i)_{i\in I}$ in respect to the total order on $I$ induced by the total order on $\mathbb{Z}$. That is, if $I=\{i_1,...,i_k\}$ with $i_1<i_2<...<i_k$ we let:
\begin{equation*}
    \bigwedge_{i \in I}v_i:= v_{i_1}\wedge v_{i_2}\wedge...\wedge v_{i_k}.
\end{equation*}

Before introducing the cap-product of sheaves and cosheaves, let us ennounce the following property satisfied by the contraction: if $W$ is a sub-$R$-module of $V$ and if $\pi: \bigwedge^p V^* \twoheadrightarrow \bigwedge^p W^*$ is the canonical projection, we have:
\begin{equation}
    \forall p'\geq p,~~~~\forall \beta \in \bigwedge^{p} V^*, ~~ \forall \alpha \in \bigwedge^{p'} W, ~~\beta.\alpha = \pi(\beta).\alpha.
    \label{yutr}
\end{equation}

Now we can define the cap-product of sheaves and cosheaves by the following proposition:

\begin{proposition}
    (Cap-product of sheaves and cosheaves) For every $\sigma \subseteq \gamma \subseteq \tau$ incident simplices of $\Gamma$, with $\mathrm{dim}~ \sigma \geq 1$, and for every integers $0\leq p \leq p'$, we name cap-product the unique bilinear map
    \begin{equation*}
        \cap: \begin{array}{rcl}
            \mathcal{F}^{p}_R(\gamma,\tau) \times \mathcal{F}_{p'}^R(\sigma,\tau)  & \rightarrow & \mathcal{F}_{p'-p}^R(\sigma,\gamma)   \\
            (\beta,\alpha) & \mapsto & \beta \cap \alpha
        \end{array}
    \end{equation*}
    which satisfies the following identity:
    \begin{equation*}
        \forall \overline{\beta} \in \bigwedge^p T_RF_{\tau}, ~~ \forall \alpha \in \mathcal{F}_{p'}^R(\sigma,\tau), ~~ [(\iota_{\sigma,\tau}^{\tau})^*(\overline{\beta})] \cap \alpha = \pi_{\tau}^{\gamma}(\overline{\beta}.\alpha)
    \end{equation*}
    where the dot $\overline{\beta}.\alpha$ denotes the contraction (Definition \ref{defcontr}), where $\pi_{\tau}^{\gamma}$ is the canonical projection $\bigwedge^{p'-p}\frac{N_R}{T_RF_{\tau}^{\perp}} \twoheadrightarrow \bigwedge^{p'-p}\frac{N_R}{T_RF_{\gamma}^{\perp}}$, and $(\iota_{\sigma,\tau}^{\tau})^*:\bigwedge^p T_RF_{\tau} \twoheadrightarrow \mathcal{F}^p_R(\sigma,\tau)$ is the dual of the canonical inclusion $\iota_{\sigma,\tau}^{\tau}:\mathcal{F}^R_p(\sigma,\tau) \hookrightarrow \bigwedge^p\frac{N_R}{T_RF_{\tau}^{\perp}}$
    \label{defcap}
\end{proposition}

\begin{proof}
    As the projection $(\iota_{\sigma,\tau}^{\tau})^*$ is surjective, we just need to prove two statements. Firstly that $\pi_{\tau}^{\gamma}(\overline{\beta}.\alpha)$ belongs to $\mathcal{F}_{p'-p}^R(\sigma,\gamma)$, and secondly that if we have two $\overline{\beta}_1, \overline{\beta}_2  \in \bigwedge^p\frac{N_R}{T_RF_{\tau}^{\perp}}$ such that $(\iota_{\sigma,\tau}^{\tau})^*(\overline{\beta}_1) = (\iota_{\sigma,\tau}^{\tau})^*(\overline{\beta}_2)$ then $\pi_{\tau}^{\gamma}(\overline{\beta}_1.\alpha) = \pi_{\tau}^{\gamma}(\overline{\beta}_2.\alpha)$.

    Let's begin by proving the first statement. Let's write $\alpha = \sum_{\sigma^1 \subseteq \sigma} \alpha_{\sigma^1}$ with $\alpha_{\sigma^1} \in \bigwedge^{p'} \mathcal{F}_{1}^R(\sigma^1,\tau)$. Then, as $\mathcal{F}_{1}^R(\sigma^1,\tau)$ is a sub-$R$-module of $\frac{N_R}{T_RF_{\tau}^{\perp}}$, we can apply Equation (\ref{yutr}) and we get:
    \begin{equation*}
        \forall \sigma^1 \subseteq \sigma,~~~~ ~~ \overline{\beta}.\alpha_{\sigma^1} = (\iota_{\sigma^1,\tau}^{\tau})^*(\overline{\beta}).\alpha_{\sigma^1} \in \mathcal{F}_{p'-p}^R(\sigma^1,\tau).
    \end{equation*}
    It proves that $\pi_{\tau}^{\gamma}(\overline{\beta}.\alpha)$ belongs to $\sum_{\sigma^1 \subseteq \sigma} \mathcal{F}_{p'-p}^R(\sigma^1,\tau) = \mathcal{F}_{p'-p}^R(\sigma,\gamma)$.

    Let's take now $\overline{\beta}_1, \overline{\beta}_2  \in \bigwedge^{p}T_RF_\tau$ such that $(\iota_{\sigma,\tau}^{\tau})^*(\overline{\beta}_1) = (\iota_{\sigma,\tau}^{\tau})^*(\overline{\beta}_2)$. We have for every $\sigma^1 \subseteq\sigma$ and $\alpha_{\sigma^1} \in \mathcal{F}_{p'}^R(\sigma^1,\tau)$:
    \begin{align*}
        \overline{\beta}_1 .\alpha_{\sigma^1} &= (\iota_{\sigma^1,\tau}^{\tau})^*(\overline{\beta}_1).\alpha_{\sigma^1} = (\iota_{\sigma^1,\tau}^{\sigma,\tau})^*\circ (\iota_{\sigma,\tau}^{\tau})^*(\overline{\beta}_1).\alpha_{\sigma^1} \\&= (\iota_{\sigma^1,\tau}^{\sigma,\tau})^*\circ (\iota_{\sigma,\tau}^{\tau})^*(\overline{\beta}_2).\alpha_{\sigma^1} = \overline{\beta}_2 .\alpha_{\sigma^1}.
    \end{align*}
    Then using that every $\alpha \in \mathcal{F}_{p'}^R(\sigma,\tau)$ admits a decomposition $\alpha = \sum_{\sigma^1 \subseteq \sigma} \alpha_{\sigma^1}$ with $\alpha_{\sigma^1} \in \mathcal{F}_{p'}^R(\sigma^1,\tau)$, we conclude the proof of the second statement.
\end{proof}

The cap-product of cosheaves and sheaves can be extended to a cap product at the level of complexes. For every integers $1\leq a \leq c \leq b$ and $0 \leq p \leq p'$, for every $\sigma^c \in \Gamma$ we have the following canonical cap-product:
\begin{equation*}
    \cap: C^{b-c}(\mathcal{F}^{p}_R(\sigma^c,*)) \times C_{a,b}(\mathcal{F}_{p'}^R) \rightarrow C_{c-a}(\mathcal{F}_{p'-p}^R(*,\sigma^c))
\end{equation*}
where for $\sigma^a \subseteq \sigma^c$, the $\mathcal{F}_{p'-p}^R(\sigma^a,\sigma^c)$-component of the cap-product $\beta \cap \alpha$ of a cochain $\beta$ of $C^{b-c}(\mathcal{F}^{p}_R(\sigma^c,*))$ and a chain $\alpha$ of $C_{a,b}(\mathcal{F}_{p'}^R)$ is:
\begin{equation*}
    \left [ \beta \cap \alpha \right ]_{\sigma^a,\sigma^c}:= \sum_{\sigma^c \subseteq \sigma^b} \beta_{\sigma^c,\sigma^b} \cap \alpha_{\sigma^a,\sigma^b}.
\end{equation*}

We can even extent it to bigger complexes, with vanishing the other coordinates of the direct sums. We obtain for every integers $0\leq q \leq q'$ and $0\leq p \leq p'$ a cap-product:
\begin{equation*}
    \cap: C^{q}(\mathcal{F}^{p}_R) \times C_{q'}(\mathcal{F}_{p'}^R) \rightarrow C_{q'-q}(\mathcal{F}_{p'-p}^R)
\end{equation*}
where for every $c-a=q'-q$ and for every $(\sigma^a,\sigma^c) \in \Omega$, the $\mathcal{F}_{p'-p}^R(\sigma^a,\sigma^c)$-component of the cap-product $\beta \cap \alpha$ of a cochain $\beta$ of $C^{q}(\mathcal{F}^{p}_R)$ and a chain $\alpha$ of $C_{q'}(\mathcal{F}_{p'}^R)$ is also:
\begin{equation*}
    \left [ \beta \cap \alpha \right ]_{\sigma^a,\sigma^c}:= \sum_{\sigma^c \subseteq \sigma^{c+q}} \beta_{\sigma^c,\sigma^{c+q}} \cap \alpha_{\sigma^a,\sigma^{c+q}}.
\end{equation*}

\subsection{Fundamental Class}

We want to generalize Poincaré Duality. We already have a cap-product. As usually for Poincaré duality, we have to define a fundamental class, and the Poincaré duality will come from the cap-product with this fundamental class. As usual in topology, we want this fundamental class to belong to the chains of top degree, and more precisely to be a generator of top degree homology group. In particular the fondamental class can exist only if top degree homology group is of rank $1$.

Here the maximal cubic cells of the hypersurface are of dimension $n-1$, indexed by some $(\sigma^1,\sigma^n) \in \Omega$. We can consider the $C_{n-1}(\mathcal{F}_p^{R})$ for $p$ an integer. However, for $p<n-1$, the homology group $H_{n-1}(\mathcal{F}_p^{R})$ is a $R$-module of rank different from $1$ in general. In the following we prove that the cycles of $C_{n-1}(\mathcal{F}_{n-1}^{R})$ are a $R$-module of rank $1$ if we assume that Hypothesis $1$ is satisfied by the triangulation $\Gamma$ for the integral domain $R$ (see Definition \ref{max}), and then the fundamental class will be defined as a generator.

\begin{lemma}
    For every $(\sigma^1,\sigma^n) \in \Omega$ we have $\mathcal{F}_{n-1}^{R}(\sigma^1,\sigma^n) \simeq R$, and in particular $\mathcal{F}_{n-1}^{\mathbb{Z}}(\sigma^1,\sigma^n) \simeq \mathbb{Z}$. Moreover we have also $\mathcal{F}_{n-1}^{R}(\sigma^1,\sigma^n) = \mathcal{F}_{n-1}^{\mathbb{Z}}(\sigma^1,\sigma^n) \otimes R$.
\end{lemma}

\begin{proof}
    For every $(\sigma^1,\sigma^n) \in \Omega$ we have $\mathcal{F}_{n-1}^{\mathbb{Z}}(\sigma^1,\sigma^n) = \bigwedge^{n-1} T_{\mathbb{Z}} \sigma^{1,\perp}$, and $T_{\mathbb{Z}} \sigma^{1,\perp}$ is of rank $n-1$, thus $\mathcal{F}_{n-1}^{\mathbb{Z}}(\sigma^1,\sigma^n) \simeq \mathbb{Z}$.

    Moreover, $T_{\mathbb{Z}}\sigma^{1\perp}$ is saturated in $N$, and then $ \bigwedge^{n-1}T_{\mathbb{Z}}\sigma^{1\perp}$ is saturated in $ \bigwedge^{n-1}N$. It precisely means that  $ \bigwedge^{n-1}T_{R}\sigma^{1\perp} =  \bigwedge^{n-1}T_{\mathbb{Z}}\sigma^{1\perp} \otimes R$ and then $\mathcal{F}_{n-1}^R(\sigma^1,\sigma^n) = \mathcal{F}_{n-1}^{\mathbb{Z}}(\sigma^1,\sigma^n) \otimes R$.
\end{proof}

To construct a fundamental class, we need the following lemma:

\begin{lemma}
   Let us denote $1_R$ the unit element of the integral domain  $R$. Then there exists a family $(1_{\sigma^1,\sigma^n})_{(\sigma^1,\sigma^n) \in \Omega}$ of generators of the $\mathcal{F}_{n-1}^{\mathbb{Z}}(\sigma^1,\sigma^n)$, $ (\sigma^1,\sigma^n) \in \Omega$ such that:
    \begin{equation*}
        d_{n-1} \left ( \bigoplus_{(\sigma^1,\sigma^n) \in \Omega} |\sigma^1|_r 1_{\sigma^1,\sigma^n} \otimes 1_R \right ) = 0.
    \end{equation*}
    \label{yt}
\end{lemma}

\begin{proof}
    The cubical subdivision associated to the poset $\Omega$ is a polyhedral complex, satisfying the diamond property. It has for balancing signature the map $\eta$ used to construct the differential $d_{n-1}$ of the cosheaf $\mathcal{F}_{n-1}^{\mathbb{Z}}$. However we can also define the classical chain complex $C_*(\Omega;\mathbb{Z})$ of $\Omega$ over $\mathbb{Z}$, with the classical differential constructed with the balancing signature $\eta$ and defining the classical polyhedral homology $H_*(\Omega;\mathbb{Z})$ of $\Omega$ over $\mathbb{Z}$.
    
    Let us fix $o(M)$ an orientation of the lattice $M$. For any $\sigma^1 \in \Gamma$ of vertices $s_1,s_2$, an orientation of $\sigma^1$ is equivalent to the data of a generator $u_{\sigma^1}$ of $T_{\mathbb{Z}}\sigma^1$. Then this generator can be completed into a basis $\mathcal{B}$ of $M$ of orientation $o(M)$. Considering the dual basis and removing the dual vector of $u_{\sigma^1}$ we obtain a basis $\mathcal{B} \backslash \{u_{\sigma^1}^* \}$ of $T_{\mathbb{Z}}\sigma^{1\perp}$. Then we have that the orientation of this basis $\mathcal{B} \backslash \{u_{\sigma^1}^* \}$ of $T_{\mathbb{Z}}\sigma^{1\perp}$ does not depend on the choice of the base $\mathcal{B}$ of orientation $o(M)$ which complete $u_{\sigma^1}$ in $M$. It is a canonical way to obtain an orientation on $T_{\mathbb{Z}}\sigma^{1\perp}$ from an orientation of $T_{\mathbb{Z}}\sigma^1$. As an orientation of $T_{\mathbb{Z}}\sigma^{1\perp}$ is equivalent to the data of a generator of $\bigwedge^{n-1}T_{\mathbb{Z}}\sigma^{1\perp}$, we obtain a canonical way to associate a generator of $\bigwedge^{n-1}T_{\mathbb{Z}}\sigma^{1\perp}$ from an orientation of $\sigma^1$. We will denote $g(o(\sigma^1))$ this generator.

    For any simplex $\tau$, and any facet $\sigma$ of $\tau$, let us fix an orientation morphism $f_{(\sigma,\tau)}$ associating an orientation of $\tau$ from an orientation of $\sigma$ in the following way: we can consider an ordering of the vertices of $\tau$ in which the vertices of $\sigma$ occupy the first positions, we obtain an ordering of the vertices of $\sigma$ by restriction. Now we can prove that if we quotient the orderings by even permutations, this map remains well defined, thus it associates an orientation of $\sigma$ from an orientation of $\tau$. As a property of a balancing signature on simplicial complexes, the balancing signature $\rho$ of $\Gamma$ can be obtained by fixing a family $o(\sigma),~ \sigma \in \Gamma$ of orientations of simplices of $\Gamma$ such that for every immediately incident couple $(\sigma,\tau) \in \mathcal{I}(\Gamma)$ we have $f_{(\sigma,\tau)}(o(\tau))=\rho(\sigma,\tau)o(\tau)$. However we have also fixed an orientation $o(M)$ on the lattice $M$, it defines orientations on each $n$-simplex of $\Gamma$, which may be not equal to the fixed orientations $o(\sigma^n),~\sigma^n \in \Gamma$. Then we can introduce the following map $\mu$:
    \begin{equation*}
        \forall \sigma^n \in \Gamma, ~~ \mu(\sigma^n):= \left \{ \begin{array}{ll}
            1 & \text{if } o(\sigma^n) = o(M) \\
            -1 &  \text{if } o(\sigma^{n}) = -o(M)
        \end{array} \right. .
    \end{equation*}
    
    Finally, we can define the following generator of $\mathcal{F}_{n-1}^{\mathbb{Z}}(\sigma^1,\sigma^n) = \bigwedge^{n-1}T_{\mathbb{Z}}\sigma^{1\perp}$:
    \begin{equation*}
        1_{\sigma^1,\sigma^n}:= \mu(\sigma^n)g(o(\sigma^1)) .
    \end{equation*}

    It leads us to define the following direct sum:
    \begin{equation*}
        \alpha:= \bigoplus_{(\sigma^1,\sigma^n) \in \Omega} |\sigma^1|_r1_{\sigma^1,\sigma^n} \otimes 1_R .
    \end{equation*}

    Let us begin by proving that $\alpha$ is a cycle for the differential $d^{2}$. Let $(\sigma^1,\sigma^{n-1}) \in \Omega$. We have two possibilities: either the minimal face $F_{\sigma^{n-1}}$ of $P$ containing $\sigma^{n-1}$ is a facet, or it is $P$ itself.
    
    In the first case $\frac{T_{R}\sigma^{1\perp}}{T_{R}F_{\sigma^{n-1}}^{\perp}}$ is of rank $n-2$ and then $\mathcal{F}_{n-1}^{R}(\sigma^1,\sigma^{n-2}) = 0$, which implies that the coordinate over $(\sigma^1,\sigma^{n-1})$ of $d^{2}_{n-1}(\alpha)$ vanishes:
    \begin{equation*}
        [d^{2}_{n-1}(\alpha)]_{\sigma^1,\sigma^{n-1}} = 0 .
    \end{equation*}

    In the second case, there exist exactly two simplices $\sigma^n_1,\sigma^n_2 \in \Gamma$ such that $\sigma^{n-1}$ is a face of both $\sigma^n_1,\sigma^n_2$. Moreover if we apply the orientation morphisms $f_{(\sigma^{n-1},\sigma^n_1)}$ and $f_{(\sigma^{n-1},\sigma^n_2)}$ on the orientation $o(M)$ of the lattice $M$, we get opposite orientations on $\sigma^{n-1}$:
    \begin{equation*}
        f_{(\sigma^{n-1},\sigma^n_1)}(o(M)) = - f_{(\sigma^{n-1},\sigma^n_2)}(o(M)) .
    \end{equation*}
    Thus we have:
    \begin{equation*}
        \rho(\sigma^{n-1},\sigma^n_1)\mu(\sigma^n_1) = -\rho(\sigma^{n-1},\sigma^n_2)\mu(\sigma^n_2) .
    \end{equation*}
    Therefore we obtain in the second case:
    \begin{align*}
        &[d^{2}_{n-1}(\alpha)]_{\sigma^1,\sigma^{n-1}} \\ &= \rho(\sigma^{n-1},\sigma^n_1)\mu(\sigma^n_1)g(o(\sigma^1)) |\sigma^1|_r \otimes 1_R + \rho(\sigma^{n-1},\sigma^n_2)\mu(\sigma^n_2)g(o(\sigma^1)) |\sigma^1|_r \otimes 1_R \\ & = 0.
    \end{align*}

    Then we have that $\alpha$ is a cycle for the differential $d^{2}$:
    \begin{equation*}
        d^{2}_{n-1}(\alpha) = 0.
    \end{equation*}

    To conclude the proof we have now to prove that it is also a cycle for the differential $d^{1}$. Let $\sigma^2 \in \Gamma$. It is a triangle, let $\sigma^1_1,\sigma^1_2,\sigma^1_3 \in \Gamma$ its edges. For $i=1,2,3$ let us denote $u_{\sigma^1_i}$ the generator of $T_{\mathbb{Z}}\sigma^1$ associated to the orientation $f_{(\sigma^1_i,\sigma^2)}(o(\sigma^2))$ of $\sigma^1$. Then, the orientation morphisms $f_{(\sigma^1_1,\sigma^2)}$, $f_{(\sigma^1_2,\sigma^2)}$ and $f_{(\sigma^1_3,\sigma^2)}$ satisfy the following property:
    \begin{equation*}
        |\sigma^1_1|_ru_{\sigma^1_1} + |\sigma^1_2|_ru_{\sigma^1_2} + |\sigma^1_3|_ru_{\sigma^1_3} = 0.
    \end{equation*}
    It implies that the associated generators $g \circ f_{(\sigma^1_i,\sigma^2)}(o(\sigma^2))$ of $\bigwedge^{n-1}T_{\mathbb{Z}}\sigma^{1\perp}_i$ for $i = 1,2,3$ satisfy:
    \begin{equation*}
        |\sigma^1_1|_r~g \circ f_{(\sigma^1_1,\sigma^2)}(o(\sigma^2)) + |\sigma^1_2|_r~g \circ f_{(\sigma^1_2,\sigma^2)}(o(\sigma^2)) + |\sigma^1_3|_r~g \circ f_{(\sigma^1_3,\sigma^2)}(o(\sigma^2)) = 0.
    \end{equation*}
    However, for all $i=1,2,3$ we have:
    \begin{equation*}
        g \circ f_{(\sigma^1_i,\sigma^2)}(o(\sigma^2)) = \rho(\sigma^1_i,\sigma^2)g(o(\sigma^1_i)).
    \end{equation*}
    Therefore we obtain for every $\sigma^n \in \Gamma, ~\sigma^2 \subseteq \sigma^n$:
    \begin{equation*}
        [d^{1}_{n-1}(\alpha)]_{\sigma^2,\sigma^n} = (-1)^{n-1}\mu(\sigma^n)\left [ \sum_{i=1}^3|\sigma^1_i|_r\rho(\sigma^1_i,\sigma^2)g(o(\sigma^1_i)) \right ] \otimes 1_R = 0.
    \end{equation*}

    The cochain $\alpha$ is then also a cycle for the differential $d^{1}_{n-1}$:
    \begin{equation*}
        d^{1}_{n-1}(\alpha) = 0.
    \end{equation*}
    
    As the differential $d_{n-1}$ is the sum of the differentials $d^{1}_{n-1}$ and $d^{2}_{n-1}$, we obtain that it is a cycle for the differential $d_{n-1}$, which concludes the proof:
    \begin{equation*}
        d_{n-1}\left ( \bigoplus_{(\sigma^1,\sigma^n) \in \Omega} |\sigma^1|_r1_{\sigma^1,\sigma^n} \otimes 1_R \right ) = 0.
    \end{equation*}
\end{proof}

\begin{corollary}
    The rank of $H_{n-1}(\mathcal{F}_{n-1}^{R})$ is at least $1$.
\end{corollary}

\begin{proof}
    If the characteristic of the integral domain $R$ is zero, then every coordinates of the element $\bigoplus_{(\sigma^1,\sigma^n) \in \Omega} |\sigma^1|_r1_{\sigma^1,\sigma^n} \otimes 1_R$ constructed in the previous lemma is non-zero, thus this element itself is non zero, and then $H_{n-1}(\mathcal{F}_{n-1}^{R})$ is non-zero.

    If the characteristic of the integral domain $R$ is a prime number $p>0$, then as the reduced length $|\sigma^1|_r, ~\sigma^1 \in \Gamma$ are coprime, there exists at least one $\sigma^1 \in \Gamma$ such that $p$ does not divide $|\sigma^1|_r$, and thus for this particular edge $\sigma^1$ we have that $|\sigma^1|_r1_{\sigma^1,\sigma^n} \otimes 1_R$ is non-zero. Then the element $\bigoplus_{(\sigma^1,\sigma^n) \in \Omega} |\sigma^1|_r1_{\sigma^1,\sigma^n} \otimes 1_R$ constructed in the previous lemma is non-zero, and thus $H_{n-1}(\mathcal{F}_{n-1}^{R})$ is non-zero.
\end{proof}

We have proven that without hypothesis $H_{n-1}(\mathcal{F}_{n-1}^{R})$ is of rank at least $1$. However we will see in the following that we need more to prove that the rank is exactly $1$, we need to assume that the triangulation $\Gamma$ satisfies the Hypothesis 1 for the integral domain $R$ (Definition \ref{max}). Let us prove some lemma needed for that.

\begin{lemma}
    Assume that the triangulation $\Gamma$ satisfies the hypothesis 1 for the integral domain $R$ (Definition \ref{max}). Let $\alpha \in C_{n-1}(\mathcal{F}_{n-1}^{R})$ such that $d_{n-1}(\alpha) = 0$. Let $\sigma^2 \in \Gamma$ a triangle of edges $\sigma^1_1$, $\sigma^1_2$ and $\sigma^1_3$. For every $\sigma^n \in \Gamma, ~\sigma^2 \subseteq \sigma^n$, if $\alpha_{\sigma^1_1,\sigma^n} = 0$, then $\alpha_{\sigma^1_2,\sigma^n} = 0$ and $\alpha_{\sigma^1_3,\sigma^n} = 0$.
\end{lemma}

\begin{proof}
    Let  $\sigma^n \in \Gamma, ~\sigma^2 \subseteq \sigma^n$ such that $\alpha_{\sigma^1_1,\sigma^n} = 0$. Then we have:
    \begin{equation*}
        [d_{n-1}(\alpha)]_{\sigma^2,\sigma^n} = [d_{n-1}^{1}(\alpha)]_{\sigma^2,\sigma^n} = (-1)^{n-1} \left [ \rho(\sigma^1_2,\sigma^2)\alpha_{\sigma^1_2,\sigma^n} + \rho(\sigma^1_3,\sigma^2)\alpha_{\sigma^1_3,\sigma^n} \right ].
    \end{equation*}
    However we know that $[d_{n-1}(\alpha)]_{\sigma^2,\sigma^n} = 0$, thus we obtain:
    \begin{equation*}
        \rho(\sigma^1_2,\sigma^2)\alpha_{\sigma^1_2,\sigma^n} = - \rho(\sigma^1_3,\sigma^2)\alpha_{\sigma^1_3,\sigma^n}.
    \end{equation*}
    This element both belongs to $\mathcal{F}_{n-1}^{R}(\sigma^1_2,\sigma^n)$ and $\mathcal{F}_{n-1}^{R}(\sigma^1_3,\sigma^n)$. However, as a consequence of the hypothesis $1$, the intersection of the $R$-modules $T_R\sigma^{1\perp}_2$ and $T_R\sigma^{1\perp}_3$ is $\{0\}$, and then the intersection of the $R$-modules $\mathcal{F}_{n-1}^{R}(\sigma^1_2,\sigma^n)$ and $\mathcal{F}_{n-1}^{R}(\sigma^1_3,\sigma^n)$ is also $\{0\}$. Then the previous element is zero. It concludes the proof of the lemma.
\end{proof}

\begin{corollary}
    Assume that the triangulation $\Gamma$ satisfies the hypothesis 1 for the integral domain $R$ (Definition \ref{max}). Let $\alpha \in C_{n-1}(\mathcal{F}_{n-1}^{R})$ such that $d_{n-1}(\alpha) = 0$. For every $\sigma^n \in \Gamma$, if there exists $\sigma^1 \subseteq \sigma^n$ such that $\alpha_{\sigma_1, \sigma^n} = 0$, then for every $\sigma^1 \subseteq\sigma^n$ we have $\alpha_{\sigma_1, \sigma^n} = 0$.
\end{corollary}

\begin{proof}
    Let $\sigma^n \in \Gamma$ and assume that there exists $\sigma^1 \subseteq \sigma^n$ such that $\alpha_{\sigma_1, \sigma^n} = 0$. Let $s$ a vertex of $\sigma^1$. For any $\tau^1 \in \Gamma$, $s \subseteq \tau^1 \subseteq \sigma^n$, as $\tau^1$ and $\sigma^1$ share the vertex $s$, they are equal or they define a unique triangle $\sigma^2 \subseteq \sigma^n$, which satisfy $\sigma^1 \subseteq \sigma^2$ and $\tau^1 \subseteq \sigma^2$. Thus, according to the previous lemma, for any $\tau^1 \in \Gamma$, $s \subseteq \tau^1 \subseteq \sigma^n$ we have $\alpha_{\tau^1,\sigma^n} = 0$.

    Let $\tau^1 \subseteq \sigma^n$ such that $s$ is not a vertex of $\tau^1$. Let $r,t$ the vertices of $\tau^1$. Then $rts$ forms a triangle of $\sigma^n$. The edge $st$ has $s$ for vertex, then according to what we just have said, we have $\alpha_{st,\sigma^n} = 0$. Thus, using the previous lemma for the triangle $rst$ we finally obtain that $\alpha_{\tau^1,\sigma^n} = 0$

    We have proven that no matter if $s$ is a vertex of an edge $\tau^1$ of $\sigma^n$, we have in any case $\alpha_{\tau^1,\sigma^n} = 0$. It concludes the proof of this corollary.
\end{proof}

\begin{lemma}
    Let $\alpha \in C_{n-1}(\mathcal{F}_{n-1}^{R})$ such that $d_{n-1}(\alpha) = 0$. For every $\sigma^n_1,\sigma^n_2 \in \Gamma$ such that there exists $\tau^{n-1} \in \Gamma$ a common face of codimension $1$ of $\sigma^n_1$ and $\sigma^n_2$, for every $\tau^1 \subseteq \tau^{n-1}$, if $\alpha_{\tau^{1},\sigma^n_1} = 0$, then $\alpha_{\tau^{1},\sigma^n_2} = 0$.
    \label{rt}
\end{lemma}

\begin{proof}
    Let $\sigma^n_1,\sigma^n_2 \in \Gamma$ such that there exists $\tau^{n-1} \in \Gamma$ a common face of codimension $1$ of $\sigma^n_1$ and $\sigma^n_2$. Let $\tau^1 \subseteq \tau^{n-1}$. Assume that $\alpha_{\tau^{1},\sigma^n_1} = 0$. If $\sigma^n_1 = \sigma^n_2$, it is clear that $\alpha_{\tau^{1},\sigma^n_2} = 0$. Assume that $\sigma^n_1 \neq \sigma^n_2$. Then the simplices of dimension $n$ of $\Gamma$ which have $\tau^{n-1}$ as a face are exactly $\sigma^n_1$ and $\sigma^n_2$. Then we have:
    \begin{equation*}
        [d_{n-1}(\alpha)]_{\tau^{1},\tau^{n-1}} = \rho(\tau^{n-1},\sigma^n_2) \pi_{\tau^{1},\sigma^n_2}^{\tau^{1},\tau^{n-1}}(\alpha_{\tau^{1},\sigma^n_2}).
    \end{equation*}
    With $\pi_{\tau^{1},\sigma^n_2}^{\tau^{1},\tau^{n-1}}:\mathcal{F}_{n-1}^R(\tau_1,\sigma_2^{n})\twoheadrightarrow \mathcal{F}_{n-1}^R(\tau_1,\tau^n)$ canonical cosheaf morphism of projection. However, as $d_{n-1}(\alpha)$ vanishes, we obtain that $\pi_{\tau^{1},\sigma^n_2}^{\tau^{1},\tau^{n-1}}(\alpha_{\tau^{1},\sigma^n_2}) = 0$.

    Moreover, the minimal face $F_{\tau^{n-1}}$ of the polytope $P$ containing the simplex $\tau^{n-1}$ is $P$ itself, because it is a face of two distinct $n$-dimensional simplices of the triangulation (if it was contained in the border of $P$, it would have been a face of only one $n$-dimensional simplex). We have $F_{\tau^{n-1}} = P = F_{\sigma^n_2}$. Therefore we obtain $\mathcal{F}_{n-1}^{R}(\tau^1,\tau^{n-1}) = \mathcal{F}_{n-1}^{R}(\tau^1,\sigma^n_2)$. The morphism of $R$-module $\pi_{\tau^{1},\sigma^n_2}^{\tau^{1},\tau^{n-1}}$ is a projection from $\mathcal{F}_{n-1}^{R}(\tau^1,\sigma^n_2)$ into $\mathcal{F}_{n-1}^{R}(\tau^1,\tau^{n-1})$ which are identical, thus it the identity morphism. Then we obtain $\alpha_{\tau^{1},\sigma^n_2} = 0$, which concludes the proof of the lemma.
\end{proof}

Now the previous lemma and the previous corollary enable us to define the fundamental class for triangulation which satisfies the hypothesis 1 for the integral domain $R$ by the following proposition:

\begin{proposition}
    (Existence of a fundamental class) Assume that the triangulation $\Gamma$ satisfies Hypothesis 1 for the integral domain $R$ (Definition \ref{max}). Then the homology group $H_{n-1}(\mathcal{F}_{n-1}^{R})$ is a free $R$-module of rank $1$. Moreover, there exists $(1_{\sigma^1,\sigma^n})_{(\sigma^1,\sigma^n) \in \Omega}$ a family of generators of the $\mathcal{F}_{n-1}^{\mathbb{Z}}(\sigma^1,\sigma^n)$, $(\sigma^1,\sigma^n) \in \Omega$ such that the direct sum $[\Omega]:= \bigoplus_{\sigma^1,\sigma^n} |\sigma^1|_r1_{\sigma^1,\sigma^n} \otimes 1_r$ is a generator of $H_{n-1}(\mathcal{F}_{n-1}^{R})$. We call it the fundamental class over $R$ of the poset $\Omega$. It is unique if $R$ is of characteristic $2$, and it is unique up to a multiplication by $-1$ otherwise.
\end{proposition}

\begin{proof}
    Let $\alpha \in C_{n-1}(\mathcal{F}_{n-1}^{R})$ such that $d_{n-1}(\alpha) = 0$. Assume that there exists $(\sigma^1_1,\sigma^n_1) \in \Omega$ such that $\alpha_{\sigma^1_1,\sigma^n_1} = 0$. Let us prove that it implies that $\alpha = 0$.

    Let $(\sigma^1_2,\sigma^n_2) \in \Omega$. If $\sigma^n_1 = \sigma^n_2$, it is clear from the last corollary that $\alpha_{\sigma^1_2,\sigma^n_2} = 0$. Assume now that $\sigma^n_1 \neq \sigma^n_2$. According to the Lemma \ref{ur}, there exist $\tau^n_1$, $\tau^n_2$,...,$\tau^n_k$ a sequence of $n$-dimensional simplices of $\Gamma$ such that $\sigma^n_1 = \tau^n_1$, $\sigma^n_2 = \tau^n_k$ and for every $1 \leq i \leq k-1$ there exists $\tau^{n-1}_i \in \Gamma$ a codimension $1$ face of both $\tau^n_i$ and $\tau^n_{i+1}$. As $\sigma^n_1 \neq \sigma^n_2$, we must have $k \geq 2$.

    For every $1\leq i \leq k-1$, let $\tau^1_i \in \Gamma$ such that $\tau^1_i \subseteq \tau^{n-1}_i$ (it exists because $n \geq 2$). Let us prove by induction that for every $1 \leq i \leq k-1$ we have $\alpha_{\tau^1_i,\tau^n_i} = 0$.
    
    For $i = 1$, as $\tau^{n-1}_1 \subseteq \tau^n_1$ and $\tau^n_1 = \sigma^n_1$, we have that $\tau^1_1$ and $\sigma^1_1$ are edges of the same simplex $\tau^n_1$. However $\alpha_{\sigma^1_1,\sigma^n_1} = 0$. Thus, according to the last corollary we have also $\alpha_{\tau^1_1,\tau^n_1} = 0$.

    Let $1 \leq i \leq k-2$ such that $\alpha_{\tau^1_i,\tau^n_i} = 0$. As $\tau^1_i \subseteq \tau^{n-1}_i$, as $\tau^{n-1}_i$ is a face of both $\tau^n_{i}$ and $\tau^n_{i+1}$, we have $\alpha_{\tau^1_i,\tau^n_{i+1}} = 0$ from Lemma \ref{rt}. Moreover, as $\tau^{n-1}_i$ and $\tau^{n-1}_{i+1}$ are both faces of $\tau^n_{i+1}$, we have that $\tau^1_i$ and $\tau^1_{i+1}$ are edges of the same simplex $\tau^n_{i+1}$. Therefore, according to the last corollary, as $\alpha_{\tau^1_i,\tau^n_{i+1}} = 0$, we have also $\alpha_{\tau^1_{i+1},\tau^n_{i+1}} = 0$.

    Then we have by induction that $\alpha_{\tau^1_{k-1},\tau^n_{k-1}} = 0$. The simplex $\tau^1_{k-1}$ is an edge of the simplex $\tau^{n-1}_{k-1}$ which is a face of both $\tau^{n}_{k-1}$ and $\sigma^n_2$. Then we obtain $\alpha_{\tau^1_{k-1},\sigma^n_2} = 0$ from Lemma \ref{rt}. However, $\sigma^1_2$ is also an edge of $\sigma^n_2$, thus we obtain $\alpha_{\sigma^1_2,\sigma^n_2} = 0$ from the last corollary.

    We have proven that for every $\alpha \in C_{n-1}(\mathcal{F}_{n-1}^{R})$ such that $d_{n-1}(\alpha) = 0$, if there exists $(\sigma^1,\sigma^n) \in \Omega$ such that $\alpha_{\sigma^1,\sigma^n} = 0$, then $\alpha = 0$. We also know from Lemma \ref{yt} that there exists $(1_{\sigma^1,\sigma^n})_{(\sigma^1,\sigma^n) \in \Omega}$ a family of non-zero elements of the $\mathcal{F}_{n-1}^{R}(\sigma^1,\sigma^n)$, $(\sigma^1,\sigma^n) \in \Omega$ such that the direct sum $[\Omega]:= \bigoplus_{\sigma^1,\sigma^n} |\sigma^1|_r1_{\sigma^1,\sigma^n} \otimes 1_R$ satisfies $d_{n-1}([\Omega]) = 0$. Let $\alpha \in C_{n-1}(\mathcal{F}_{n-1}^{R})$ such that $d_{n-1}(\alpha) = 0$ and let us prove that $\alpha$ is proportional to $[\Omega]$.

    If $\alpha = 0$ then we have clearly $\alpha = 0. [\Omega]$. Assume that $\alpha \neq 0$. There exists $(\sigma^1_1,\sigma^n_1) \in \Omega$ such that $\alpha_{\sigma^1_1,\sigma^n_1} \neq 0$. Then, as $1_{\sigma^1_1,\sigma^n_1} \otimes 1_R$ is a generator of $\mathcal{F}_{n-1}^R(\sigma^1_1,\sigma^n_1)$, there exists a non-zero element $\lambda \in R$ such that $\alpha_{\sigma^1_1,\sigma^n_1} = \lambda.(1_{\sigma^1_1,\sigma^n_1} \otimes 1_R)$. Thus we have also:
    \begin{equation*}
        |\sigma^1_1|_r\alpha_{\sigma^1_1,\sigma^n_1} = \lambda.(|\sigma^1_1|_r1_{\sigma^1_1,\sigma^n_1} \otimes 1_R).
    \end{equation*}
    Then the chain $|\sigma^1_1|_r\alpha - \lambda[\Omega]$ satisfies:
    \begin{equation*}
        \left ( |\sigma^1_1|_r\alpha - \lambda[\Omega]  \right )_{\sigma^1_1,\sigma^n_1} = 0.
    \end{equation*}

    According to our conclusion in the first part of the current proof, as $d_{n-1}(|\sigma^1_1|_r\alpha - \lambda[\Omega]) = 0$ we obtain that the whole chain $|\sigma^1_1|_r\alpha - \lambda[\Omega]$ is zero:
    \begin{equation*}
        |\sigma^1_1|_r\alpha = \lambda[\Omega].
    \end{equation*}

    Let us denote for every $(\sigma^1,\sigma^n) \in \Omega$, the coefficient $\omega_{\sigma^1,\sigma^n} \in R$ such that $\alpha_{\sigma^1,\sigma^n} = \omega_{\sigma^1,\sigma^n}.(1_{\sigma^1,\sigma^n}\otimes 1_R)$. We have the identities:
    \begin{equation*}
        \forall(\sigma^1,\sigma^n)\in\Omega,~~ |\sigma^1_1|_r\omega_{\sigma^1,\sigma^n} = |\sigma^1|_r\lambda.
    \end{equation*}

    As the integers $|\sigma^1|_r, ~\sigma^1\in\Gamma$ are coprime, there exists a family $\epsilon_{\sigma^1} \in\mathbb{Z}, ~\sigma^1\in\Gamma$ satisfying the Bezout relation:
    \begin{equation*}
        \sum_{\sigma^1 \in \Gamma}\epsilon_{\sigma^1}|\sigma^1|_r = 1.
    \end{equation*}

    Then, choosing for every $\sigma^1\in\Gamma$ a particular $\sigma^n_{\sigma^1} \in \Gamma$ satisfying that $(\sigma^1,\sigma^n_{\sigma^1})\in\Omega$, we obtain:
    \begin{equation*}
        \lambda=\sum_{\sigma^1 \in \Gamma}\epsilon_{\sigma^1}|\sigma^1_1|_r.\omega_{\sigma^1,\sigma^n_{\sigma^1}} = |\sigma^1_1|_r\sum_{\sigma^1 \in \Gamma}\epsilon_{\sigma^1}.\omega_{\sigma^1,\sigma^n_{\sigma^1}}.
    \end{equation*}

    We have proven that $|\sigma^1_1|_r$ divides $\lambda$ in the integral domain $R$. Moreover, $|\sigma^1_1|_r$ is non-zero according to hypothesis $1$ satisfied by the triangulation for the integral domain $R$ (Definition \ref{max}). Thus denoting $\mu:= \sum_{\sigma^1 \in \Gamma}\epsilon_{\sigma^1}.\omega_{\sigma^1,\sigma^n_{\sigma^1}}$ the quotient, we obtain:
    \begin{equation*}
        \alpha = \mu[\Omega].
    \end{equation*}
    
    Finally we have proven that $[\Omega]$ is a generator of $H_{n-1}(\mathcal{F}_p^{R})$, which concludes the proof.
\end{proof}

\subsection{Cap-product with the fundamental class}

We have proven in Lemma \ref{yt} that without hypothesis we can define the $[\Omega]:=\bigoplus_{\sigma^1,\sigma^n} |\sigma^1|_r1_{\sigma^1,\sigma^n} \otimes 1_R$, with $1_R$ the unit element of $R$, and the $1_{\sigma^1,\sigma^n}$ generators of the $\mathcal{F}_{n-1}^{\mathbb{Z}}(\sigma^1,\sigma^n)$ such that $d_{n-1}([\Omega]) =0$ (see Definition \ref{redlength} for the definition of the reduced length $|\sigma^1|_r$ of an edge $\sigma^1\in\Gamma$). However, without the hypothesis $1$ (see Definition \ref{max}) this element $[\Omega]$ will not be called the fundamental class because it is not a generator of $H_{n-1}(\mathcal{F}_{n-1}^R)$. This homology group can be of rank higher than $1$, and the definition of $[\Omega]$ is not unique up to a multiplication by $-1$. However, the Cap-product with the element $[\Omega]$ is defined without supposing the hypothesis $1$. In the following we will fix $[\Omega]$ without supposing the hypothesis $1$, except if we precise it, and in this case $[\Omega]$ would be the fundamental class.

\begin{definition}
    (Cap-product with $[\Omega]$) For any $n$-dimensional polytope $P$ and triangulation $\Gamma$ with integral vertices we define the following linear maps:
    \begin{equation*}
        \forall \sigma^{n-q}\in \Gamma,~~ \phi_{\sigma^{n-q}}: \begin{array}{rcl}
             C^{q}(\mathcal{F}^{p}_R(\sigma^{n-q},*)) & \rightarrow & C_{n-1-q}(\mathcal{F}_{n-1-p}^R(*,\sigma^{n-q}))  \\
             \beta & \mapsto & \beta \cap [\Omega] ~;
        \end{array}
    \end{equation*}
    \begin{equation*}
        \phi^q: \begin{array}{rcl}
             C^{q}\mathcal{F}^{p}_R) & \rightarrow & C_{n-1-q}(\mathcal{F}_{n-1-p}^R)  \\
             \beta & \mapsto & \beta \cap [\Omega] 
        \end{array}.
    \end{equation*}
    \label{defphi}
\end{definition}

By definition of the cap product, these two maps satisfy the following relation:
    \begin{equation*}
        \phi^q  \left ( \bigoplus_{b-a=q}\bigoplus_{\sigma^a\subseteq\sigma^b} \beta_{\sigma^a,\sigma^b} \right ) = \bigoplus_{\sigma^{n-q}\in\Gamma} \phi_{\sigma^{n-q}}\Bigg(\bigoplus_{\sigma^{n-q}\subseteq\sigma^n}\beta_{\sigma^{n-q},\sigma^n}\Bigg).
    \end{equation*}

The map $\phi^q$ is defined at the level of complexes. With good hypothesis it induces isomorphisms at the level of cohomology and homology. In the following section we enounce some lemma and we prove the partial Poincaré duality and the Poincaré duality using them.

\section{Proof of the partial and complete Poincaré duality}

\subsection{The lemmas and their consequences}

From the beginning of this article we do not remind the index $p$ of the cosheaves $\mathcal{F}_{p}^R$ and the sheaves $\mathcal{F}^{p}_R$ in the notations of the differentials $d^{1}$, $d^{2}$, $d$, $\delta_{1}$, $\delta_{2}$, $\delta$, in the notations of the map $\phi^q$ and in the notation of the spectral sequences $E^*_{*,*}(d^{1})$, $E^*_{*,*}(d^{2})$, $E_*^{*,*}(\delta_{1})$, $E_*^{*,*}(\delta_{2})$. The reason is that the index $p$ was mostly constant and there were no possible ambiguity. However in the next statements there will be possible ambiguity. Thus, if necessarily, we will precise the index $p$ by using the notations $d^{1,(p)}$, $d^{2,(p)}$, $d^{(p)}$, $\delta_{1,(p)}$, $\delta_{2,(p)}$, $\delta_{(p)}$ for the differentials and $E^*_{*,*}(d^{1,(p)})$, $E^*_{*,*}(d^{2,(p)})$, $E_*^{*,*}(\delta_{1,(p)})$, $E_*^{*,*}(\delta_{2,(p)})$ for the spectral sequences. For the map $\phi^q$ we will also use the notation $\phi^q_{(p)}$ and it will design the map starting from $C^{q}(\mathcal{F}^{p}_R)$.

To prove partial and complete Poincaré duality theorem, let's announce three lemma about the first pages of the spectral sequences $E_{*,*}^*(d^{1,(p)})$ and $E^{*,*}_*(\delta_{2,(p)})$. We will detail the proof of those lemma in the next section.

\begin{lemma}
    (For partial Poincaré duality) If the triangulation is $(k,R)$-primitive for an integer $2\leq k \leq n-1$, then for any integer $p<k$ the first page of the spectral sequences $E_{*,*}^*(d^{1,(p)})$ satisfies:
    \begin{equation*}
        E^1_{a,b}(d^{1,(p)}) = \left \{ \begin{array}{ll}
            \mathrm{ker}(d^{1,(p)}_{1,b}) & \text{if } a=1 \\
            \bigoplus_{\sigma^b \in \Gamma}H_{b-a}(\mathcal{F}_p^R(*,\sigma^b)) & \text{if } 1<a\le p+1 \text{ and } k<b\leq n \\
            0 & \text{otherwise}
        \end{array} \right. .
    \end{equation*}
    \label{4}
\end{lemma}

\begin{lemma}
    (For complete Poincaré duality) If the triangulation is $R$-primitive, then for any integer $p$ the first page of the spectral sequences $E_{*,*}^*(d^{1,(p)})$ satisfies:
    \begin{equation*}
        E^1_{a,b}(d^{1,(p)}) = \left \{ \begin{array}{ll}
            \mathrm{ker}(d^{1,(p)}_{1,b}) & \text{if } a=1 \\
            0 & \text{otherwise}
        \end{array} \right. .
    \end{equation*}
    \label{4p}
\end{lemma}

\begin{lemma}
    (For partial and complete Poncaré duality) If the polytope $P$ is $R$-non-singular, then for any integer $p$ the first page of the spectral sequence $E^{*,*}_*(\delta_{2,(p)})$ satisfies the following property:
    \begin{equation*}
        E_1^{a,b}(\delta_{2,(p)}) = \left \{ \begin{array}{ll}
            \mathrm{coker}(\delta_{2,(p)}^{a,n-1}) & \text{if } b=n \\
            0 & \text{otherwise}
        \end{array} \right. .
    \end{equation*}
    \label{5}
\end{lemma}

For a $R$-non-singular polytope, this last lemma means that for any $p$ the spectral sequence $E_*^{*,*}(\delta_{2,(p)})$ vanishes at the first page except on a line. As the degree of the $r$-th page is $(-r,1-r)$, it implies that it degenerates at most at the second page, where it is reduced to a line. As it converges toward the cohomology $H^*(\mathcal{F}^p_R)$, we get the following proposition:

\begin{proposition}
    If the polytope $P$ is $R$-non-singular, then for any integer $p$ we have the following identity:
    \begin{equation*}
        H^{q}(\mathcal{F}^{p}_R) = E_2^{n-q,n}(\delta_{2,(p)}) .
    \end{equation*}
    \label{10a}
\end{proposition}

Similarly, for a $R$-primitive triangulation, the Lemma \ref{4p} means that for any $p$ the spectral sequence $E_{*,*}^*(d^{1,(p)})$ vanishes at the first page except on a colum. As the degree of the $r$-th page is $(1-r,-r)$, it implies that it degenerates at most at the second page, where it is reduced to a column. As it converges toward the cohomoglogy $H_*(\mathcal{F}_p^R)$, we get the following proposition:

\begin{proposition}
    If the triangulation is $R$-primitive, then for any integer $p$ we have the following identity:
    \begin{equation*}
        H_{q}(\mathcal{F}_{p}^R) = E^2_{1,q+1}(d^{1,(p)}) .
    \end{equation*}
    \label{10}
\end{proposition}

If the triangulation is $(k,R)$-primitive for an integer $2\leq k \leq n-1$, then for an integer $p<k$, the first page of the spectral sequence $E_{*,*}^*(d^{1,(p)})$ vanishes except on the half colum $a=1, b \leq k$ and on the rectangle $1 \leq a \leq p+1 ~;~ k<b\leq n $. As the differential of the $r$-th page is of degree $(1-r,-r)$, then for any $r\geq 2$ the terms of this rectangle the differential starting from a term $(1,b)$ with $b<k-p$ arrives to a zero (the term $(2-r,b-r)$), and the differential arriving to a term $(1,b)$ with $b<k-p$ starts from a zero (the term $(r,b+r)$). Indeed, as $r\geq 2$, the term $(r,b+r)$ can be non-zero only if $b+r> k$ and $r\leq p+1$, which implies $b\geq k-p$, in contradiction with $b<k-p$. Thus we obtain the identity $E_{1,b}^2(d^{1,(p)})=E_{1,b}^{\infty}(d^{1,(p)})$ for any $b<k-p$. Moreover, for any $b_0<k-p$, at the second page this term $(1,b_0)$ is the only one which is non-zero on the diagonal of terms $(a,b)$ satisfying $a+b=b_0-1$. However it is this diagonal which contributes to the homology group $H_{b_0-1}(\mathcal{F}_p^R)$, and then this homology group is equal to this only one non-zero term $E_{1,b_0}^{\infty}(d^{1,(p)}) = E_{1,b_0}^2(d^{1,(p)})$. We sum up this conclusion in the following proposition:

\begin{proposition}
    If the triangulation is $(k,R)$-primitive for an integer $2 \leq k \leq n-1$, then we have:
    \begin{equation*}
        \forall p<k, ~~ \forall q <k-p-1, ~~ H_q(\mathcal{F}_p^R) = E^2_{1,q+1}(d^{1,(p)}) .
    \end{equation*}
    \label{10c}
\end{proposition}

Let's announce now tree lemma concerning the map $\phi^q$. As for the previous lemma, we will also prove them in a following part.

\begin{lemma}
    (For partial and complete Poincaré duality) If the triangulation is $(k,R)$-primitive for an integer $2 \leq k \leq n$, then the image of the map $\phi^q_{(p)}$ is the following:
    \begin{equation*}
        \forall q \geq n-k,~~~~\mathrm{im}(\phi^q_{(p)}) = \mathrm{ker}(d^{1,(n-1-p)}_{1,n-q}) .
    \end{equation*}
    \label{lebhb}
    Moreover, for $q<n-k$ we still have $d^{1,(n-1-p)}_{1,n-q}\circ\phi^q_{(p)} = 0$.
\end{lemma}

\begin{lemma}
    If the polytope $P$ is $R$-non-singular and if the triangulation satisfies the hypothesis $2$ for the integral domain $R$ (see Definition \ref{max}), then we have the following identity:
    \begin{equation*}
        \mathrm{ker}(\phi^q_{(p)}) \cap C^{n-q,n}(\mathcal{F}^p_R) = \mathrm{im}(\delta_{2,(p)}^{n-q,n-1}).
    \end{equation*}
    \label{ejnf}
\end{lemma}

\begin{lemma}
    For any lattice polytope and triangulation with integral vertices we have the following relation of commutation:
    \begin{equation*}
        \phi^{q+1}_{(p)} \circ \delta_{1,(p)}^q = (-1)^{q+1} d^{2,(n-1-p)}_{n-q-1} \circ \phi^q_{(p)}.
    \end{equation*}
    \label{8}
\end{lemma}

\subsection{Proof of the complete Poincaré duality}

\begin{corollary}
    (Corollary of Lemmas \ref{4p}, \ref{5} \ref{lebhb} and \ref{ejnf}). If the polytope $P$ is $R$-non-singular and if the triangulation is $R$-primitive, then the maps $\phi^*$ induce the following isomorphisms between the terms of the first pages of the spectral sequences $E_*^{*,*}(\delta_{2,(p)})$ and $E^*_{*,*}(d^{1,(n-1-p)})$:
    \begin{equation*}
        \Tilde{\phi}^q_{1,(p)}: E_1^{n-q,n}(\delta_{2,(p)}) \xrightarrow{\sim} E^1_{1,n-q}(d^{1,(n-1-p)}).
    \end{equation*}
\end{corollary}

\begin{proof}
    As a $R$-primitive triangulation is a $n$-primitive triangulation, we have that the statement of the Lemma \ref{ejnf} is valid for any $q$ in the case of a $R$-primitive triangulation:
    \begin{equation*}
        \forall p,\forall q, ~~~~\mathrm{im}(\phi^q_{(p)}) = \mathrm{ker}(d^{1,(n-1-p)}_{1,n-q}).
    \end{equation*}
    Moreover the Lemma \ref{lebhb} is also valid because the $R$-primitivity of the triangulation implies the hypothesis 2 for the integral domain $R$ (see Proposition \ref{jt}). Then it implies that the restriction of $\phi^q_{(p)}$ to $C^{n-q,n}(\mathcal{F}^p_R)$ can be factorised into an isomorphism defined from $\mathrm{coker}(\delta_{2,(p)}^{n-q-1,n})$ to its image. However from the last equation we know that its image is $\mathrm{ker}(d^{1,(n-1-p)}_{1,n-q})$, thus we have an isomorphism from $\mathrm{coker}(\delta_{2,(p)}^{n-q-1,n})$ to $\mathrm{ker}(d^{1,(n-1-p)}_{1,n-q})$. Then we can conclude the proof using Lemmas \ref{4p} and \ref{5}.
\end{proof}

\begin{proposition}
    If the polytope $P$ is $R$-non-singular and if the triangulation is $R$-primitive, then the maps $\phi^*$ induce the following isomorphisms between the terms of the second pages of the spectral sequences $E_*^{*,*}(\delta_{2,(p)})$ and $E^*_{*,*}(d^{1,(n-1-p)})$:
    \begin{equation*}
        \Tilde{\phi}^q_{2,(p)}: E_2^{n-q,n}(\delta_{2,(p)}) \xrightarrow{\sim} E^2_{1,n-q}(d^{1,(n-1-p)}).
    \end{equation*}
    \label{11}
\end{proposition}

\begin{proof}
    The consequence of Lemmas $\ref{lebhb},\ref{ejnf}$ and $\ref{8}$ is the following global relation of commutation valid for all $p$ and $q$:
    \begin{equation*}
        \phi^{q+1}_{(p)} \circ \delta_{(p)}^q = (-1)^{q+1}d^{(n-1-p)}_{n-q-1} \circ \phi^q_{(p)}.
    \end{equation*}
    
    As the spectral sequences $E_*^{*,*}(\delta_{2,(p)})$ and $E^*_{*,*}(d^{1,(n-1-p)})$ come from a filtration of the cohomology $H^*(\mathcal{F}^p_R))$ and a filtration of the homology $H_*(\mathcal{F}_{n-1-p}^R))$ respectively, the differentials of their successive pages are induced by the associated differentials of complexes $\delta_{(p)}^*$ and $d^{(n-1-p)}_*$ respectively. Let's then denote $\partial_{r,(p)}^{*,*}$ and $\partial^{r,(n-1-p)}_{*,*}$ the differentials of the $r$-th pages $E_r^{*,*}(\delta_{2,(p)})$ and $E^r_{*,*}(d^{1,(n-1-p)})$ respectively. Thus, as we know from the last proposition that the maps $\phi^*_{(p)}$ induce isomorphisms $\Tilde{\phi}^*_{1,(p)}$ between the first pages of the corresponding spectral sequences, we get the following relation between the differential of the first pages $\partial_{1,(p)}^{*,*}$ and $\partial^{1,(n-1-p)}_{*,*}$:
    \begin{equation*}
        \partial^{1,(n-1-p)}_{1,n-q} = (-1)^{q+1} \Tilde{\phi}^{q+1}_{1,(p)}\circ \partial_{1,(p)}^{n-q,n} \circ (\Tilde{\phi}^q_{1,(p)})^{-1}.
    \end{equation*}

    Then the isomorphisms $\Tilde{\phi}^*_{1,(p)}$ are not only isomorphisms between the terms of the first pages of the spectral sequences, but they are also compatible with the differentials $\partial_{1,(p)}^{*,*}$ and $\partial^{1,(n-1-p)}_{*,*}$. Therefore they induce isomorphisms at the second pages:

    \begin{equation*}
        \bar{\phi}^q_{2,(p)}: E_2^{n-q,n}(\delta_{2,(p)}) \xrightarrow{\sim} E^2_{1,n-q}(d^{1,(n-1-p)}).
    \end{equation*}
\end{proof}

To conclude, the Theorem \ref{comp} of Complete Poincaré duality comes directly from the Propositions \ref{10a}, \ref{10} and \ref{11}.
\subsection{Proof of the partial Poincaré duality}
\begin{corollary}
    (Corollary of Lemmas \ref{4}, \ref{5} \ref{lebhb} and \ref{ejnf}). If the polytope $P$ is $R$-non-singular and if the triangulation is $(k,R)$-primitive for an integer $2\leq k \leq n-1$, then for any $p \geq n-k$ the maps $\phi^*_{(p)}$ induce the following isomorphisms between some of the terms of the first pages of the spectral sequences $E_*^{*,*}(\delta_{2,(p)})$ and $E^*_{*,*}(d^{1,(n-1-p)})$:
    \begin{equation*}
       \forall p \geq n-k,~~\forall q\geq n-k,~~ \Tilde{\phi}^q_{1,(p)}: E_1^{n-q,n}(\delta_{2,(p)}) \xrightarrow{\sim} E^1_{1,n-q}(d^{1,(n-1-p)}).
    \end{equation*}
\end{corollary}

\begin{proof}
    The Lemma \ref{ejnf} implies that the restriction of $\phi^q_{(p)}$ to $C^{n-q,n}(\mathcal{F}^p_R)$ can be factorised into an isomorphism defined from $\mathrm{coker}(\delta_{2,(p)}^{n-q-1,n})$ to its image. However the Lemma \ref{lebhb} is also valid because the $(k,R)$-primitivity for an integer $k\geq 2$ implies the hypothesis $2$ (see Proposition \ref{jt}). Then for $q\geq n-k$ the image is $\mathrm{ker}(d^{1,(n-1-p)}_{1,n-q})$, thus if $q\geq n-k$ we have an isomorphism from $\mathrm{coker}(\delta_{2,(p)}^{n-q-1,n})$ to $\mathrm{ker}(d^{1,(n-1-p)}_{1,n-q})$. Then we can conclude the proof using Lemmas \ref{4} and \ref{5} (Lemma \ref{4} have to be applied for $n-1-p$ instead of $p$, that is why we get the condition $n-1-p < k$ and thus $p\geq n-k$).
\end{proof}

\begin{proposition}
    If the polytope $P$ is $R$-non-singular and if the triangulation is $R$-primitive, then for $p \geq n-k$ the maps $\phi^*_{(p)}$ induce the following isomorphisms between some of the terms of the second pages of the spectral sequences $E_*^{*,*}(\delta_{2,(p)})$ and $E^*_{*,*}(d^{1,(n-1-p)})$:
    \begin{equation*}
        \forall p \geq n-k,~~\forall q\geq n-k+1,~~\Tilde{\phi}^q_{2,(p)}: E_2^{n-q,n}(\delta_{2,(p)}) \xrightarrow{\sim} E^2_{1,n-q}(d^{1,(n-1-p)}).
    \end{equation*}
    \label{11p}
\end{proposition}

\begin{proof}
    The consequence of Lemma $\ref{lebhb},\ref{ejnf}$ and $\ref{8}$ is the following global relation of commutation:
    \begin{equation*}
        \phi^{q+1}_{(p)} \circ \delta_{(p)}^q = (-1)^{q+1}d^{(n-1-p)}_{n-q-1} \circ \phi^q_{(p)}.
    \end{equation*}
    
    As the spectral sequences $E_*^{*,*}(\delta_{2,(p)})$ and $E^*_{*,*}(d^{1,(n-1-p)})$ come from a filtration of the cohomology $H^*(\mathcal{F}^p_R))$ and a filtration of the homology $H_*(\mathcal{F}_{n-1-p}^R))$ respectively, the differentials of their successive pages are induced by the associated differentials of complexes $\delta_{(p)}^*$ and $d^{(n-1-p)}_*$ respectively. Let's then denote $\partial_{r,(p)}^{*,*}$ and $\partial^{r,(n-1-p)}_{*,*}$ the differentials of the $r$-th pages $E_r^{*,*}(\delta_{2,(p)})$ and $E^r_{*,*}(d^{1,(n-1-p)})$ respectively. Thus, as we know from the last proposition that for $p \geq n-k$ the maps $\phi^*_{(p)}$ induce isomorphisms $\Tilde{\phi}^*_{1,(p)}$ between the terms $(n-q,n)$ and $(1,n-q)$ for $q\geq n-k$ of the first pages of these spectral sequences, we get the following relation between the differential of the first pages $\partial_{1,(p)}^{*,*}$ and $\partial^{1,(n-1-p)}_{*,*}$:
    \begin{equation*}
        \forall p \geq n-k,~~ \forall q \geq n-k,~~ \partial^{1,(n-1-p)}_{1,n-q} = (-1)^{q+1} \Tilde{\phi}^{q+1}_{1,(p)}\circ \partial_{1,(p)}^{n-q,n} \circ (\Tilde{\phi}^q_{1,(p)})^{-1}.
    \end{equation*}
    
    Let $p \geq n-k$. For any $q \geq n-k$, the first differential starting from $E_1^{n-q,n}(\delta_{2,(p)})$ arrives on $E_1^{n-q-1,n}(\delta_{2,(p)})$ which is also a term in the "isomorphism zone" because $q+1 \geq n-k$. Moreover, for $q \geq n-k$, the first differential arriving on $E_1^{n-q,n}(\delta_{2,(p)})$ starts from $E_1^{n-q+1,n}(\delta_{2,(p)})$ which is a term in the "isomorphism zone" if and only if $q-1\geq n-k$. Therefore we finally have that the maps $\Tilde{\phi}^*_{1,(p)}$ induces thes following isomorphisms at the second page:

    \begin{equation*}
        \forall p \geq n-k, ~~ \forall q \geq n-k+1,~~\forall\bar{\phi}^q_{2,(p)}: E_2^{n-q,n}(\delta_{2,(p)}) \xrightarrow{\sim} E^2_{1,n-q}(d^{1,(n-1-p)}).
    \end{equation*}
\end{proof}

\begin{proof}
    \underline{Proof of Theorem \ref{part} of partial Poincaré duality:}
    
    From Proposition \ref{10c} we have:
    \begin{equation*}
        \forall p\geq n-k,~~ \forall n-1-q<k-n+p,~~H_{n-1-q}(\mathcal{F}_p^R)=E^2_{1,n-q}(d^{1,(n-1-p)}).
    \end{equation*}
    And from Proposition \ref{10a} we have also:
    \begin{equation*}
        \forall p, ~~ \forall q,~~H^q(\mathcal{F}^p_R) = E^{n-q,n}_2(\delta_{2,(p)}).
    \end{equation*}
    However, $n-1-q < k-n+p$ means $2n-k\leq p+q$, if we restrict us to the non-zero zone $0\leq p <n-1$ and $0\leq q < n-1$, it implies $q \geq n-k+1$, and $p\geq n-k+1 \geq n-k$. Then we can apply the last proposition and conclude the proof of partial Poincaré duality theorem.
\end{proof}

\section{Proof of the lemmas}

To prove the lemmas we use many times the notion of a saturated sub-module. Let us recall the definition:

\begin{definition}
    (Saturation) If $R$ is an integral domain, if $L$ is a $R$-module, and $S$ a sub-module of $L$, we say that $S$ is saturated in $L$ if:
    \begin{equation*}
        \forall y\in L, \forall \lambda\in R \backslash\{0\},~~\lambda y \in S \implies y \in S.
    \end{equation*}
    \label{sat}
\end{definition}

\subsection{The homological spectral sequence: proof of Lemma \ref{4} and \ref{4p}}

Let's begin by reminding the following identity satisfied by the first page of the spectral sequence $E^*_{*,*}(d^{1})$:
\begin{equation*}
    E^1_{a,b}(d^{1})= \bigoplus_{\sigma^b \in \Gamma} H_{b-a}(\mathcal{F}_p^R(*,\sigma^b)).
\end{equation*}

It shows that to prove Lemmas \ref{4} and \ref{4p} we just need to study the properties of the homology groups $H_{b-a}(\mathcal{F}_p^R(*,\sigma^b))$. For the Lemma \ref{4p}, we just need to prove that $H_{b-a}(\mathcal{F}_p^R(*,\sigma^b))=0$ for any $a\neq 1$ and $R$-primitive simplex $\sigma^b$, because then we get that the first page of the spectral sequence vanishes at the first page for $a\neq 1$, and then the fact that the terms of the column $a=1$ correspond to the kernel of the differentials $d^1_{1,b}$ is just a consequence of the definition of the spectral sequence $E^*_{*,*}(d^1)$ (the differentials of the $0$-th page corresponds to $d^1$) and the fact that at the $0$-th page we have $E^0_{0,b} = C_{0,b}(\mathcal{F}_p^R)=0$.

For the Lemma \ref{4}, for a $(k,R)$-primitive triangulation with $2 \leq k \leq n-1$, we also need to prove that $H_{b-a}(\mathcal{F}_p^R(*,\sigma^b))=0$ for any $a\neq 1$ and any $R$-primitive simplex $\sigma^b$, which will implies that $E^1_{a,b}(d^{1})=0$ for $a \neq 1$ and $b\leq k$ (because in a $(k,R)$-primitive triangulation, every $b$-simplex for $b\leq k$, is $R$-primitive). However we also need to prove that if $p<k$ and $b>k$, any $(k,R)$-primitive simplex $\sigma^b$ satisfies $H_{b-a}(\mathcal{F}_p^R(*,\sigma^b))=0$ for any $a>p+1$.

Finally, the proof of the Lemma \ref{4} and \ref{4p} is a consequence of the two following propositions that we will prove in this section:

\begin{proposition}
    Let $\sigma^b \in \Gamma$. Let $r:=\mathrm{rank} ~ T_R F_{\sigma^b}$. Assume that the simplex $\sigma^b$ is $R$-primitive (see Definition \ref{primsimp}). Then, for any integer $p \geq 0$, the homology of the chain complex $C_{q}(\mathcal{F}_p(*,\sigma^b))$ equipped with the differential $d^{*,\sigma^b}_{q}$ is torsion-free and satisfies:
    \begin{equation*}
        \forall q \neq b-1, ~~H_q(\mathcal{F}^R_p(*,\sigma^b)) = 0,
    \end{equation*}
    \begin{equation*}
        \mathrm{rank}_R ~ H_{b-1}(\mathcal{F}^R_p(*,\sigma^b)) = \binom{r}{p+1} - \binom{r-b}{p+1}.
    \end{equation*}
    \label{theohom}
\end{proposition}

\begin{proposition}
    Let $\sigma^b$ a $(k,R)$-primitive simplex for an integer $2\leq k \leq n-1$, with $b>k$. Then we have:
    \begin{equation*}
        \forall p<k,~~\forall a>p+1, ~~H_{b-a}(\mathcal{F}_p^R(*,\sigma^b)) = 0.
    \end{equation*}
    \label{prop25}
\end{proposition}

Proposition \ref{theohom} can be deduced from the following proposition due to Philipp Jell, Johanes Rau and Kris Shaw:

\begin{proposition}
    (Philipp Jell, Johanes Rau, Kris Shaw) If a $b$-simplex $\Delta^b$ with integral vertices in $\mathbb{Z}^n$ is primitive (that is that the directing vectors of the adjacent edges of a vertex form a basis of the integral tangent lattice $T_{\mathbb{Z}}\Delta^b$), if this $b$-simplex $\Delta^b$ belongs to a triangulation of a lattice polytope $P$, then the resulting cosheaves $\mathcal{F}_p^{\mathbb{Z}}$ of this triangulation satisfies:
    \begin{equation*}
        \forall q \neq b-1,~~H_q(\mathcal{F}_p^{\mathbb{Z}}(*,\Delta^b))=0.
    \end{equation*}
    \label{jrs}
\end{proposition}

This proposition was formulated in their article \cite{JRS 2018} in terms of matroidal fans. In their proof of their Proposition $5.5$ they write that the tropical Borel-Moore homology $H_{p,q}^{BM}(V;\mathbb{Z})$ of a matroidal fan $V$ of dimension $r$ of $\mathbb{R}^s$ vanishes except if $q=r$. Here a primitive $b$-simplex $\Delta^b$ living in $T_{\mathbb{R}}F_{\Delta^b}$ ($F_{\Delta^b}$ denoting the minimal face of the polytope $P$ containing $\Delta^b$) is dual to a matroidal fan of the form $V:=\frac{N_{\mathbb{R}}}{T_{\mathbb{R}}\Delta^{b\perp}}\oplus H_{trop}(\Delta^b)$ where $H_{trop}(\Delta^b)\subseteq \frac{T_{\mathbb{R}}\Delta^{b\perp}}{T_{\mathbb{R}}F_{\Delta^b}^{\perp}}$ is the tropical hyperplan dual to $\Delta^b\subseteq T_{\mathbb{R}}\Delta^b$. This matroidal fan lives in $\frac{N_{\mathbb{R}}}{T_{\mathbb{R}}\Delta^{b\perp}}\oplus \frac{T_{\mathbb{R}}\Delta^{b\perp}}{T_{\mathbb{R}}F_{\Delta^b}^{\perp}} = \frac{N_{\mathbb{R}}}{T_{\mathbb{R}}F_{\Delta^b}^{\perp}}  \simeq \mathbb{R}^{\mathrm{dim} F_{\Delta^b}}$. This matroidal fan is of dimension $\mathrm{dim} F_{\Delta^b}-1$, then according to \cite{JRS 2018} its tropical Borel Moore Homology $H_{p,q}^{BM}(V;\mathbb{Z})$ vanishes for all $q \neq \mathrm{dim} F_{\Delta^b}-1$.

However, looking at equation 5.1 in \cite{JRS 2018} we can identify the cellular complex $C^{BM,Cell}_{p,*}(V;\mathbb{Z})$ with the complex $C_*(\mathcal{F}_p^{\mathbb{Z}}(*,\Delta^b))$ in the following way:
$$C^{BM,Cell}_{p,q}(V;\mathbb{Z})=C_{q-\mathrm{dim} F_{\Delta^b}+b}(\mathcal{F}_p^{\mathbb{Z}}(*,\Delta^b)).$$
With the identity $H_{p,q}^{BM}(V;\mathbb{Z})=H_{q,p}^{BM,Cell}(V;\mathbb{Z})$ from equation 5.1 of \cite{JRS 2018}, it leads to:
\begin{equation*}
    H_{q-\mathrm{dim} F_{\Delta^b}+b}(\mathcal{F}_p^{\mathbb{Z}}(*,\Delta^b)) = H^{BM}_{p,q}(V;\mathbb{Z}).
\end{equation*}

Finally we obtain that in our case the claim $H^{BM}_{p,q}(V;\mathbb{Z})=0$ for all $q \neq \mathrm{dim} F_{\Delta^b}-1$ of Philipp Jell, Johanes Rau, Kris Shaw can be reformulated into the Proposition \ref{jrs}. As a consequence of this Proposition \ref{jrs} we can prove the Proposition \ref{theohom}:

\begin{proof}
    \underline{Proof of Proposition \ref{theohom}:} Consider a $R$-primitive simplex $\sigma^b$ of the triangulation $\Gamma$ of our lattice polytope $P$ and denote $r=\mathrm{rank} ~ T_RF_{\sigma^b}$ (where $F_{\sigma^b}$ is the minimal face of $P$ that contains $\sigma^b$). Now denoting $(e_1,...,e_n)$ the canonical basis of $\mathbb{Z}^n$ let us define the following objects:
    \begin{align*}
        &\Delta^b:=\mathrm{Conv}(0,e_1,...,e_b),\\
        &P':=\mathrm{Conv}(-e_1,...,-e_r,e_1,...,e_n), \\
        &\Gamma':=\{\text{faces of } \mathrm{Conv}(0,\epsilon_1e_1,...,\epsilon_re_r,e_{r+1},...,e_n) ;~~\epsilon_1,...,\epsilon_r\in\{-1,+1\}\}.
    \end{align*}
    $\Delta^b$ is a primitive $b$-simplex, $P'$ is a lattice polytope of dimension $n$ of $\mathbb{Z}^n$, $\Gamma'$ is a triangulation of $P'$, $\Delta^b$ is an element of this triangulation $\Gamma'$ and $F'_{\Delta^b}=\mathrm{Conv}(-e_1,...,-e_r,e_1,...,e_r)$ is the minimal face of $P'$ containing $\Delta^b$, it is of dimension $r$. We can apply the Proposition \ref{jrs} to this simplex $\Delta^b$ with the cosheaves $\mathcal{F}_p^{'\mathbb{Z}}$ of this triangulation $\Gamma'$:
    \begin{equation*}
        \forall q \neq b-1,~~H_q(\mathcal{F}_p^{'\mathbb{Z}}(*,\Delta^b))=0.
    \end{equation*}
    According to the thesis of Jules Chenal \cite{Ch 2024}, Proposition 3.1.14, the primitivity of the simplex $\Delta^b$ implies that for any face $\Delta^a$ of $\Delta^b$ the $\mathbb{Z}$-module $\mathcal{F}_p^{'\mathbb{Z}}(\Delta^a,\Delta^b)$ is saturated in $\bigwedge^p\frac{N}{T_{\mathbb{Z}}F_{\Delta^b}^{'\perp}}$ (see Definition \ref{sat} for the saturation of a sub-module).
    
    However, if $A_1,...,A_k$ are saturated $\mathbb{Z}$-sub-module of a $\mathbb{Z}$-module $B$, then for any integral domain $R$ the tensor products $A_1 \otimes R$, ..., $A_k\otimes R$ can be identified as $R$-sub-modules of $B\otimes R$, and if $A_1+...+A_k$ is saturated in $B$, then the tensor product $(A_1+...+A_k) \otimes R$ can also be identified to a sub-module of $B \otimes R$ and it satisfies the identidy $A_1 \otimes R + ... + A_k \otimes R = (A_1+...+A_k) \otimes R$. For any $\Delta^a\subseteq \Delta^b$ we can apply this result to $\mathcal{F}^{'\mathbb{Z}}_p(\Delta^a,\Delta^b)=\sum_{\Delta^1 \subseteq \Delta^a}\mathcal{F}^{'\mathbb{Z}}_p(\Delta^1,\Delta^b)$ knowing that all of the $\mathcal{F}^{'\mathbb{Z}}_p(\Delta^1,\Delta^b)$ and also $\mathcal{F}^{'\mathbb{Z}}_p(\Delta^a,\Delta^b)$ are saturated in $\bigwedge^p\frac{N}{T_{\mathbb{Z}}F_{\Delta^b}^{'\perp}}$. We obtain the following identity of sub-modules of $\bigwedge^p\frac{N_R}{T_RF_{\Delta^b}^{'\perp}}$:
    \begin{equation}
        \forall \Delta^a\subseteq \Delta^b,~~\mathcal{F}_p^{'R}(\Delta^a,\Delta^b) = \mathcal{F}_p^{'\mathbb{Z}}(\Delta^a,\Delta^b) \otimes R.
        \label{tens}
    \end{equation}
    
    Let us denote $d^{'\mathbb{Z},*,\Delta^b}_{b-a}$ and $d^{'R,*,\Delta^b}_{b-a}$ the differentials of $C_{b-a}(\mathcal{F}^{'\mathbb{Z}}_p(*,\Delta^b))$ and $C_{b-a}(\mathcal{F}^{'R}_p(*,\Delta^b))$ respectively. As they are induced by the inclusions $\mathcal{F}_p^{'\mathbb{Z}}(\Delta^a,\Delta^b)\hookrightarrow \mathcal{F}_p^{'\mathbb{Z}}(\Delta^{a+1},\Delta^b)$ and $\mathcal{F}_p^{'R}(\Delta^a,\Delta^b)\hookrightarrow \mathcal{F}_p^{'R}(\Delta^{a+1},\Delta^b)$ respectively, it implies that they commute with the morphisms $x \mapsto x \otimes \lambda$, for any $\lambda \in R$. Then the image $\mathrm{im}~d^{'R,*,\Delta^b}_{b-a}$ satisfies:
    \begin{equation*}
        \mathrm{im}~d^{'R,*,\Delta^b}_{b-a}=\{y_1\otimes \lambda_1+...+y_k\otimes\lambda_k,~k\geq 1,~y_1,...,y_k\in \mathrm{im}~d^{'\mathbb{Z},*,\Delta^b}_{b-a},~\lambda_1,...,\lambda_k\in R\}.
    \end{equation*}

    In the previous identity the sums of the tensor products $y_i \otimes\lambda_i$ are computed as elements of $\bigoplus_{\Delta^{a+1} \subseteq \Delta^b} \bigwedge^p\frac{N}{T_{\mathbb{Z}}F_{\Delta^b}^{'\perp}} \otimes R = \bigoplus_{\Delta^{a+1} \subseteq \Delta^b} \bigwedge^p\frac{N_R}{T_{R}F_{\Delta^b}^{'\perp}}$. To conclude that it is equal to $\mathrm{im}~d^{'\mathbb{Z},*,\Delta^b}_{b-a} \otimes R$ we need to have the saturation of $\mathrm{im}~d^{'\mathbb{Z},*,\Delta^b}_{b-a}$ in $\bigoplus_{\Delta^{a+1} \subseteq \Delta^b} \bigwedge^p\frac{N}{T_{\mathbb{Z}}F_{\Delta^b}^{'\perp}}$. The saturation of $\mathcal{F}_p^{'\mathbb{Z}}(\Delta^a,\Delta^b)$ in $\bigwedge^p\frac{N}{T_{\mathbb{Z}}F_{\Delta^b}^{'\perp}}$ implies the saturation of $C_{b-a}(\mathcal{F}^{'\mathbb{Z}}_p(*,\Delta^b))$ in $\bigoplus_{\Delta^a \subseteq \Delta^b} \bigwedge^p\frac{N}{T_{\mathbb{Z}}F_{\Delta^b}^{'\perp}}$ which implies also the saturation of the kernel $\mathrm{ker}~d^{'\mathbb{Z},*,\Delta^b}_{b-a}$ in $\bigoplus_{\Delta^a \subseteq \Delta^b} \bigwedge^p\frac{N}{T_{\mathbb{Z}}F_{\Delta^b}^{'\perp}}$. However the homology vanishes for $b-a\neq b-1$, then the image $\mathrm{im}~d^{'\mathbb{Z},*,\Delta^b}_{b-a}$ is also saturated in $\bigoplus_{\Delta^{a+1} \subseteq \Delta^b} \bigwedge^p\frac{N}{T_{\mathbb{Z}}F_{\Delta^b}^{'\perp}}$ if $a \neq 0$. However, for $a=0$ the complex is zero and the image of the differential is also zero. Then for any $a$ the image $\mathrm{im}~d^{'\mathbb{Z},*,\Delta^b}_{b-a}$ is saturated in $\bigoplus_{\Delta^{a+1} \subseteq \Delta^b} \bigwedge^p\frac{N_R}{T_{R}F_{\Delta^b}^{'\perp}}$ and we obtain:
    \begin{equation}
        \forall a,~~~~\mathrm{im}~d^{'R,*,\Delta^b}_{b-a}=\mathrm{im}~d^{'\mathbb{Z},*,\Delta^b}_{b-a} \otimes R.
        \label{tens2}
    \end{equation}

    In particular this identity implies an equality of rank:
    \begin{equation}
        \forall a,~~~~\mathrm{rank}_{\mathbb{Z}}~\mathrm{im}~d^{'\mathbb{Z},*,\Delta^b}_{b-a} = \mathrm{rank}_{R}~\mathrm{im}~d^{'R,*,\Delta^b}_{b-a}.
        \label{rktens2}
    \end{equation}

    However from equation \ref{tens} we have also $C_{b-a}(\mathcal{F}_p^{'R}(*,\Delta^b))=C_{b-a}(\mathcal{F}_p^{'\mathbb{Z}}(*,\Delta^b))\otimes R$ which also leads to an equality of rank:
    \begin{equation}
        \forall a,~~~~\mathrm{rank}_{\mathbb{Z}}~C_{b-a}(\mathcal{F}_p^{'\mathbb{Z}}(*,\Delta^b)) = \mathrm{rank}_{R}~C_{b-a}(\mathcal{F}_p^{'R}(*,\Delta^b)).
        \label{rktens3}
    \end{equation}
    Then using the rank theorem we obtain:
    \begin{equation*}
        \forall a,~~\mathrm{rank}_R ~\mathrm{ker}~d^{'R,*,\Delta^b}_{b-a} = ~\mathrm{rank}_{\mathbb{Z}}~\mathrm{ker}~d^{'\mathbb{Z},*,\Delta^b}_{b-a}.
    \end{equation*}
    As the homology $H_{b-a}(\mathcal{F}_p^{'\mathbb{Z}}(*,\Delta^b))$ vanishes for $a\neq1$, and using equation \ref{rktens2} we obtain:
    \begin{equation*}
        \forall a\neq1,~~\mathrm{rank}_R ~\mathrm{ker}~d^{'R,*,\Delta^b}_{b-a} =\mathrm{rank}_{\mathbb{Z}} ~\mathrm{im}~d^{'{\mathbb{Z}},*,\Delta^b}_{b-a-1} = \mathrm{rank}_{R} ~\mathrm{im}~d^{'R,*,\Delta^b}_{b-a-1}.
    \end{equation*}
    To conclude that the homology $H_{b-a}(\mathcal{F}_p^{'R}(*,\Delta^b))$ vanishes for all $a \neq 1$, it is now enough to prove that the $R$-module $\mathrm{im}~d^{'R,*,\Delta^b}_{b-a-1}$ is saturated in $\bigoplus_{\Delta^a \subseteq \Delta^b}\bigwedge^p\frac{N_R}{T_RF^{'\perp}_{\Delta^b}}$. However the fact that $\mathrm{im}~d^{'\mathbb{Z},*,\Delta^b}_{b-a-1}$ is saturated in $\bigoplus_{\Delta^a \subseteq \Delta^b}\bigwedge^p\frac{N}{T_{\mathbb{Z}}F^{'\perp}_{\Delta^b}}$, combined with the identity \ref{tens2} proves it (and it enven proves it in the case $a=1$). Then we obtain that the homology vanish for all $a\neq 1$, and it is torison-free for $a=1$:
    \begin{align*}
        \forall a\neq 1,~~~~H_{b-a}(\mathcal{F}_p^{'R}(*,\Delta^b))=0, \\
        H_{b-1}(\mathcal{F}_p^{'R}(*,\Delta^b)) \text{ is torsion-free}.
    \end{align*}

    We can compute the rank of $H_{b-1}(\mathcal{F}_p^{'R}(*,\Delta^b))$ using the Euler characteristics. However, from the identity \ref{rktens3} this Euler characteristics does not depend on the integral domain $R$. As the simplex $\Delta^b$ is primitive, any face $\Delta^a\subseteq \Delta^b$ is also primitive. Then we can apply the Lemma 2.4 of \cite{BMR 2024} (applied for $k:=1$, $F:=F'_{\Delta^b}$ and $\sigma:=\Delta^a$), and we obtain:
    \begin{equation*}
        \mathrm{rank}_{\mathbb{F}_2}~\mathcal{F}^{'\mathbb{F}_2}_p(\Delta^a,\Delta^b)=\binom{r}{p}-\binom{r-a}{p-a}.
    \end{equation*}
    As the number of $a$-dimensional faces of a $b$-simplex is $\binom{b+1}{a+1}$, we obtain that the Euler characteristics is the following:
    \begin{align*}
        \chi(\mathcal{F}_p^{'R}(*,\Delta^b)) &= \sum_{a=1}^b(-1)^{b-a}\binom{b+1}{a+1}\Bigg[\binom{r}{p}-\binom{r-a}{p-a}\Bigg] \\
        & =\sum_{a=1}^b(-1)^{b-a}\Bigg [ \binom{b}{a}+\binom{b}{a+1} \Bigg]\Bigg[\binom{r}{p}-\binom{r-a}{p-a}\Bigg] \\
        &=\sum_{a=1}^b(-1)^{b-a}\binom{b}{a}\Bigg[\binom{r}{p}-\binom{r-a}{p-a}-\binom{r}{p}+\binom{r-a+1}{p-a+1}\Bigg]
        \\ &=\sum_{a=1}^b(-1)^{b-a}\binom{b}{a}\binom{r-a}{p+1-a}
        \\ & =\sum_{a=1}^b (-1)^{b-a}\sum_{D\in \mathcal{P}_{a}(\{1,...,b\})}\mathrm{Card}\bigcap_{i\in D} \{X \in \mathcal{P}_{p+1}(\{1,...,r\}),i\in X\}
        \\ & = (-1)^{b-1} \mathrm{Card} \bigcup_{i=1}^b \{X \in \mathcal{P}_{p+1}(\{1,...,r\}),i\in X\}
        \\& =(-1)^{b-1}\Bigg[\binom{r}{p+1}-\binom{r-b}{p+1}\Bigg]
    \end{align*}
    where $\mathcal{P}_{p+1}(\{1,...,r\})$ denotes the set of subsets of cardinal $p+1$ of $\{1,...,r\}$. Then we obtain the rank of $H_{b-1}(\mathcal{F}_p^{'R}(*,\Delta^b))$:
    \begin{equation*}
        \mathrm{rank}~H_{b-1}(\mathcal{F}_p^{'R}(*,\Delta^b)) = (-1)^{b-1}\chi(\mathcal{F}_p^{'R}(*,\Delta^b)) = \binom{r}{p+1}-\binom{r-b}{p+1}.
    \end{equation*}
        
    Let us denote denote $s,a_1,...,a_b$ the vertices of the simplex $\sigma^b$. By definition of the $R$-primitivity (see Definition \ref{primsimp}), the family $\mathcal{B}_1=\{\frac{sa_1}{|\sigma^b|},...,\frac{sa_b}{|\sigma^b|}\}$ forms a $R$-basis of the tangent space $T_R\sigma^b$ (Remind that $|\sigma^b|$ denotes the greater common divisor of the integral lengths of the edges of $\sigma^b$, see Definition \ref{index}). Moreover, as the $\mathbb{Z}$-module $T_{\mathbb{Z}}\sigma^b$ is saturated in $T_{\mathbb{Z}}F_{\sigma^b}$ and $T_{\mathbb{Z}}F_{\sigma^b}$ is saturated in $\mathbb{Z}^n$, there exist integral vectors $u_{b+1},...,u_n$ of $\mathbb{Z}^n$ such that:
    \begin{align*}
        T_{\mathbb{Z}}\sigma^b\oplus\bigoplus_{i=b+1}^r \mathbb{Z}u_i=T_{\mathbb{Z}}F_{\sigma^b}~~~~;~~~~
        T_{\mathbb{Z}}F_{\sigma^b}\oplus\bigoplus_{i=r+1}^n \mathbb{Z}u_i=\mathbb{Z}^n.
    \end{align*}
    Then the family $\mathcal{B}_2:=\mathcal{B}_1\cup\{u_{b+1},...,u_{r}\}$ forms a $R$-basis of $T_RF_{\sigma^b}$ and the family $\mathcal{B}_3:=\mathcal{B}_2\cup\{u_{r+1},...,u_n\}$ forms a $R$-basis of $\mathbb{Z}^n\otimes R$. We can consider the $\mathbb{Z}$-linear map $\psi: \mathbb{Z}^n\rightarrow\mathbb{Z}^n$ sending each $e_i, ~ 1 \leq i \leq b$ to $\frac{sa_1}{|\sigma^b|}$, and each $e_i,~b+1\leq i \leq n$ to $u_i$. Then as $\mathcal{B}_3$ is a $R$-basis of $\mathbb{Z}^n\otimes R$, it induces by tensorisation an isomorphism of $R$-module $\psi_R:\mathbb{Z}^n\otimes R \overset{\sim}{\rightarrow} \mathbb{Z}^n\otimes R$. This isomorphism satisfies $\psi_R(T_RF'_{\Delta^b})=T_RF_{\sigma^b}$, then its dual satisfies $\psi_R^*(T_RF^{\perp}_{\sigma^b})=T_RF^{'\perp}_{\Delta^b}$, and then it induces the following factorization:
    \begin{equation*}
        \Tilde{\psi}_R^*: \frac{N_R}{T_RF^{\perp}_{\sigma^b}} \overset{\sim}{\rightarrow}\frac{N_R}{T_RF^{'\perp}_{\Delta^b}}.
    \end{equation*}
    Moroever the morphism $\psi$ maps the $b$-simplex $\Delta^b$ to a dilation of the $b$-simplex $\sigma^b$ (by the factor $1/|\sigma^b|$), then for each edge of $\Delta^1\subseteq\Delta^b$, denoting $\sigma^1\subseteq \sigma^b$ the corresponding edge of $\sigma^b$, the map $\psi$ induces by restriction and co-restriction an isomorphism $T_{\mathbb{Z}}\Delta^1\overset{\sim}{\rightarrow} T_{\mathbb{Z}}\sigma^1$. Then $\psi_R$ induces an isomorphism $T_{R}\Delta^1\overset{\sim}{\rightarrow} T_{R}\sigma^1$ and thus its dual $\psi_R^*$ induces an isomorphism $T_{R}\sigma^{1\perp}\overset{\sim}{\rightarrow} T_{R}\Delta^{1\perp}$. Finally $\Tilde{\psi}_R^*$ induces by restriction and co-restriction isomorphisms $\frac{T_{R}\sigma^{1\perp}}{T_RF^{\perp}_{\sigma^b}}\overset{\sim}{\rightarrow} \frac{T_{R}\Delta^{1\perp}}{T_RF^{'\perp}_{\Delta^b}}$ which induces also isomorphims $ \bigwedge^p\frac{T_{R}\sigma^{1\perp}}{T_RF^{\perp}_{\sigma^b}}\overset{\sim}{\rightarrow} \bigwedge^p\frac{T_{R}\Delta^{1\perp}}{T_RF^{'\perp}_{\Delta^b}}$. By sums, it induces between any face $\Delta^a\subseteq \Delta^b$ and its corresponding $\sigma^a\subseteq\sigma^b$ an isomorphism $\mathcal{F}^{'R}_p(\Delta^a,\Delta^b)\overset{\sim}{\rightarrow}\mathcal{F}^R_p(\sigma^a,\sigma^b)$. As these family of isomorphims came from a same map $\psi$ sending $\Delta^b$ to $\sigma^b$ compatible with the combinatorial of the simplices $\Delta^b$ and $\sigma^b$, these isomorphisms induces isomorphisms between the chain complexes which commutes with the differential (as soon as we have chosen compatible balancing signature between the poset of the faces of $\Delta^b$ and the poset of the faces of $\sigma^b$). However the homology does not depend on the choice of a balancing signature, then we can choose compatible ones, and we obtain at the end that their homology are isomorphic. It proves the proposition: the homology $H_{b-a}(\mathcal{F}_p^R(*,\sigma^b))$ is torsion-free for all $a$, and we have:
    \begin{align*}
        \forall a \neq 1,~~ H_{b-a}(\mathcal{F}_p^R(*,\sigma^b))&=0, \\
        \mathrm{rank}_R ~H_{b-1}(\mathcal{F}_p^R(*,\sigma^b)) &= \binom{r}{p+1}-\binom{r-b}{p+1}.
    \end{align*}
    
\end{proof}

\begin{proof}
    \underline{Proof of Proposition \ref{prop25}:}
    Let us fix $\sigma^b$ a $(k,R)$-primitive simplex for $2 \leq k \leq n-1$ and $b>k$. Then any $a$-face $\sigma^a$ of $\sigma^b$ with $a\leq k$ is a $R$-primitive simplex. Let us fix a $a$-face $\sigma^a$ of $\sigma^b$, with $a \leq k$. Similarly as in the proof of Proposition \ref{theohom}, we can construct a primitive $a$-dimensional simplex $\Delta^a$ which belongs to a primitive triangulation $\Gamma'$ of a non-singular lattice polytope $P'$, such that the minimal face $F_{\Delta^a}'$ of $P'$ containing $\Delta^a$ is of dimension $\mathrm{dim}~F_{\sigma^b}$ ($F_{\sigma^b}$ denotes the minimal face of the lattice polytope $P$ containing the simplex $\sigma^b$). We have also an injective morphism of $\mathbb{Z}$-module $\psi: \mathbb{Z}^n \rightarrow \mathbb{Z}^n$ which sends the simplex $\Delta^a$ to a dilatation of the simplex $\sigma^a$ such that it induces an isomorphism of $R$-module $\psi_R:\mathbb{Z}^n \otimes R \rightarrow \mathbb{Z}^n \otimes R$ satisfying $\psi_R(T_RF'_{\Delta^a})=T_RF_{\sigma^b}$. Then we obtain similarly that its dual $\psi_R^*$ induces isomophisms $\mathcal{F}^{'R}_p(\sigma^a,\sigma^b)\overset{\sim}{\rightarrow}\mathcal{F}_p^R(\Delta^a,\Delta^a)$. However, we obtain also a similar equation as \ref{tens}: $\mathcal{F}_p^{'R}(\Delta^a,\Delta^a) = \mathcal{F}_p^{'\mathbb{Z}}(\Delta^a,\Delta^a) \otimes R$ which implies that the rank of the $R$-module $\mathcal{F}_p^R(\Delta^a,\Delta^a)$ does not depend on the integral domain $R$. Then we obtain:
    \begin{equation*}
        \mathrm{rank}_R ~\mathcal{F}^R_p(\sigma^a,\sigma^b) = \mathrm{rank}_R ~\mathcal{F}^R_p(\Delta^a,\Delta^a) = \mathrm{rank}_{\mathbb{F}_2}~\mathcal{F}^{'\mathbb{F}_2}_p(\Delta^a,\Delta^a).
    \end{equation*}
    However, as the polytope $P'$ is non-singular and the triangulation $\Gamma'$ is primitive we can use the Lemma 2.4 of \cite{BMR 2024} (applied for $k:=1$, $F:=F'_{\Delta^a}$ and $\sigma:=\Delta^a$):
    \begin{equation*}
        \mathrm{rank}_{\mathbb{F}_2}~\mathcal{F}^{'\mathbb{F}_2}_p(\Delta^a,\Delta^a)=\binom{\mathrm{dim}F'_{\Delta^a}}{p}-\binom{\mathrm{dim}F'_{\Delta^a}-a}{p-a}.
    \end{equation*}
    Therefore using $\mathrm{dim}F'_{\Delta^a}=\mathrm{dim}F_{\sigma^b}$ we finally obtain for every face $\sigma^a$ of $\sigma^b$ with $a\leq k$:
    \begin{equation*}
        \mathrm{rank}_R ~\mathcal{F}^R_p(\sigma^a,\sigma^b) = \binom{\mathrm{dim}F_{\sigma^a}}{p}-\binom{\mathrm{dim}F_{\sigma^a}-a}{p-a}.
    \end{equation*}
    Then we obtain in the case $p<a\leq k$, for every $\sigma^a \subseteq \sigma^b$:
    \begin{equation}
        \mathrm{rank}_R ~ \mathcal{F}_p^R(\sigma^a,\sigma^b) = \binom{\mathrm{dim}F_{\sigma^b}}{p} = \mathrm{rank}~ \bigwedge^p\frac{N_R}{T_RF_{\sigma^b}^{\perp}}.
        \label{equrank}
    \end{equation}
    However, using again the simplex $\Delta^a$ of the triangulation $\Gamma'$ of the polytope $P'$, and the related morphism $\psi: \mathbb{Z}^n \rightarrow \mathbb{Z}^n$, we have similarly as is in the proof of Proposition \ref{theohom} that the $\mathbb{Z}$-module $\mathcal{F}^{'\mathbb{Z}}_p(\Delta^a,\Delta^a)$ is saturated in $\bigwedge^p\frac{N}{T_{\mathbb{Z}}F_{\sigma^b}^{\perp}}$ (see Definition \ref{sat} for the saturation of a sub-module). Using again the similar identity as equation \ref{tens} ($\mathcal{F}_p^{'R}(\Delta^a,\Delta^a) = \mathcal{F}_p^{'\mathbb{Z}}(\Delta^a,\Delta^a) \otimes R$) it implies that $\mathcal{F}^{'R}_p(\Delta^a,\Delta^a)$ is a $R$-module saturated in $\bigwedge^p\frac{N_R}{T_RF_{\sigma^b}^{\perp}}$. Then using the isomorphism $\psi_R^*:\frac{N_R}{T_RF_{\sigma^b}^{\perp}}\overset{\sim}{\rightarrow}\frac{N_R}{T_RF^{'\perp}_{\Delta^a}}$ we finally obtain that $\mathcal{F}_p^R(\sigma^a,\sigma^b)$ is a $R$-module saturated in $\bigwedge^p\frac{N_R}{T_RF_{\sigma^b}^{\perp}}$. However equation \ref{equrank} shows that there are of same rank, thus they are equal. We obtain for any $p<a\leq k$ and every $\sigma^a \subseteq \sigma^b$:
    \begin{equation*}
        \mathcal{F}_p^R(\sigma^a,\sigma^b) = \bigwedge^p\frac{N_R}{T_RF_{\sigma^b}^{\perp}}.
    \end{equation*}
    However, for any face $\sigma^a\subseteq\sigma^b$ with $a\geq k$ there exists a face $\sigma^k \subseteq \sigma^a$ and then if $p<k$ we have $\mathcal{F}_p^R(\sigma^a,\sigma^b)\supseteq\mathcal{F}_p^R(\sigma^k,\sigma^b)=\bigwedge^p\frac{N_R}{T_RF_{\sigma^b}^{\perp}}$. As we also know that $\mathcal{F}_p^R(\sigma^a,\sigma^b)\subseteq\bigwedge^p\frac{N_R}{T_RF_{\sigma^b}^{\perp}}$, we finally get:
    \begin{equation*}
        \forall p<k, ~~\forall p < a\leq b, ~~\forall \sigma^a \subseteq \sigma^b,~~\mathcal{F}_p(\sigma^a,\sigma^b) = \bigwedge^p\frac{N_R}{T_RF_{\sigma^b}^{\perp}} \simeq R^{\binom{\mathrm{dim}F_{\sigma^b}}{p}}.
    \end{equation*}

    We have obtained that for any $p<k$, $\mathcal{F}_p(*,\sigma^b)$ is a constant cosheaf for the $\sigma^a \subseteq \sigma^b$ with $p < a\leq b$. We have the following isomorphisms:
    \begin{equation*}
        \forall p<k,~~\forall p<a\leq b, ~~C_{b-a}(\mathcal{F}_p^R(*,\sigma^b)) = C_{b-a}(\sigma^b;R^{\binom{\mathrm{dim}F_{\sigma^b}}{p}}).
    \end{equation*}
    where $C_{*}(\sigma^b;R^{\binom{\mathrm{dim}F_{\sigma^b}}{p}})$ is the classical simplicial chains of the simplex $\sigma^b$ with coefficients in $R^{\binom{\mathrm{dim}F_{\sigma^b}}{p}}$. These isomorphisms commute with the differential and we obtain isomorphism with the simplicial homology of $\sigma^b$:
    \begin{equation*}
     \forall p<k,~~\forall p+1<a\leq b,~~H_{b-a}(\mathcal{F}_p^R(*,\sigma^b)) = H_{b-a}(\sigma^b;R^{\binom{r}{p}})=0.
    \end{equation*}
\end{proof}

\subsection{The cohomological spectral sequence: proof of Lemma \ref{5}}

Let's begin by reminding the following identity:
\begin{equation*}
    E_1^{a,b}(\delta_{2}):= \bigoplus_{\sigma^a \in \Gamma} H^{b-a}(\mathcal{F}^p_R(\sigma^a,*)).
\end{equation*}

It shows that to prove Lemma \ref{4} we just need to study the properties of the cohomology groups $H^{b-a}(\mathcal{F}^p_R(\sigma^a,*))$. In particular, as the Lemma \ref{5} enounces that $E_1^{a,b}(\delta_{2})=0$ for every $b\neq n$, we have to prove that  $H^{b-a}(\mathcal{F}^p_R(\sigma^a,*)) = 0$ for every $b\neq n$. To prove that we need an assumption of local $R$-non-singularity on the lattice polytope $P$. It is our following proposition:

\begin{proposition}
    Let $\sigma^a \in \Gamma$. Remind that $F_{\sigma^a}$ denotes the minimal face of the polytope $P$ containing $\sigma^a$. If the polytope $P$ is $R$-non-singular on its face $F_{\sigma^a}$ (see Definition \ref{locprim}), then the cohomology $H^{*}(\mathcal{F}^p_R(\sigma^a,*))$ is torsion-free and satisfies:
    \begin{equation*}
        H^{b-a}(\mathcal{F}_R^p(\sigma^a,*)) = \left \{ \begin{array}{ll}
            0 & \text{if } b\neq n \\
            \mathcal{F}_R^{p - n + \mathrm{dim} ~ F_{\sigma^a}}(\sigma^a,\sigma^a) & \text{if } b=n
        \end{array} \right . .
    \end{equation*}
    The isomorphism $H^{n-a}(\mathcal{F}_R^p(\sigma^a,*)) = \mathcal{F}_R^{p - n + \mathrm{dim} ~ F_{\sigma^a}}(\sigma^a,\sigma^a)$ is canonical up to a choice of a generator of $\bigwedge^{n-\mathrm{dim}F_{\sigma^a}}\frac{M_R}{T_RF_{\sigma^a}}$.
    \label{theocohom}
\end{proposition}

Brugallé, Lopez de Medrano and Rau have proven this proposition in \cite{BMR 2024} (in their proof of Lemma 3.1), in the case of a non-singular lattice polytope, for the integral domain $R=\mathbb{F}_2$ and for homology:

\begin{proposition}
    \cite{BMR 2024} For a non-singular lattice polytope $P$ with a triangulation $\Gamma$ we have:
    \begin{equation*}
        H_{b-a}(\mathcal{F}^{\mathbb{F}_2}_p(\sigma^a,*)) = \left \{ \begin{array}{ll}
            0 & \text{if } b\neq n \\
            \mathcal{F}^{\mathbb{F}_2}_{p - n + \mathrm{dim} ~ F_{\sigma^a}}(\sigma^a,\sigma^a) & \text{if } b=n
        \end{array} \right . .
    \end{equation*}
    \label{BMR}
\end{proposition}

As their article \cite{BMR 2024} concerns only primitive triangulations, their proof is written for a primitive triangulation $\Gamma$, but in reality the primitivity of $\Gamma$ is used nowhere in their proof, that is why we enounce here their proposition for a non-necessarily primitive triangulation $\Gamma$. Their cosheaves $\mathcal{F}_p^{\mathbb{F}_2}$ are also defined for a coarser cellular complex, indexed by the poset $\Xi=\{(F,\sigma),~ F \text{ face of } P,~\sigma\in \Gamma,~\mathrm{dim} \sigma\geq 1\}$. However, according to our discussion in the beginning of the preliminaries of this article, Jules Chenal has proven in his PhD thesis \cite{Ch 2024} that the resulting homology is the same as with the cubical subdivision, indexed by the poset $\Omega=\{(\sigma,\tau),~\sigma,\tau \in \Gamma,~\sigma \subseteq \tau,~\mathrm{dim}\sigma \geq 1\}$ that we use in this article. It is the reason why we have written this Proposition \ref{BMR} of Brugallé, Lopez de Medrano and Rau in terms of the cubical subdivision.

Some arguments of the proof of Brugallé, Lopez de Medrano and Rau of Proposition \ref{BMR} can be adapted to our context to prove our Proposition \ref{theocohom}. The hypothesis of $R$-non-singularity of the polytope on the face $F_{\sigma^a}$ can be used instead of their stronger hypothesis of non-singularity of $P$ and we obtain the following identity:
\begin{equation}
    H^{b-a}(\mathcal{F}^p_R(\sigma^a,*))= \bigoplus_{r+t = p} \mathcal{F}^r_R(\sigma^a,\sigma^a) \otimes H^{b}(\mathcal{G}^t_{R,\sigma^a})
    \label{dirsum}
\end{equation}
where $\mathcal{G}^t_{R,\sigma^a}$ is a cellular sheaf on the poset $L_{\sigma^a}:=\{F \text{ face of }P,~~ F\supseteq \sigma^a\}$ graded by the dimension of the faces of $P$, defined by $\mathcal{G}^t_{R,\sigma^a}(F):=\bigwedge^t\frac{T_RF}{T_RF_{\sigma^a}}$, whose sheaves morphisms are induced by the inclusions $T_RF\subseteq T_RG$ for any incident faces $F\subseteq G$ of $P$ containing $\sigma^a$.

Now to conclude the proof of our Proposition \ref{theocohom} we have to prove that the cohomology $H^{b}(\mathcal{G}^t_{R,\sigma^a})$ vanishes except if $b=n=t+\mathrm{dim}F_{\sigma^a}$, which is the purpose of the following lemma:

\begin{lemma}
    If the polytope $P$ is non-singular of the face $F_{\sigma^a}$, then we have for any integers $t,b$:
    \begin{equation*}
        H^{b}(\mathcal{G}^t_{R,\sigma^a}) = \left  \{ \begin{array}{ll}
             0 & \text{ if } b \neq n \text{ or } n \neq t+\mathrm{dim}F_{\sigma^a} \\
             R& \text{ if } b = n = t+\mathrm{dim}F_{\sigma^a}
        \end{array} \right. .
    \end{equation*}
    \label{homwedge}
\end{lemma}

We cannot use the same strategy as Brugallé, Lopez de Medrano and Rau in \cite{BMR 2024}, who use some properties of non-singular projective toric varieties, because here the polytope $P$ is not supposed to be non-singular but only $R$-non-singular on the face $F_{\sigma^a}$. However, fixing a basis of $\mathbb{Z}^n\otimes R$ adapted to the definition of $R$-non-singularity of the polytope $P$ on the face $F_{\sigma^a}$ (see Definition \ref{locprim}), we are able to compute this cohomology by comparison with the simplicial cohomology of certain abstract simplices. It is the purpose of the following proof:

\begin{proof}
    \underline{Proof of Lemma \ref{homwedge}:}
    Let us fix the set of vectors $\{e_1,...,e_{n-\mathrm{dim}F_{\sigma^a}}\}$ of $M_R$ adapted for the definition of $R$-non-singularity of the polytope $P$ on the face $F_{\sigma^a}$ (see Definition \ref{locprim}). Using Definition \ref{locprim} of local $R$-non-singularity and Remark \ref{bij} we have a bijection of sets:
    \begin{align*}
       \gamma:\{F  \text{ face of } P,~ F \supseteq \sigma^a\} &\simeq  \mathcal{P}(\{1,...,{n-\mathrm{dim}F_{\sigma^a}}\}) \\
        F ~~~~~~~~&\mapsto ~~~~I \text{ such that } T_RF=T_RF_{\sigma^a} \oplus \bigoplus_{i \in I} Re_{i}
    \end{align*}
    where $\mathcal{P}(\{1,...,{n-\mathrm{dim}F_{\sigma^a}}\})$ denotes the power set of $\{1,...,{n-\mathrm{dim}F_{\sigma^a}}\}$.
    
    Let us denote $\{e_1',...,e_{n-\mathrm{dim}F_{\sigma^a}}'\}$ the image of $\{e_1,...,e_{n-\mathrm{dim}F_{\sigma^a}}\}$ by the projection $M_R \twoheadrightarrow \frac{M_R}{T_RF_{\sigma^a}}$. Then for any face $F$ of $P$ containing $\sigma^a$ the family $(\bigwedge_{i\in I}e_i')_{I \subseteq\gamma(F),~\mathrm{Card}I=t}$ forms a basis of $\mathcal{G}^t_{R,\sigma^a}(F)$. It induces the following decomposition of the cochain complex $C^{b}(\mathcal{G}^t_{R,\sigma^a})$:
    \begin{equation}
        C^{b}(\mathcal{G}^t_{R,\sigma^a}) = \bigoplus_{I\subseteq \{1,...,{n-\mathrm{dim}F_{\sigma^a}}\},~\mathrm{Card}I=t}C^{b}(D_{\gamma^{-1}(I)};R)\otimes \bigwedge_{i\in I}e_i'.
        \label{erty}
    \end{equation}
    where $D_{\gamma^{-1}(I)}:=\{F \text{ face of }P \text{ containing } \gamma^{-1}(I)\}$ and $C^{b}(D_{\gamma^{-1}(I)};R)$ are the cochains with coefficients in $R$ over $D_{\gamma^{-1}(I)}$ ordered by inclusion and graded by the dimension, with a balancing signature inherited from the balancing signature of the polytope $P$ as a polyhedral complex. Using the bijection $\gamma$, $D_{\gamma^{-1}(I)}$ is isomorphic as a poset to the poset of the subsets of $\{1,...,n-\mathrm{dim}F_{\sigma^a}\}$ which contain $I$. However, if $n>\mathrm{dim}F_{\sigma^a}+t$ the reduced simplicial cochains of an abstract simplex $\Delta^{n-\mathrm{dim}F_{\sigma^a}-t-1}$ of dimension $n-\mathrm{dim}F_{\sigma^a}-t-1$ can be indexed by the poset of subsets of $\{1,...,n-\mathrm{dim}F_{\sigma^a}-t\}$, which is isomorphic the poset of the subsets of $\{1,...,n-\mathrm{dim}F_{\sigma^a}\}$ which contain $I$ (because $\mathrm{Card}I=t$). Then if $n >\mathrm{dim}F_{\sigma^a}+t$ we can compute the cohomology of the polyhedral complex $D_{\gamma^{-1}(I)}$ using the reduced simplicial cohomology of an abstract simplex $\Delta^{n-\mathrm{dim}F_{\sigma^a}-\mathrm{Card}I-1}$ of dimension $n-\mathrm{dim}F_{\sigma^a}-t-1$:
    \begin{equation*}
        \forall n> \mathrm{dim}F_{\sigma^a}+t,~~H^{b}(D_{\gamma^{-1}(I)};R) =\Tilde H^{b-\mathrm{dim}F_{\sigma^a}-t-1}(\Delta^{n-\mathrm{dim}F_{\sigma^a}-t-1};R) =0.
    \end{equation*}
    The difference between the degrees $b$ and $b-\mathrm{dim}F_{\sigma^a}-t-1$ of the two cohomologies comes from the fact that there is a shift between the gradings of the corresponding cochain complexes.

    If $n<\mathrm{dim}F_{\sigma^a}+t$ the cohomology $H^{b}(D_{\gamma^{-1}(I)};R)$ also vanishes because $D_{\gamma^{-1}(I)}$ is empty. In the case $n=\mathrm{dim}F_{\sigma^a}+t$ we have $D_{\gamma^{-1}(I)}=\{P\}$, and $P$ is of dimension $n$, then we obtain that the cohomology groups $H^{b}(D_{\gamma^{-1}(I)};R)$ vanish in every degree $b$ except in degree $b=n$ where the cohomology is $R$. Finally we obtain the general identity:
    \begin{equation*}
        H^{b}(D_{\gamma^{-1}(I)};R)=\Bigg\{\begin{array}{ll}
            R & \text{if } b=n=\mathrm{dim}F_{\sigma^a}+t \\
            0 & \text{if } b\neq n \text{ or }n\neq \mathrm{dim}F_{\sigma^a}+t
        \end{array}.
    \end{equation*}

    It concludes the proof using equation \ref{erty}.
\end{proof}

Now we wan prove Proposition \ref{theocohom} using Lemma \ref{homwedge}:

\begin{proof}
    \underline{Proof of Proposition \ref{theocohom}:} From equation \ref{dirsum} and Lemma \ref{homwedge} we obtain:
    \begin{equation*}
        H^{b-a}(\mathcal{F}_R^p(\sigma^a,*)) = \left \{ \begin{array}{ll}
            0 & \text{if } b\neq n \\
            \mathcal{F}_R^{p - n + \mathrm{dim} ~ F_{\sigma^a}}(\sigma^a,\sigma^a)\otimes H^n(\mathcal{G}^{n-\mathrm{dim}F_{\sigma^a}}_{R,\sigma^a}) & \text{if } b=n
        \end{array} \right . .
    \end{equation*}
    In the case $b=n$, we can choose a generator of $H^n(\mathcal{G}^{n-\mathrm{dim}F_{\sigma^a}}_{R,\sigma^a})=R$ (which corresponds to a choice of a generator of $\bigwedge^{n-\mathrm{dim}F_{\sigma^a}}\frac{M_R}{T_RF_{\sigma^a}}$), which concludes the proof.
\end{proof}

We can now conclude the proof of Lemma \ref{5}, which was the purpose of this subsection. As we assume that the polytope $P$ is $R$-non-singular, then it is $R$-non-singular on every one of its faces. Then we obtain:
\begin{align*}
    E_1^{a,b}(\delta_{2})&:= \bigoplus_{\sigma^a \in \Gamma} H^{b-a}(\mathcal{F}^p_R(\sigma^a,*)) = \Bigg\{ \begin{array}{ll}
         \bigoplus_{\sigma^a \in \Gamma}\mathcal{F}^{p-n+\mathrm{dim}F_{\sigma^a}}_R(\sigma^a,\sigma^a)& \text{if } b = n \\
         0 & \text{if } b \neq n 
    \end{array} \\
    &= \Bigg\{ \begin{array}{ll}
         \mathrm{coker} ~ \delta^{a,n-1}_{2}& \text{if } b = n \\
         0 & \text{if } b \neq n 
         \end{array} .
\end{align*}

\subsection{The cap-product with the fundamental class: Proof of the Lemmas \ref{lebhb}, \ref{ejnf} and \ref{8}}

Lemmas \ref{lebhb}, \ref{ejnf} and \ref{8} concern the relations between the map $\phi^q_{(p)}$ (defined in Definition \ref{defphi}) and the differentials of the chain complexes of cosheaves $\mathcal{F}_p^R$ and the cochain complex of sheaves $\mathcal{F}^p_R$. We already know from the last section the rank of $\mathrm{ker} ~ d^{1,(n-1-p)}_{1,n-q}$ if the triangulation is $R$-primitive (Proposition \ref{theohom}), and the rank of $\mathrm{coker} ~ \delta^{n-q,n-1}_{2}$ if the polytope $P$ is $R$-non-singular (Proposition \ref{theocohom}). Thus to prove the Lemma \ref{lebhb}, we want to compute the rank of $\mathrm{im} ~ \phi^q_{(p)}$ and to prove that it is included into $\mathrm{ker} ~ d^{1,(n-1-p)}_{1,n-q}$, and to prove the Lemma \ref{ejnf} we want to compute the rank of $\mathrm{ker}~ \phi^q_{(p)} \cap C^{n-q,n}(\mathcal{F}^p_R)$ and to prove that it contains $\mathrm{im} ~ \delta^{n-q,n-1}_{2,(p)}$. However, those two computations are related by the rank theorem, and we just need to compute one of the two. Let us prove the following proposition:

\begin{proposition}
    Let us fix a simplex $\sigma^a \in \Gamma$. Assume that the triangulation satisfies the hypothesis 1 for the integral domain $R$ (see Definition \ref{max}). Let us denote $V=\sum_{\sigma^1 \subseteq \sigma^a}T_R \sigma^1$. Then the map $\phi_{\sigma^{a},(p)}$ (see Definition \ref{defphi}) satisfies:
    \begin{equation*}
        \mathrm{rank} ~ \phi_{\sigma^{a},(p)} =\binom{\mathrm{dim}F_{\sigma^a}}{p-n+\mathrm{dim}F_{\sigma^a}}-\binom{\mathrm{dim}F_{\sigma^a}-\mathrm{dim}V}{p-n+\mathrm{dim}F_{\sigma^a}-\mathrm{dim}V}.
    \end{equation*}
    Moreover, if we assume that the triangulation satisfies the hypothesis 2 for the integral domain $R$ (see Definition \ref{max}), the image of $\phi_{\sigma^{a},(p)}$ is a $R$-module saturated in $C_{a-1}(\mathcal{F}_{n-1-p}^R(*,\sigma^a))$ (see Definition \ref{sat} for the saturation of a sub-module).
    \label{ranphi}
\end{proposition}

In the following we fix a simplex $\sigma^a \in \Gamma$. For any $\sigma^n \supseteq \sigma^a$ we have the surjective map $(\iota_{\sigma^a,\sigma^n}^{\sigma^n})^*: \bigwedge^pM_R \twoheadrightarrow \mathcal{F}^p_R(\sigma^a,\sigma^n)$. Then the image of $\phi_{\sigma^{a},(p)}$ is equal to the image of the following map:
\begin{align*}
    \psi:\bigoplus_{\sigma^n \supseteq \sigma^a}\bigwedge^pM_R &\rightarrow  C_{a-1}(\mathcal{F}_{n-1-p}(*,\sigma^a)) \\
    \bigoplus_{\sigma^n \supseteq \sigma^a} \beta_{\sigma^n} & \mapsto \phi_{\sigma^{a},(p)} \Bigg( \bigoplus_{\sigma^n \supseteq \sigma^a} (\iota_{\sigma^a,\sigma^n}^{\sigma^n})^*(\beta_{\sigma^n} ) \Bigg) .
\end{align*}

By definition of $\phi_{\sigma^{a},(p)}$ (Defintion \ref{defphi}) and by definition of the cap-product (Proposition \ref{defcap}), we have:
\begin{equation*}
    \psi \Bigg( \bigoplus_{\sigma^n \supseteq \sigma^a} \beta_{\sigma^n} \Bigg) = \bigoplus_{\sigma^1 \subseteq \sigma^a} \sum_{\sigma^n \supseteq \sigma^a}\pi^{\sigma^a}(\beta_{\sigma^n}.[\Omega]_{\sigma^1,\sigma^n}) = \bigoplus_{\sigma^1 \subseteq \sigma^a} \pi^{\sigma^a} \Bigg( \sum_{\sigma^n \supseteq \sigma^a}\beta_{\sigma^n}.[\Omega]_{\sigma^1,\sigma^n} \Bigg)
\end{equation*}
where the dot denotes the contraction morphism (Definition \ref{defcontr}), $[\Omega]_{\sigma^1,\sigma^n}$ denotes the $(\sigma^1,\sigma^n)$-coordinate of $[\Omega]$, and $\pi^{\sigma^a}$ denotes the projection $\bigwedge^{n-1-p}N_R\twoheadrightarrow \bigwedge^{n-1-p}\frac{N_R}{T_RF_{\sigma^a}^{\perp}}$.

Remind that the element $[\Omega]$ is defined by $[\Omega]:=\bigoplus_{\sigma^1,\sigma^n} |\sigma^1|_r1_{\sigma^1,\sigma^n} \otimes 1_R$, with the $1_{\sigma^1,\sigma^n}$ generators of the $\mathcal{F}_{n-1}^R(\sigma^1,\sigma^n)$, and satisfying $d_{n-1}([\Omega]) =0$. Moreover, in the proof of Lemma \ref{yt}, the element $1_{\sigma^1,\sigma^n}$ was defined in the following way:
\begin{equation*}
    1_{\sigma^1,\sigma^n} = \mu(\sigma^n)g(o(\sigma^1))
\end{equation*}
where $o(\sigma), ~\sigma \in \Gamma$ is a family of orientation fixed on the simplices of $\Gamma$; compatible with the balancing signature of the simplicial complex $\Gamma$ ;  where an orientation $o(M)$ is fixed on the ambient lattice $M$ and $\mu(\sigma^n)$ is defined by $\mu(\sigma^n)=1$ if $o(\sigma^n)=o(M)$, and $\mu(\sigma^n)=-1$ if $o(\sigma^n)=-o(M)$ ; $g(o(\sigma^1))$ is a generator of $\bigwedge^{n-1}T_{\mathbb{Z}}\sigma^{1\perp}$ obtained canonically from the orientation $o(\sigma^1)$ of $\sigma^1$ and the fixed orientation $o(M)$ of $M$.

We obtain:

\begin{equation*}
    \psi \Bigg( \bigoplus_{\sigma^n \supseteq \sigma^a} \beta_{\sigma^n} \Bigg) = \bigoplus_{\sigma^1 \subseteq \sigma^a} \pi^{\sigma^a} \Bigg( \Bigg[ \sum_{\sigma^n \supseteq \sigma^a}\mu(\sigma^n)\beta_{\sigma^n}\Bigg].(|\sigma^1|_rg(o(\sigma^1))\otimes1_R) \Bigg).
\end{equation*}

It shows that the map $\psi$ admits a factorisation $\psi =\overline{\psi} \circ \epsilon$ where:
\begin{align*}
    \epsilon: \bigoplus_{\sigma^n \supseteq \sigma^a} \bigwedge^p M_R & \twoheadrightarrow \bigwedge^p M_R &\text{and }~~~ \overline{\psi}: \bigwedge^p M_R & \twoheadrightarrow  C_{a-1}(\mathcal{F}_{n-1-p}(*,\sigma^a)) \\
    \bigoplus_{\sigma^n \supseteq \sigma^a} \beta_{\sigma^n} & \mapsto \sum_{\sigma^n \supseteq \sigma^a} \mu(\sigma^n)\beta_{\sigma^n} & \alpha & \mapsto \bigoplus_{\sigma^1 \subseteq \sigma^a}\pi^{\sigma^a} \Bigg( \alpha.(|\sigma^1|_rg(o(\sigma^1))\otimes1_R) \Bigg) .
\end{align*}

Since for every $\sigma^n \in \Gamma$ the term $\mu(\sigma^n)$ is invertible in $R$ (because it is $+1$ or $-1$), the map $\epsilon$ is surjective and we have $\mathrm{im} ~ \overline{\psi} = \mathrm{im} ~ \psi = \mathrm{im}~\phi_{\sigma^{a},(p)}$. Thus it is enough to study the map $\overline{\psi}$ to prove proposition \ref{ranphi}:

\begin{proof}
    Proof of Proposition \ref{ranphi}.
    
    Let us fix a basis $(f_1,...,f_{\mathrm{rank} V})$ of $V$, and let $\{f_{\mathrm{rank}V+1},...,f_{\mathrm{dim}F_{\sigma^a}}\}$ a set of vectors which complete it, to obtain a basis $(f_1,...,f_{\mathrm{dim} F_{\sigma^a}})$ of $T_R F_{\sigma^a}$. Let $\{e_1,...,e_{n-\mathrm{dim}F_{\sigma^a}}\}$ a set of vectors which complete $(f_1,...,f_{\mathrm{dim} F_{\sigma^a}})$, to obtain a basis of $M_R$. Let us denote $(f_1^*,...,f_{\mathrm{dim} F_{\sigma^a}}^*,e_1^*,...,e_{n-\mathrm{dim} F_{\sigma^a}}^*)$ the dual basis of $N_R$. For every $\sigma^1 \subseteq \sigma^a$, let us denote $u_{\sigma^1}$ a generator of $T_R\sigma^1$. We can complete it to obtain a basis $(u_{\sigma^1},v_1^{\sigma^1},...,v_{\mathrm{rank}V -1}^{\sigma^1})$ of $V$, and then the following basis of $M_R$:
    \begin{equation*}
        (u_{\sigma^1},v_1^{\sigma^1},...,v_{\mathrm{rank}V -1}^{\sigma^1},f_{\mathrm{rank}V+1},...,f_{\mathrm{dim}F_{\sigma^a}},e_1,...e_{n-\mathrm{dim}F_{\sigma^a}}).
    \end{equation*}
    The dual basis is of the form:
    \begin{equation*}
        (u_{\sigma^1}^*,(v_1^{\sigma^1})^*,...,(v_{\mathrm{rank}V -1}^{\sigma^1})^*,f_{\mathrm{rank}V+1}^*,...,f_{\mathrm{dim}F_{\sigma^a}}^*,e_1^*,...e_{n-\mathrm{rank}V}^*)
    \end{equation*}
    where the vectors $f_{\mathrm{rank}V+1}^*,...,f_{\mathrm{dim}F_{\sigma^a}}^*,e_1^*,...e_{n-\mathrm{dim}F_{\sigma^a}}^*$ are the same as the ones that appear in the dual basis $(f_1^*,...,f_{\mathrm{dim} F_{\sigma^a}}^*,e_1^*,...,e_{n-\mathrm{dim} F_{\sigma^a}}^*)$.
    
    Then $((v_1^{\sigma^1})^*, ..., (v_{\mathrm{dim}F_{\sigma^a} -1}^{\sigma^1})^*, f_{\mathrm{rank}V+1}^*, ..., f_{\mathrm{dim}F_{\sigma^a}}^*, e_1^* ,..., e_{n-\mathrm{dim}F_{\sigma^a}}^*)$ is a basis of $T_R \sigma^{1\perp}$. Then $(v_1^{\sigma^1})^*\wedge...\wedge(v_{\mathrm{dim}F_{\sigma^a}-1}^{\sigma^1})^*\wedge f_{\mathrm{rank}V+1}^*\wedge...\wedge f_{\mathrm{dim}F_{\sigma^a}} ^*\wedge e_1^*\wedge... \wedge e_{n-\mathrm{dim}F_{\sigma^a}}^*$ and $g((o(\sigma^1)) \otimes 1_R$ are both generators of $\bigwedge^{n-1}T_R\sigma^{1\perp}$, thus there exists an invertible element $\lambda_{\sigma^1}$ of $R$ such that:
    \begin{align*}
        &(v_1^{\sigma^1})^*\wedge...\wedge(v_{\mathrm{dim}F_{\sigma^a} -1}^{\sigma^1})^*\wedge f_{\mathrm{rank}V+1}^* \wedge...\wedge f_{\mathrm{dim}F_{\sigma^a}}^* \wedge e_1^*\wedge... \wedge e_{n-\mathrm{dim}F_{\sigma^a}}^* =\lambda_{\sigma^1} g(o(\sigma^1)) \otimes 1_R.
    \end{align*}

    Let $\alpha \in \bigwedge^p M_R$. We can write:
    \begin{equation*}
        \alpha = \sum_{k=0}^p\sum_{A \in \mathcal{P}_{p-k}(\{1,...,\mathrm{dim}F_{\sigma^a}\})}\sum_{B \in \mathcal{P}_{k}(\{1,...,n-\mathrm{dim} F_{\sigma^a}\})}\alpha_{A,B} \bigwedge_{i \in A}f_i \wedge\bigwedge_{j\in B} e_i
    \end{equation*}
    where $\mathcal{P}_l(\{1,...,s\})$ denotes the set of subsets of cardinal $l$ of $\{1,...,s\}$. However, for every $k\neq n- \mathrm{dim} F_{\sigma^a}$, $B \in \mathcal{P}_{k}(\{1,...,n-\mathrm{dim} F_{\sigma^a}\})$, and $A \in \mathcal{P}_{p-k}(\{1,...,\mathrm{dim}F_{\sigma^a}\})$, for every $\sigma^1 \subseteq \sigma^a$ we have:
    \begin{equation*}
        \pi^{\sigma^a}\Bigg(\Bigg[\bigwedge_{i \in A}f_i \wedge\bigwedge_{j\in B} e_i \Bigg].\Bigg[\bigwedge_{i=1}^{\mathrm{dim}F_{\sigma^a} -1}(v_i^{\sigma^1})^* \wedge\bigwedge_{i=\mathrm{rank}V+1}^{\mathrm{dim}F_{\sigma^a}}f_i^* \wedge\bigwedge_{j=1}^{n-\mathrm{dim}F_{\sigma^a}}e_j^* \Bigg] \Bigg) = 0.
    \end{equation*}

    Thus every coordinate of $\alpha$ contributes to zero to $\overline{\psi}(\alpha)$ except the ones for which $k=n-\mathrm{dim}F_{\sigma^a}$, which corresponds to $B=\{1,...,n-\mathrm{dim}F_{\sigma^a}\}$ Thus we obtain:
    \begin{equation*}
        \overline{\psi}(\alpha) = \overline{\psi} \Bigg( \sum_{A \in \mathcal{P}_{p-n+\mathrm{dim} F_{\sigma^a}}(\{1,...,\mathrm{dim}F_{\sigma^a}\})}\alpha_{A,\{1,...,n-\mathrm{dim} F_{\sigma^a}\} } \bigwedge_{i \in A}f_i \wedge\bigwedge_{j=1}^{n-\mathrm{dim} F_{\sigma^a}} e_i\Bigg).
    \end{equation*}

    Let us denote $\overline{\alpha}:=\sum_{A \in \mathcal{P}_{p-n+\mathrm{dim} F_{\sigma^a}}(\{1,...,\mathrm{rank}V\})}\alpha_{A,\{1,...,n-\mathrm{dim} F_{\sigma^a}\} } \bigwedge_{i \in A}f_i$. As we assume the triangulation to satisfy the hypothesis 1 for the integral domain $R$, we have for every $\sigma^1 \subseteq \sigma^a$ that $|\sigma^1|_r$ does not vanish in the integral domain $R$. Thus we have:
    \begin{align*}
        &\overline{\psi}(\alpha) = 0 \iff \forall \sigma^1 \subseteq \sigma^a, ~\overline{\alpha}.\Bigg[\bigwedge_{i=1}^{\mathrm{rank}V -1}(v_i^{\sigma^1})^*\wedge\bigwedge_{i=\mathrm{rank}V+1}^{\mathrm{dim}F_{\sigma^a}} f_i^*\Bigg] =0.
    \end{align*}

    Let us identify the vectors $(v_i^{\sigma^1})^*$ and the vector $u_{\sigma^1}^*$, for every $\sigma^1 \subseteq \sigma^a$, with their projection in $\frac{N_R}{V^{\perp}}=V^*$. Then for every $\sigma^1 \subseteq \sigma^a$ we have that $(u_{\sigma^1}^*,(v_1^{\sigma^1})^*,...,(v_{\mathrm{rank}V-1}^{\sigma^1})^*)$ forms a basis of $V^*$, dual to the basis $(u_{\sigma^1},v^{\sigma^1}_1,...,v^{\sigma^1}_{\mathrm{dim}F_{\sigma^a}})$ of $V$. Then the vector $\bigwedge_{i=1}^{\mathrm{rank}V -1}(v_i^{\sigma^1})^*$ is uniquely determined from the vector $u_{\sigma^1}$, up to an invertible factor of $R$. We have the following equality of rank:
    \begin{equation*}
        \mathrm{rank} ~ \sum_{\sigma^1 \subseteq \sigma^a} \bigwedge_{i=1}^{\mathrm{rank}V -1}(v_i^{\sigma^1})^* = \mathrm{rank}~ \sum_{\sigma^1 \subseteq \sigma^a} Ru_{\sigma^1} = \mathrm{rank}~ V = \mathrm{rank}~\bigwedge^{\mathrm{rank} V -1}V^*.
    \end{equation*}

    Thus we obtain that the family $( \bigwedge_{i=1}^{\mathrm{rank}V-1}(v_i^{\sigma^1})^*)_{\sigma^1 \subseteq \sigma^a}$ generates $\bigwedge^{\mathrm{rank} V -1}V^*$. Then we have:
    \begin{align*}
        &\overline{\psi}(\alpha) = 0 \\
        &\iff \forall x \in\bigwedge^{\mathrm{rank} V -1}\Bigg(\sum_{i=1}^{\mathrm{rank}V}Rf_i^*\Bigg), ~~\overline{\alpha}.\Bigg[x\wedge\bigwedge_{i=\mathrm{rank}V+1}^{\mathrm{dim}F_{\sigma^a}} f_i^*\Bigg] =0 \\
        & \iff \forall j \in \{1,...,\mathrm{rank}~V\},~~\overline{\alpha}.\Bigg[ \bigwedge_{i\in \{1,...,\mathrm{rank}V\}\backslash \{j\}} f_i^*\wedge\bigwedge_{i=\mathrm{rank}V+1}^{\mathrm{dim}F_{\sigma^a}} f_i^*\Bigg] =0 \\
        & \iff \forall A \in \mathcal{P}_{p-n+\mathrm{dim} F_{\sigma^a}}(\{1,...,\mathrm{dim}F_{\sigma^a}\}), ~~~~A \not\supseteq \{1,...,\mathrm{rank}V\} \implies\alpha_{A,\{1,...n-\mathrm{dim} F_{\sigma^a}\}}=0.
    \end{align*}

    Then we obtain:
    \begin{equation*}
        \mathrm{rank} ~\mathrm{ker} ~\overline{\psi}  = \mathrm{rank} ~ \bigwedge^p M_R - \Bigg[\binom{\mathrm{dim}F_{\sigma^a}}{p-n+\mathrm{dim}F_{\sigma^a}}-\binom{\mathrm{dim}F_{\sigma^a}-\mathrm{rank}V}{p-n+\mathrm{dim}F_{\sigma^a}-\mathrm{rank}V}\Bigg].
    \end{equation*}

    Finally we obtain:
    \begin{equation*}
        \mathrm{rank}~ \phi_{\sigma^{a},(p)} = \mathrm{rank} ~\overline{\psi} = \binom{\mathrm{dim}F_{\sigma^a}}{p-n+\mathrm{dim}F_{\sigma^a}}-\binom{\mathrm{dim}F_{\sigma^a}-\mathrm{rank}V}{p-n+\mathrm{dim}F_{\sigma^a}-\mathrm{rank}V}.
    \end{equation*}

    Now let us assume that the triangulation satisfies the hypothesis 2 for the integral domain $R$, and let us prove that the image of $\phi_{\sigma^{a},(p)}$ is saturated in $C_{a-1}(\mathcal{F}_{n-1-p}^R(*,\sigma^a))$. We just need to prove that the image of $\overline{\psi}$ is saturated. Let $y \in C_{a-1}(\mathcal{F}_{n-1-p}^R(*,\sigma^a))$ and $\mu \in R\backslash\{0\}$ such that $\overline{\psi}(\alpha) = \mu y$. For $\sigma^1 \subseteq \sigma^a$, using the properties of contraction we have:
    \begin{equation*}
        \mu y_{\sigma^1} =\pi^{\sigma^a} \Bigg( \alpha.(|\sigma^1|_rg(o(\sigma^1))\otimes1_R) \Bigg)= \pm\overline{\alpha}.\Bigg[\frac{|\sigma^1|_r}{\lambda_{\sigma^1}}\bigwedge_{i=1}^{\mathrm{rank}V -1}(v_i^{\sigma^1})^*\wedge\bigwedge_{i=\mathrm{rank}V+1}^{\mathrm{dim}F_{\sigma^a}} f_i^*\Bigg].
    \end{equation*}
    However we have already proven that the family  $( \bigwedge_{i=1}^{\mathrm{rank}V-1}(v_i^{\sigma^1})^*)_{\sigma^1 \subseteq \sigma^a}$ generates $\bigwedge^{n-1}\sum_{i=1}^{\mathrm{rank}V}Rf_i^*$. Then for any $1 \leq i\leq \mathrm{dim}V$ there exists a family of coefficients $(a^i_{\sigma^1})_{\sigma^1\subseteq \sigma^a}$ such that:
    \begin{equation*}
        \bigwedge_{j\in\{1,...,\mathrm{rank}V\}\backslash\{i\}}f_j^*=\sum_{\sigma^1\subseteq \sigma^a}a^i_{\sigma^1}\bigwedge_{i=1}^{\mathrm{rank}V-1}(v_i^{\sigma^1})^*.
    \end{equation*}
    Then for any $A \in \mathcal{P}_{p-n+\mathrm{dim} F_{\sigma^a}}(\{1,...,\mathrm{dim}F_{\sigma^a}\})$, if $A \not\supseteq\{1,...,\mathrm{rank}V\}$, there exists $i\in\{1,...,\mathrm{rank}V\}$ such that $i\not\in A$ and we have:
    \begin{equation*}
        \alpha_{A,\{1,...,n-\mathrm{dim}F_{\sigma^a}\}} \bigwedge_{j\in\{1,...\mathrm{rank}V\}\backslash(A\cup\{i\})}f_j^* = \sum_{\sigma^1\subseteq \sigma^a}\pm a^i_{\sigma^1}\frac{\lambda_{\sigma^1}}{|\sigma^1|_r} \mu y_{\sigma^1}.
    \end{equation*}
    We have used here the assumption that the triangulation satisfies the hypothesis $2$ for the integral domain $R$, which enable us to consider an inverse of $|\sigma^1|_r$ in $R$.
    
    From the last equation, $\mu$ divides $\alpha_{A,\{1,...,n-\mathrm{dim}F_{\sigma^a}\}}$ for any $A \in \mathcal{P}_{p-n+\mathrm{dim} F_{\sigma^a}}(\{1,...,\mathrm{dim}F_{\sigma^a}\})$ such that $A \not\supseteq\{1,...,\mathrm{rank}V\}$. However, we already know that if $A \supseteq\{1,...,\mathrm{rank}V\}$, then the term $\alpha_{A,\{1,...,n-\mathrm{dim}F_{\sigma^a}\}}$ contributes to zero in $\phi_{\sigma^{a},(p)}(\alpha)$. Then, for any $A \in \mathcal{P}_{p-n+\mathrm{dim} F_{\sigma^a}}(\{1,...,\mathrm{dim}F_{\sigma^a}\})$ such that $A \not\supseteq\{1,...,\mathrm{rank}V\}$ we can define $\beta_A$ the element of $R$ such that $\mu\beta_A = \alpha_{A,\{1,...,\mathrm{dim}F_{\sigma^a}\}}$ and we obtain:
    \begin{equation*}
        \phi_{\sigma^{a},(p)}\Bigg( \sum_{A \in \mathcal{P}_{p-n+\mathrm{dim} F_{\sigma^a}}(\{1,...,\mathrm{dim}F_{\sigma^a}\}), ~A\not\supseteq\{1,...\mathrm{rank}V\}}\beta_A\bigwedge_{i\in A}f_i^*\wedge \bigwedge_{j=1}^{n-\mathrm{dim}F_{\sigma^a}}e_j^* \Bigg) = y.
    \end{equation*}
    It proves that the image of $\phi_{\sigma^{a},(p)}$ is saturated in $C_{a-1}(\mathcal{F}_{n-1-p}^R(*,\sigma^a))$.
\end{proof}

\begin{lemma}
    Without hypothesis on the triangulation, we have:
    \begin{equation*}
        d^{*,\sigma^a,(n-1-p)}_{a-1} \circ \phi_{\sigma^a,(p)} = 0.
    \end{equation*}
\end{lemma}

\begin{proof}
    It is enough to prove that $d^{*,\sigma^a,(n-1-p)}_{a-1} \circ \psi =0$, with the map $\psi$ define previously. We have:
    \begin{align*}
        d^{*,\sigma^a,(n-1-p)}_{a-1} \circ \psi\Bigg( \bigoplus_{\sigma^n \supseteq \sigma^a}\beta_{\sigma^n}\Bigg) &= \bigoplus_{\sigma^2 \subseteq \sigma^a} (-1)^{a-1} \sum_{\sigma^1 \subset \sigma^2} \sum_{\sigma^n \supseteq \sigma^a}\rho(\sigma^1,\sigma^2)\pi^{\sigma^a}_{\sigma^n}\Bigg(\beta_{\sigma^n}.[\Omega]_{\sigma^1,\sigma^n} \Bigg) \\
        & =\bigoplus_{\sigma^2 \subseteq \sigma^a} \pi^{\sigma^a} \Bigg( \sum_{\sigma^n \supseteq \sigma^a} \Bigg[ \beta_{\sigma^n} . \Bigg( d^{1,(n-1)}_{n-1}([\Omega]) \Bigg)_{\sigma^2,\sigma^n} \Bigg]\Bigg).
    \end{align*}
    However, we know that without any hypothesis (Proposition \ref{yt}) we have:
    \begin{equation*}
        d^{1,(n-1)}_{n-1}([\Omega]) = 0.
    \end{equation*}
    It concludes the proof.
\end{proof}

\begin{lemma}
    Without hypothetis on the triangulation we have:
    \begin{equation*}
        \phi_{\sigma^a,(p)} \circ \delta_{\sigma^a,*,(p)}^{n-a-1} = 0.
    \end{equation*}
\end{lemma}

\begin{proof}
    Let $\alpha \in C^{n-a-1}(\mathcal{F}_p(\sigma^a,*))$ We have:
    \begin{equation*}
        \delta_{\sigma^a,*,(p)}^{n-a-1}(\alpha)=\bigoplus_{\sigma^a \subseteq \sigma^n}\sum_{\sigma^a \subseteq \sigma^{n-1} \subset \sigma^{n}} \rho(\sigma^{n-1},\sigma^{n}) (\pi_{\sigma^{a},\sigma^{n}}^{\sigma^{a},\sigma^{n-1}})^* (\alpha_{\sigma^a,\sigma^{n-1}}).
    \end{equation*}
    For every ${\sigma^a \subseteq \sigma^{n-1}\subseteq \sigma^n}$, we know that the map $(\iota_{\sigma^a,\sigma^{n-1}}^{\sigma^{n-1}})^*: \bigwedge ^pT_RF_{\sigma^{n-1}} \twoheadrightarrow \mathcal{F}^p_R(\sigma^a,\sigma^{n-1})$ is surjective. Thus there exists $\beta_{\sigma^{n-1}} \in \bigwedge ^pT_RF_{\sigma^{n-1}} $ such that $(\iota_{\sigma^a,\sigma^{n-1}}^{\sigma^{n-1}})^*(\beta_{\sigma^{n-1}})=\alpha_{\sigma^a,\sigma^{n-1}}$. Moreover, as $\pi^{\sigma^{n-1}}_{\sigma^n}$ is just the inclusion $\bigwedge^pT_RF_{\sigma^{n-1}} \hookrightarrow \bigwedge^p M_R$, we have:
    \begin{align*}
        (\pi_{\sigma^{a},\sigma^{n}}^{\sigma^{a},\sigma^{n-1}})^*(\alpha_{\sigma^a,\sigma^n}) &= (\pi_{\sigma^{a},\sigma^{n}}^{\sigma^{a},\sigma^{n-1}})^*\circ(\iota_{\sigma^a,\sigma^{n-1}}^{\sigma^{n-1}})^*(\beta_{\sigma^n}) = (\iota_{\sigma^a,\sigma^{n}}^{\sigma^{n}})^*\circ(\pi^{\sigma^{n-1}}_{\sigma^n})^*(\beta_{\sigma^n}) = (\iota_{\sigma^a,\sigma^{n}}^{\sigma^{n}})^*(\beta_{\sigma^n}).
    \end{align*}

    Then we obtain:
    \begin{align*}
       \phi_{\sigma^a,(p)} \circ \delta_{\sigma^a,*,(p)}^{n-a-1}(\alpha) &= \bigoplus_{\sigma^1 \subseteq \sigma^a} \sum_{\sigma^n \supseteq \sigma^a}\pi^{\sigma^a}\Bigg( \Bigg[ \sum_{\sigma^a \subseteq \sigma^{n-1} \subset \sigma^{n}} \rho(\sigma^{n-1},\sigma^{n}) \beta_{\sigma^{n-1}} \Bigg].[\Omega]_{\sigma^1,\sigma^n} \Bigg) \\
       & =\bigoplus_{\sigma^1 \subseteq \sigma^a} \sum_{\sigma^a \subseteq \sigma^{n-1} } \pi^{\sigma^a}_{\sigma^{n-1}} \circ \pi^{\sigma^{n-1}} \Bigg( \beta_{\sigma^{n-1}}.\Bigg[ \sum_{\sigma^n\supset \sigma^{n-1}} \rho(\sigma^{n-1},\sigma^n)[\Omega]_{\sigma^1,\sigma^n} \Bigg] \Bigg) \\
       &= \bigoplus_{\sigma^1 \subseteq \sigma^a} \sum_{\sigma^a \subseteq \sigma^{n-1}} \pi^{\sigma^a}_{\sigma^{n-1}} \Bigg( \beta_{\sigma^{n-1}}.\Bigg[ \sum_{\sigma^n\supset \sigma^{n-1}} \rho(\sigma^{n-1},\sigma^n)\pi^{\sigma^{n-1}}([\Omega]_{\sigma^1,\sigma^n}) \Bigg] \Bigg) \\
       & = \bigoplus_{\sigma^1 \subseteq \sigma^a} \sum_{\sigma^a \subseteq \sigma^{n-1}} \pi^{\sigma^a}_{\sigma^{n-1}} \Bigg( \beta_{\sigma^{n-1}}.\Bigg[ d^{2,(n-1)}_{n-1}([\Omega]) \Bigg]_{\sigma^1,\sigma^{n-1}} \Bigg).
    \end{align*}
    However, we know that without any hypothesis (Lemma \ref{yt}) we have:
    \begin{equation*}
        d^{2,(n-1)}_{n-1}([\Omega]) = 0.
    \end{equation*}
    It concludes the proof.
\end{proof}

\begin{lemma}
    If the polytope $P$ is $R$-non-singular on the face $F_{\sigma^a}$, and if the triangulation satisfies the hypothesis 2 for the integral domain $R$, then we have:
    \begin{equation*}
        \mathrm{ker} ~ \phi_{\sigma^a,(p)} = \mathrm{im}~ \delta^{n-1-a}_{\sigma^a,*,(p)}.
    \end{equation*}
\end{lemma}

\begin{proof}
    Firstly, let us remark that the cohomology group $H^{n-a}(\mathcal{F}^p_R(\sigma^a,*))\simeq \mathcal{F}_R^p(\sigma^a,\sigma^a)$ is torsion-free. Thus the image of $\delta^{n-a-1}_{\sigma^a,*,(p)}$ is saturated in $C^{n-a}(\mathcal{F}^p_R(\sigma^a,*))$ (see Definition \ref{sat} for the saturation of a sub-module). Let us denote $V:=\sum_{\sigma^1 \subseteq \sigma^a}T_R \sigma^1$. Then, using the saturation of the image of $\delta^{n-a-1}_{\sigma^a,*,(p)}$ we have:
    \begin{align*}
        \mathrm{rank} ~ \delta^{n-a-1}_{\sigma^a,*,(p)} &= \mathrm{rank}~ C^{n-a}(\mathcal{F}^p_R(\sigma^a,*))- \mathrm{rank} ~ \mathcal{F}_R^p(\sigma^a,\sigma^a) \\&= \mathrm{rank}~ C^{n-a}(\mathcal{F}^p_R(\sigma^a,*))-\Bigg[ \binom{\mathrm{dim}F_{\sigma^a}}{p-n+\mathrm{dim}F_{\sigma^a}} -\binom{\mathrm{dim}F_{\sigma^a}-\mathrm{rank}V}{p-n+\mathrm{dim}F_{\sigma^a}-\mathrm{rank}V} \Bigg]\\
        &= \mathrm{rank} ~ \mathrm{ker}~ \phi_{\sigma^a,(p)}.
    \end{align*}
    Now we conclude the proof using again that the image of $\delta^{n-a-1}_{\sigma^a,*,(p)}$ is saturated in $C^{n-a}(\mathcal{F}^p_R(\sigma^a,*))$.
\end{proof}

We are now ready to prove the Lemma \ref{ejnf}.

\begin{proof}
    Proof of the Lemma \ref{ejnf}.

    If the polytope $P$ is $R$-non-singular, then it is $R$-non-singular on every one of its faces. Thus we have:
    \begin{align*}
        \mathrm{ker} ~ \phi^q_{(p)} ~ \cap C^{n-q,n}(\mathcal{F}^p_R) &=\bigoplus_{\sigma^{n-q}\in\Gamma}\mathrm{ker}~\phi_{\sigma^{n-q},(p)} = \bigoplus_{\sigma^{n-q}\in\Gamma}~ \mathrm{im} ~\delta^{q-1}_{\sigma^{n-q},*,(p)} = \mathrm{im}~ \delta^{n-q,n-1}_{2,(p)}.
    \end{align*}
\end{proof}

\begin{lemma}
    If the simplex $\sigma^a$ is $R$-primitive and if the triangulation satisfies the hypothesis 2 for the integral domain $R$, then we have:
    \begin{equation*}
        \mathrm{ker}~ d^{*,\sigma^a,(n-1-p)}_{a-1} = \mathrm{im}~\phi_{\sigma^a,(p)}.
    \end{equation*}
\end{lemma}

\begin{proof}
    From Proposition \ref{theohom} we have:
    \begin{align*}
        \mathrm{rank}~\mathrm{ker}~ d^{*,\sigma^a,(n-1-p)}_{a-1} = \binom{\mathrm{dim}F_{\sigma^a}}{n-p}-\binom{\mathrm{dim}F_{\sigma^a}-a}{n-p}.
    \end{align*}

    However, as we assume the simplex $\sigma^a$ to be $R$-primitive, we have:
    \begin{equation*}
        \mathrm{rank}~\sum_{\sigma^1 \subseteq \sigma^a}T_R\sigma^1 = \mathrm{rank}~T_R\sigma^a = a.
    \end{equation*}

    Thus we obtain by Proposition \ref{ranphi}:
    \begin{equation*}
        \mathrm{rank}~\mathrm{ker}~ d^{*,\sigma^a,(n-1-p)}_{a-1} =\mathrm{rank}~\phi_{\sigma^a,(p)}.
    \end{equation*}
    Now, as the Proposition \ref{ranphi} also announces that the image of $\phi_{\sigma^a,(p)}$ is saturated in $C_{a-1}(\mathcal{F}_p^R(*,\sigma^a))$, it concludes the proof (see Definition \ref{sat} for the saturation of a sub-module).
\end{proof}

We are now ready to prove the Lemma \ref{lebhb}.

\begin{proof}
    Proof of the Lemma \ref{lebhb}.

    If the triangulation is $(k,R)$-primitive, then each of its $n-q$-simplices for $q\geq n-k$ is $R$-primitive. Moreover the hypothesis $2$ is also a consequence of the $(k,R)$-primitivity (see Proposition \ref{jt}). Then we can apply the lemma above for any $k$-simplex, and we get for $q \geq n-k$:
    \begin{align*}
        \mathrm{im} ~ \phi^q_{(p)} &=\bigoplus_{\sigma^{n-q}\in\Gamma}\mathrm{im}~\phi_{\sigma^{n-q},(p)} = \bigoplus_{\sigma^{n-q}\in\Gamma}~ \mathrm{ker} ~d^{*,\sigma^{n-q},(n-1-p)}_{n-q-1} = \mathrm{ker}~ d^{1,(n-1-p)}_{1,n-q}.
    \end{align*}
\end{proof}

Now let us prove the Lemma \ref{8}:

\begin{proof}
    Proof of the Lemma \ref{8}.

    We have previously defined a map $\psi:\bigoplus_{\sigma^n\supseteq \sigma^a}\bigwedge^p M_R \rightarrow C_{a-1}(\mathcal{F}_{n-1-p}(*,\sigma^a))$ which depends on a fixed simplex $\sigma^a$. Let us then denote it $\psi_{\sigma^a}$. Then we can define:
    \begin{equation*}
        \psi^q:=\bigoplus_{\sigma^{n-q} \in \Gamma}\psi_{\sigma^{n-q}}.
    \end{equation*}

    This map satisfies $\psi^q = \phi^q_{(p)} \circ \gamma^q$, where $\gamma^q:\bigoplus_{(\sigma^{n-q},\sigma^n)\in\Omega}\bigwedge^p M_R \twoheadrightarrow C^{n-q,n}(\mathcal{F}^p_R)$ is the canonical projection.

    Let $\alpha \in C^{n-q,n}(\mathcal{F}_R^p)$. We have:
    \begin{align*}
        \psi(\alpha) = \bigoplus_{\sigma^1 \subseteq \sigma^{n-q}}\sum_{\sigma^n \supseteq \sigma^{n-q}} \pi^{\sigma^{n-q}}(\alpha_{\sigma^a,\sigma^n}.[\Omega]_{\sigma^1,\sigma^n}).
    \end{align*}
    Then we have:
    \begin{align*}
        &d^{2,(n-1-p)}_{1,n-q} \circ \psi^q(\alpha) \\&= \bigoplus_{(\sigma^1,\sigma^{n-q-1})\in\Omega}\sum_{\sigma^{n-q-1}\subset \sigma^{n-q}\subseteq \sigma^n}\rho(\sigma^{n-q-1},\sigma^{n-q})\pi_{\sigma^{n-q}}^{\sigma^{n-q-1}}\circ\pi^{\sigma^{n-q}}(\alpha_{\sigma^{n-q},\sigma^n}.[\Omega]_{\sigma^1,\sigma^n})
        \\& = \bigoplus_{(\sigma^1,\sigma^{n-q-1})\in\Omega}\sum_{\sigma^{n-q-1}\subset \sigma^{n}}\pi^{\sigma^{n-q-1}}\Bigg( \Bigg[
        \sum_{\sigma^{n-q-1}\subset \sigma^{n-q}\subseteq \sigma^n}\rho(\sigma^{n-q-1},\sigma^{n-q})\alpha_{\sigma^{n-q},\sigma^n}\Bigg].[\Omega]_{\sigma^1,\sigma^n}\Bigg).
    \end{align*}
    However, we have also:
    \begin{align*}
        &\delta^{n-q,n}_{1,(p)} \circ \gamma^q(\alpha) \\&= \bigoplus_{(\sigma^{n-q-1},\sigma^n)\in\Omega}(-1)^{q+1}\sum_{\sigma^{n-q-1} \subset \sigma^{n-q}} \rho(\sigma^{n-q-1},\sigma^{n-q})(\iota^{\sigma^{n-q},\sigma^n}_{\sigma^{n-q-1},\sigma^n})^*\circ (\iota^{\sigma^n}_{\sigma^{n-q},\sigma^n})^*(\alpha_{\sigma^{n-q},\sigma^n}) \\
        &= \bigoplus_{(\sigma^{n-q-1},\sigma^n)\in\Omega} (\iota^{\sigma^n}_{\sigma^{n-q-1},\sigma^n})^* \Bigg( (-1)^{q+1}\sum_{\sigma^{n-q-1} \subset \sigma^{n-q}}\rho(\sigma^{n-q-1},\sigma^{n-q})\alpha_{\sigma^{n-q},\sigma^n}\Bigg) \\
        &=\gamma^{q+1} \Bigg( \bigoplus_{(\sigma^{n-q-1},\sigma^n)\in\Omega} (-1)^{q+1}\sum_{\sigma^{n-q-1} \subset \sigma^{n-q}}\rho(\sigma^{n-q-1},\sigma^{n-q})\alpha_{\sigma^{n-q},\sigma^n} \Bigg).
    \end{align*}
    Thus we have:
    \begin{align*}
        &\phi^{q+1}_{(p)}\circ \delta^{n-q,n}_{1,(p)} \circ \gamma^q(\alpha) \\
        &=\phi^{q+1}_{(p)}\circ \gamma^{q+1} \Bigg( \bigoplus_{(\sigma^{n-q-1},\sigma^n)\in\Omega} (-1)^{q+1}\sum_{\sigma^{n-q-1} \subset \sigma^{n-q}}\rho(\sigma^{n-q-1},\sigma^{n-q})\alpha_{\sigma^{n-q},\sigma^n}\Bigg) \\
        & = \psi^{q+1} \Bigg( \bigoplus_{(\sigma^{n-q-1},\sigma^n)\in\Omega} (-1)^{q+1}\sum_{\sigma^{n-q-1} \subset \sigma^{n-q}}\rho(\sigma^{n-q-1},\sigma^{n-q})\alpha_{\sigma^{n-q},\sigma^n}\Bigg) \\
        &= \bigoplus_{(\sigma^1,\sigma^{n-q-1})\in\Omega}\sum_{\sigma^{n-q-1}\subset \sigma^{n}}\pi^{\sigma^{n-q-1}}\Bigg( \Bigg[(-1)^{q+1}\sum_{\sigma^{n-q-1}\subset \sigma^{n-q}\subseteq \sigma^n} \rho(\sigma^{n-q-1},\sigma^{n-q})\alpha_{\sigma^{n-q},\sigma^n}\Bigg].[\Omega]_{\sigma^1,\sigma^n}\Bigg) \\
        &=(-1)^{q+1}d^{2,(n-1-p)}_{1,n-q} \circ \phi^q_{(p)} \circ \gamma^q(\alpha).
    \end{align*}
    By surjectivity of the map $\gamma^q$ we obtain:
    \begin{equation*}
        \phi^{q+1}_{(p)}\circ \delta^{n-q,n}_{1,(p)} = (-1)^{q+1}d^{2,(n-1-p)}_{1,n-q} \circ \phi^q_{(p)}.
    \end{equation*}
    However we also know that $\mathrm{im} ~ \phi^q_{(p)} \subseteq C_{1,n-q}(\mathcal{F}_{n-1-p}^R(*,\sigma^{n-q}))$, thus $d^{2,(n-1-p)}_{a+1+q-n,a} \circ \phi^q_{(p)} = 0$ for every $a\neq n-q$. Moreover $\phi^{q+1}_{(p)}$ vanishes on every $C^{a,a+q+1}(\mathcal{F}^p_R)$ for $a \neq n-q-1$, and $\mathrm{im}~\delta^{a+1,a+1+q}_{1,(p)} \subseteq C^{a,a+q+1}(\mathcal{F}^p_R)$ thus we have $\phi^{q+1}_{(p)}\circ\delta^{a+1,a+1+q}_{1,(p)} =0$ for every $a \neq n-q-1$. Therefore the previous equations was enough to conclude the proof:
    \begin{equation*}
        \phi^{q+1}_{(p)} \circ \delta_{1,(p)}^q = (-1)^{q+1} d^{2,(n-1-p)}_{n-q-1} \circ \phi^q_{(p)}.
    \end{equation*}
\end{proof}


\section{Bibliography}

\begin{thebibliography}{99999999999}

\bibitem[Ak 2023]{Ak 2023} Edvard Aksnes, \emph{\href{https://arxiv.org/abs/2112.03680}{Tropical Poincaré Duality spaces}}, Advances in Geometry, 23(3):345-370, August 2023.

\bibitem[AP 2023]{AP 2023} Omid Amini, Matthieu Piquerez \emph{\href{http://omid.amini.perso.math.cnrs.fr/Publications/hodge-theory-tropical-fans.pdf}{Hodge Theory for Tropical Fans}}, Bulletin de la société mathématique de France, 2023.

\bibitem[ARS 2021]{ARS 2021} Charles Arnal, Arthur Renaudineau, Kris Shaw, \emph{\href{https://arxiv.org/abs/1907.06420}{Lefschetz section theorems for}} \emph{\href{https://arxiv.org/abs/1907.06420}{tropical hypersurfaces}}, Annales Henri Lebesgue, 4:1347-1387, 2021.

\bibitem[BMR 2024]{BMR 2024} Erwan Brugallé, Lucia Lopez de Medrano, Johannes Rau, \emph{\href{https://arxiv.org/pdf/1604.01838}{Combinatorial Patchworking: Back from Tropical Geometry}}, American Mathematical Society, Volume 377, Number 10, Pages 6793–6826, October 2024.

\bibitem[Ch 2024]{Ch 2024} Jules Chenal, \emph{\href{https://theses.hal.science/tel-04791180v1}{Homologie des T-Hypersurfaces}}, PhD Thesis, 2024.

\bibitem[Ch 2025]{Ch 2025} Jules Chenal, \emph{\href{https://arxiv.org/pdf/2308.08548}{A Poincaré-Lefschetz Theorem for cellular Cosheaves and an} \href{https://arxiv.org/pdf/2308.08548}{application to the tropical homology of Orbifold Varieties}}, Annales Henri Lebesgues, 8:277-327, 2025.

\bibitem[De 2026]{De 2026} Samuel Dentan, \emph{Bounding the Betti numbers of real algebraic hypersurfaces obtained by almost primitive patchworking}, to appear in 2026.

\bibitem[IKMZ 2019]{IKMZ 2018} Ilia Itenberg, Ludmil Katzarkov, Grigory Mikhalkin and Ilia Zharkov, \emph{\href{https://arxiv.org/pdf/1604.01838.pdf}{Tropical Homology}}, Mathematische Annalen, 374(1-2):963–1006, 2019.

\bibitem[JRS 2018]{JRS 2018} Philipp Jell, Johannes Rau, and Kris Shaw, \emph{\href{https://arxiv.org/abs/1711.07900.pdf}{Lefschetz (1,1) theorem} \href{https://arxiv.org/abs/1711.07900.pdf}{in tropical geometry}}, Épijournal de Géométrie Algébrique, 2(11), 2018.

\bibitem[RS 2023]{RS 2023} Arthur Renaudineau, Kris Shaw, \emph{\href{https://arxiv.org/pdf/1805.02030}{Bounding the Betti numbers of real hypersurfaces}} \emph{\href{https://arxiv.org/pdf/1805.02030}{near the tropical limit}}, Annales scientifiques de l'Ecole Normale Supérieure, 4ème série, tome 56, fascicule 3, p 945-980, 2023.

\bibitem[Ri 1993]{Ri 1993} Jean-Jacques Risler, \emph{\href{https://www.numdam.org/article/SB_1992-1993__35__69_0.pdf}{Construction d'hypersurfaces réelles [d'après Viro]}}, Astérisque tome 216, Séminaire Bourbaki n°763, p 69-86, 1993.

\bibitem[Vi 1984]{Vi 1984} Oleg Viro, \emph{\href{https://link.springer.com/chapter/10.1007/BFb0099934}{Gluing of plane real algebraic curves and constructions of curves of} \href{https://link.springer.com/chapter/10.1007/BFb0099934}{degrees 6 and 7}}, Topology p187-200, volume 1060, 1984.

\bibitem[Vi 2006]{Vi 2006} Oleg Viro, \emph{\href{https://arxiv.org/abs/math/0611382}{Patchworking real algebraic varieties}}, arxiv.org, 2006.

\end{thebibliography}
 \end{document}